\documentclass{article}

  \usepackage[preprint]{neurips_2026}

\usepackage[utf8]{inputenc} 
\usepackage[T1]{fontenc}    
\usepackage{hyperref}       
\usepackage{url}            
\usepackage{booktabs}       
\usepackage{amsfonts}       
\usepackage{nicefrac}       
\usepackage{microtype}      
\usepackage{xcolor}         

\usepackage{caption}
\usepackage{wrapfig}
\usepackage{algorithm}
\usepackage{algpseudocode}

\usepackage{enumitem}
\usepackage{amsmath}
\usepackage{amssymb}
\usepackage{amsthm}
\usepackage{mathtools}
\usepackage{bm}
\usepackage{graphicx}
\usepackage{natbib}
\usepackage{tikz}
\usetikzlibrary{arrows.meta, positioning, calc, decorations.pathreplacing}

\hypersetup{
  colorlinks=true,
  citecolor=red,
  linkcolor=blue,
  urlcolor=blue
}

\theoremstyle{plain}
\newtheorem{theorem}{Theorem}

\theoremstyle{plain}
\newtheorem{proposition}{Proposition}

\theoremstyle{plain}
\newtheorem{lemma}{Lemma}

\theoremstyle{plain}
\newtheorem{corollary}{Corollary}

\theoremstyle{plain}
\newtheorem{assumption}{Assumption}

\theoremstyle{plain}
\newtheorem{definition}{Definition}

\theoremstyle{plain}
\newtheorem{remark}{Remark}

\theoremstyle{plain}
\newtheorem{example}{Example}

\usepackage{titletoc} 

\title{Equilibrium and Pricing in Consumer Networks with Nonlinear Utilities: An Online Shape-Constrained Learning Approach}

\author{%
    Daniele Bracale\thanks{Corresponding author.} \\
  Department of Statistics\\
  University of Michigan\\
  Ann Arbor, MI \\
  \texttt{dbracale@umich.edu} \\
  \And
  George Michailidis\\
  Department of Statistics\\
  University of California \\
  Los Angeles, CA \\
  \texttt{gmichail@g.ucla.edu} \\
}

\begin{document}

\maketitle

\begin{abstract}
We study optimal monopoly pricing over consumer networks governed by general nonlinear utilities. In our framework, a consumer's utility is jointly determined by an individualized price and the consumption choices of their peers, propagated through a directed and signed social graph. This formulation encapsulates a broad class of utility functions; it strictly generalizes the traditional linear-quadratic framework to include logit-type discrete choice, isoelastic, and Stone-Geary utilities under a single theoretical umbrella. We first establish the existence and uniqueness of the consumer-side equilibrium under general contraction and variational conditions, explicitly accommodating asymmetric and signed network externalities. Leveraging this equilibrium characterization, we analyze targeted price discrimination within community-structured and influencer-driven markets. To this end, we introduce a generalized measure of network influence that extends classical Katz–Bonacich centrality beyond the Euclidean domain. Finally, addressing the challenge of unknown consumer utility functions, we develop a shape-constrained, tuning-parameter-free learning approach utilizing isotonic regression, for which we establish strict no-regret convergence guarantees. Supported by extensive simulations, our results seamlessly integrate equilibrium analysis and nonparametric learning into a cohesive monopoly pricing framework.
\end{abstract}



\section{Introduction and Related Work}

Consumers rarely make purchasing decisions in isolation. In modern markets, individual demand is shaped not only by prices and intrinsic preferences, but also by the consumption choices of peers. Social influence, peer adoption, congestion, exclusivity, and word-of-mouth effects create dependencies among consumers. Consequently, market demand is more accurately modeled as a joint equilibrium system rather than as a collection of independent, reduced-form demand curves.

This interconnectedness is central to monopoly pricing over networks. A seller with price-discrimination capabilities can maximize revenue by offering targeted discounts to influential consumers who drive substantial downstream adoption, while charging premiums to those with limited influence. While existing network-pricing models capture this mechanism, they often rely on restrictive structural assumptions. Specifically, they often dictate that utility functions must be linear-quadratic, and that the underlying network of influence is both symmetric and nonnegative.

Both restrictions are severely limiting in modern applications. Consumer responses to price changes are often nonlinear --- as seen in isoelastic, Stone-Geary, discrete-choice, and logit-type models. Furthermore, real-world networks are typically directed and exhibit both positive and negative externalities. A consumer may exert influence without being influenced in return, and peer adoption might enhance utility through peer effects or actively degrade it through congestion, scarcity, or snob effects.

In this paper, we study monopoly pricing over consumer networks governed by general nonlinear utilities. The seller posts discriminatory prices, and consumers determine their optimal consumption levels via a network game. While the network topology captures cross-consumer dependencies, the nonlinear utility functions capture individual-level demand curvature and preference heterogeneity. 

We develop a comprehensive equilibrium and pricing theory for this \textit{general nonlinear} model, explicitly accommodating \textit{directed (i.e., non-necessarily symmetric) and signed network topologies}. We establish conditions for the \textit{existence and uniqueness of the consumer equilibrium}, analyze how nonlinear demand responses shape optimal discriminatory prices, and map the resulting pricing rules to a generalized measure of network influence. Further, for practical settings where marginal utility functions are unknown, we introduce a shape-constrained \textit{learning approach} to efficiently estimate them. Overall, this paper unifies network game theory, nonlinear monopoly pricing, and nonparametric learning, demonstrating how classical centrality-based insights fundamentally extend beyond the restrictive linear-quadratic benchmark.

\textbf{Prior Work on Pricing in Networks.}
A key benchmark for this work is \cite{candogan2012optimal}, who study a monopolist pricing over a social network with \emph{positive} local externalities. Their analysis is based on a linear-quadratic utility specification to derive explicit discriminatory prices tied to Bonacich centrality. In contrast, our paper moves beyond the linear-quadratic setting by introducing a general class of convex utility functions $h_i$, allowing the induced marginal response map $\phi_i=h_i'$ to be nonlinear. This unified framework natively captures both standard linear-quadratic models and nonlinear specifications like logit-type and isoelastic demand. Further, we do not restrict attention to symmetric or nonnegative networks; we allow the interaction matrix $\mathbf G$ to be directed and signed, and our equilibrium theory is built around general contraction and variational conditions rather than a restricted single centrality-based closed form.

Subsequent works extend this network-pricing perspective but remain tied to specialized parametric environments. For instance, \cite{bloch2013pricing} characterizes centrality-based pricing under specific consumption and price externalities, \cite{chen2018competitive} explores competitive pricing for substitute goods, and \cite{du2016optimal} analyzes network effects strictly within a multinomial-logit demand system. Rather than deriving explicit formulas for another specialized model, we provide a unified equilibrium and pricing framework that accommodates a vastly broader class of utility curvatures and network topologies while preserving interpretable, centrality-driven pricing rules for consumers.

More recent literature has expanded the original \cite{candogan2012optimal} line of inquiry in other directions. For instance, \cite{li2025pricing} investigates monopoly pricing under \emph{negative} local externalities, emphasizing the role of the complement graph within a strict linear-quadratic setting. These negative network effects — often termed \emph{snob effects} \citep{navon1995product} — occur when a consumer's utility degrades as peer consumption increases, naturally arising in environments characterized by congestion or crowding \citep{kohlberg1983equilibrium}. In contrast, our unified framework simultaneously accommodates both positive and negative externalities via a signed interaction matrix $\mathbf G$ and fundamentally generalizes the demand structure to allow nonlinear marginal consumer responses $\phi_i$.


Finally, unlike the aforementioned literature, we address the critical challenge of model estimation. By recovering the unknown marginal response functions via isotonic regression, we provide a robust pipeline for computing revenue-maximizing prices, introducing a data-driven learning layer largely absent from classical network-pricing theory.

\textbf{Prior Work on Demand Learning.}
Our work also fundamentally differs from recent literature on \emph{online learning in sequential price competition} \cite{li2024lego,bracale2025revenue,bracale2025online,goyal2023learning}. These papers analyze oligopolistic environments where competing sellers learn macroscopic, \emph{seller-level demand functions} governed directly by the market price vector, typically assuming a single-index model specification. In such settings, the learning objective centers on inferring a reduced-form demand curve to optimize strategic updates in a multi-seller pricing game. 

In contrast, our work focuses on a \emph{single monopolist} pricing over a \emph{network of consumers}. The primitive objects are not seller-level demand curves, but consumer utilities $u_i(x_i,\mathbf x_{-i};p_i)$ as defined in \eqref{eq:general_utility_unified}. For any posted price vector $\mathbf p$, consumers strategically interact over the network to determine a joint equilibrium consumption profile $\mathbf x^\star(\mathbf p)$. Consequently, rather than learning reduced-form market demands, we estimate the underlying microeconomic primitives, allowing us to explicitly model how personalized pricing shapes the structural equilibrium of the network.



Consequently, our approach necessitates a fundamentally different notion of demand. While the aforementioned competitive pricing literature models aggregate demand as a direct scalar function of market prices, our framework treats demand as an \emph{endogenous equilibrium outcome}. Each consumer's demand $x_i^\star(\mathbf p)$ is driven not only by direct price effects but also indirectly through the structural network responses of their peers.

Finally, this work complements the broader literature on monopoly dynamic pricing without networks \citep{fan2024policy,javanmard2019dynamic,bracale2025dynamic,tullii2024improved}. Although these studies explore online learning under flexible demand models, they assume consumers act in \textit{isolation}. By endogenizing demand as the outcome of a strategic network game, we seamlessly bridge advanced nonlinear pricing theory with structural, micro-level consumer interactions.

The \textbf{key contributions} of the work are summarized as follows:
\begin{enumerate}[leftmargin=1.5em,noitemsep,topsep=0pt]
\item \textbf{A unified equilibrium framework with general utilities and flexible networks.} We establish a comprehensive monopoly pricing framework that strictly generalizes classical benchmarks along two key dimensions. First, we allow consumers to have heterogeneous utility specifications of the form in \eqref{eq:general_utility_unified}, where the convex regularizers $h_i$ vary across individuals. This seamlessly recovers standard models under a single umbrella, including the linear--quadratic specification \citep{ballester2006s,li2025price,chen2018competitive} and logit-type models \citep{chen2021duopoly,du2016optimal}. Second, we explicitly accommodate interaction matrices $\mathbf{G}$ that are \emph{directed} and \emph{signed}, extending models with symmetric and purely positive externalities \citep{candogan2012optimal} or negative  \citep{li2025pricing}. Under these highly flexible assumptions, we derive distinct \emph{contraction} and \emph{variational} conditions that guarantee the existence and uniqueness of the consumer equilibrium for any price vector. We theoretically bridge this to prior literature by proving that, in the classical symmetric and nonnegative setting, these distinct uniqueness conditions perfectly coincide and reduce to the familiar spectral threshold \citep{chen2018competitive,li2025price}.

\item \textbf{Price discrimination across communities and influencers.} We analyze how targeted price discrimination emerges across community structures and in the presence of influencers. We first show that pricing dynamics decompose naturally across independent communities. We then introduce a generalized definition of network influence that extends classical Katz--Bonacich centrality \citep{kohlberg1983equilibrium, li2025pricing} from the Euclidean domain to the non-Euclidean dual space induced by the marginal response map $\boldsymbol{\phi}$. Using this metric, we prove a strict lower bound on optimal prices for low-influence consumers, formally demonstrating that steep discounts are mathematically reserved exclusively for influencers.

\item \textbf{A tuning-free learning algorithm via isotonic estimation.} For settings where the customer utility functions are unknown, we develop a no-regret learning algorithm for the seller's revenue maximization problem. Our approach pairs price optimization with nonparametric isotonic estimation of the marginal response functions, yielding a fully data-driven, tuning-free pricing pipeline. To theoretically ground this method, we recover a finite-sample concentration inequality for the uniform error of the isotonic estimator — a result of independent statistical interest (Theorem~\ref{thm:Dumbgen}) — which allows us to rigorously bound the expected pricing regret.
\end{enumerate}

\textbf{General notation.} We denote matrices by bold uppercase letters, vectors by bold lowercase letters, and real-valued scalars by non-bold letters. For each integer $N$, $\mathbf{I}_N$ and $\mathbf{O}_N$ denote the $N \times N$ identity matrix and zero matrix, respectively. For a given $N \times N$ matrix $\mathbf{M}$, we define $\rho(\mathbf{M}):=\sup\{|\lambda|:\lambda\in\sigma(\mathbf{M})\}$, where $\sigma(\mathbf{M}):=\{\lambda\in\mathbb C:\mathbf{M}-\lambda \mathbf{I}_N
\text{ is not invertible}\}$. The value $\rho(\mathbf{M})$ and $\sigma(\mathbf{M})$ are also known as the spectral radius of $\mathbf{M}$ and spectrum of $\mathbf{M}$, respectively. We write $\mathbf{M} \succ \mathbf{M'}$ if $\mathbf{z}^{\top} \mathbf{M} \mathbf{z}>\mathbf{z}^{\top} \mathbf{M}'\mathbf{z}$ for all $\mathbf{z} \in \mathbb{R}^N$ with $\mathbf{z}\neq \mathbf{0}$.

\section{Problem Formulation}

\textbf{Model Setup.}  We consider a monopolist selling a product to a set of consumers $\mathcal{N} = \{1, \dots, N\}$ embedded in a social network. The network topology is represented by an adjacency matrix $\mathbf{G} \in \mathbb{R}^{N \times N}$ with zero diagonal ($g_{ii} = 0$), where $g_{ij}$ measures the influence of consumer $j$ on consumer $i$. Crucially, unlike previous works \citep{chen2018competitive}, we allow $\mathbf{G}$ to be both \emph{directed} (asymmetric) and \emph{signed} (permitting negative weights/externalities). The monopolist acts as the sole product provider with full price discrimination capabilities. Given the posted prices, each consumer $i \in \mathcal{N}$ chooses a consumption level $x_i$ from a feasible set $\mathcal{X}_i \subseteq \mathbb{R}_+$.

\textbf{Consumer objective.} Suppose the seller posts a price vector $\mathbf{p}=(p_1,\dots,p_N)^{\top} \in \mathcal{P}$, where $\mathcal{P} \subseteq \mathbb{R}^N$ is closed and convex, and $p_i$ is the discriminatory price offered to consumer $i$. Each consumer $i \in \mathcal{N}$ selects a consumption level $x_i \in \mathcal{X}_i$ to maximize their utility. The feasible set $\mathcal{X}_i$ is typically $[0,\infty)$, bounded as $[0, a]$, or $[0,1]$ when $x_i$ represents the probability of adoption. The utility function $u_i$ is defined as:
\begin{equation}\label{eq:general_utility_unified}
\textstyle  u_i(x_i,\mathbf{x}_{-i};p_i)
=
 s_i(\mathbf{x}_{-i};p_i) x_i- h_i(x_i), \qquad \text{where} \quad s_i(\mathbf{x}_{-i};p_i)\triangleq a_i-b_ip_i+\delta \sum_{j=1}^N g_{ij}x_j.
\end{equation}
Here, $a_i$ represents the intrinsic baseline utility, $b_i$ is the price sensitivity, and $h_i: \mathcal{X}_i \to \mathbb{R}$ is a convex function capturing decreasing marginal returns. The term $\delta (\mathbf{G} \mathbf{x})_i = \delta \sum_{j=1}^N g_{ij}x_j$ encapsulates the local network spillover, with $\delta \ge 0$ scaling the overall intensity of peer influence. Because $g_{ij}$ can take any sign, the model natively captures both positive and negative externalities: the cross-partial derivative $\frac{\partial^2 u_i}{\partial x_i \partial x_j} = \delta g_{ij}$ explicitly shows that higher consumption by peer $j$ can either increase ($g_{ij} > 0$) or decrease ($g_{ij} < 0$) consumer $i$'s marginal utility.

\textbf{Seller objective.} Anticipating the network equilibrium $\mathbf{x}^\star(\mathbf{p})$ generated by the consumers, the monopolist seeks to determine a price vector $\mathbf{p}^\star \in \mathcal{P}$ that maximizes total profit. The seller's optimization problem aims to maximize revenue:
\begin{equation}\label{eq:Opt-prob}
  \textstyle  \mathbf{p}^\star \in \arg\max_{\mathbf{p} \in \mathcal{P}} r(\mathbf{p}), \qquad \text{where} \quad r(\mathbf{p}) \triangleq \mathbf{p}^\top \mathbf{x}^\star(\mathbf{p}).
\end{equation}

\begin{remark}\label{remark:classical-models}
The formulation in \eqref{eq:general_utility_unified} captures a broad class of canonical utility models via the choice of the function $h_i(x)$. For instance, it natively recovers the \emph{Linear--Quadratic} model ($h_i(x)=\frac{1}{2}x^2$) for standard goods \citep{ballester2006s,li2025price,chen2018competitive} and the \emph{Discrete Choice/Logit} model ($h_i(x)=x\ln x+(1-x)\ln(1-x)$) for adoption probabilities \citep{chen2021duopoly,du2016optimal}. Further, as detailed in Appendix~\ref{ex:utilities-covered}, this framework seamlessly accommodates several other prominent specifications, including \emph{Stone--Geary} (subsistence goods), \emph{Exponential} (risk aversion), and \emph{Isoelastic} (Box--Cox) utility models.
\end{remark}

Relaxing the linear--quadratic assumption substantially transforms the seller's pricing problem. As shown in Appendix~\ref{app:linear-logistic}, the standard linear--quadratic model harbors a strict invariance property: for any symmetric, nonnegative $\mathbf G$, optimal monopoly prices are decoupled from the network topology, yielding $p_i^\star=\frac{a_i}{2b_i}$, $i\in\mathcal N$. This decoupling, however, is a brittle artifact of linearity. Under general nonlinear utilities, this invariance breaks down, meaning the network topology actively dictates both the equilibrium consumption and the seller's optimal pricing strategy.

\subsection{Consumers Nash Equilibrium}

For fixed prices $p_i$ and peer consumption $\mathbf{x}_{-i}$, we define consumer $i$'s \textit{best response} as the set:
\begin{equation}\label{eq:general_BR_map_monopolostic}
\textstyle    \Gamma_i(\mathbf{x}_{-i};p_i) \triangleq \arg\max_{x_i\in\mathcal{X}_i} u_i(x_i,\mathbf{x}_{-i};p_i).
\end{equation}

\begin{definition}\label{def:NE}
For every $\mathbf{p} \in \mathcal{P}$, a \textit{Nash Equilibrium} (NE) is a consumption vector 
$$
\mathbf{x}^{\star}(\mathbf{p}) = (x_1^{\star}(\mathbf{x}_{-1}^{\star},p_1),\dots, x_N^{\star}(\mathbf{x}_{-N}^{\star},p_N))\in\mathcal{X}, \quad 
x_i^{\star}(\mathbf{x}_{-i}^{\star},p_i) \in \Gamma_i(\mathbf{x}^\star_{-i};p_i), \quad \forall i\in\mathcal{N}.
$$
When each $h_i$ is continuously differentiable, 
this condition is equivalently characterized by the variational inequality 
$$ \text{for every }\mathbf{p} \in \mathcal{P}, \quad
\left \langle h_i'(x_i^{\star}) -s_i(\mathbf{x}_{-i}^{\star};p_i), y_i-x_i^{\star} \right \rangle \geq 0, \quad \forall y_i \in \mathcal{X}_i, \quad \forall i \in \mathcal{N}.
$$
\end{definition}
For the remainder of the paper, we assume that the consumer game admits a NE. This is a mild assumption, as existence holds under standard weak conditions for both unbounded (e.g., linear--quadratic) and bounded (e.g., discrete-choice model) feasible sets. The following theorem establishes one such general existence guarantee.

\begin{theorem}[Existence of Nash equilibrium]
\label{thm:general_competitive_NE_existence_monopolist}
Assume that $h_i$ is continuously differentiable for each $i \in \mathcal{N}$ and one of the following holds: \\
\textbf{1. Compactness:} $\forall i \in \mathcal{N}$, $\mathcal{X}_i\subseteq \mathbb R$ is nonempty convex and compact set. \\
\textbf{2. Coercivity:} $\forall i \in \mathcal{N}$, $\mathcal{X}_i\subseteq \mathbb R$ is nonempty, convex and closed, and $\lim_{x_i \in \mathcal{X}_i,|x_i|\to\infty}\tfrac{h_i(x_i)}{|x_i|}= +\infty$. \\
Then, the  game induced by the best-response mapping \eqref{eq:general_BR_map_monopolostic} admits at least one Nash equilibrium.
\end{theorem}

The proof of Theorem~\ref{thm:general_competitive_NE_existence_monopolist} follows directly from Theorem 2.30 in \cite{tsekrekos2024variational}, leveraging their variational inequality framework.

\begin{remark}
The canonical utility functions discussed in Appendix~\ref{ex:utilities-covered} satisfy the assumptions of Theorem~\ref{thm:general_competitive_NE_existence_monopolist}. The Discrete-Choice model guarantees existence via its compact domain $\mathcal{X}_i=[0,1]$. For unbounded specifications (Linear--Quadratic, Stone--Geary, Exponential, and Isoelastic), existence holds because of the strict convexity and coercivity of their respective $h_i$ functions.
\end{remark}

\subsection{Uniqueness of Nash equilibrium}

Assuming the strong convexity of the $h_i$ functions, we establish equilibrium uniqueness via two independent analytical routes: a Banach contraction argument and a variational inequality approach. These methods yield distinct uniqueness conditions that do not generally imply each other. Notably, however, we show that these two criteria collapse into a single, unified condition precisely when $\mathbf{G}$ is symmetric and entry-wise non-negative.

\begin{assumption}\label{ass:strong-convexity}
Assume that $h_i$ is continuously differentiable and $\mu_i$-strongly convex with respect to the Euclidean norm $|\cdot |$; i.e.,  for all $x,y\in \mathbb{R}$, $\textstyle h_i(y)\ge h_i(x)+  h'_i(x) \cdot (y-x) + \frac{\mu_i}{2} |y-x |^2$.
\end{assumption}

\begin{remark} \label{rem:BR_fixed_point_characterization}
Because $h_i$ is strongly convex, the consumer's best response \eqref{eq:general_BR_map_monopolostic} is uniquely defined, rendering $\Gamma_i$ single-valued. For differentiable $h_i$, letting $\phi_i(x_i) \triangleq h_i'(x_i)$, the unconstrained first-order condition is:
\begin{equation}\label{FOC:consumer}
 \textstyle   a_i - b_ip_i + \delta\sum_{j=1}^N g_{ij}x_j - \phi_i(x_i) = 0, \qquad \forall i \in \mathcal{N}.
\end{equation}
Since strict convexity implies $\phi_i$ is invertible, the constrained best response on the interval $\mathcal{X}_i$ is simply the projection $\Gamma_i(\mathbf{x}_{-i}; p_i) = \Pi_{\mathcal{X}_i} \{ \phi_i^{-1}(a_i - b_ip_i + \delta\sum_{j=1}^N g_{ij}x_j) \}$. Stacking these equations yields the global first-order condition:
\begin{equation}\label{FOC:consumer-global}
    \mathbf{F}(\mathbf{x}) = -\mathbf{B}\mathbf{p}, \qquad \text{where} \quad \mathbf{F}(\mathbf{x}) \triangleq \mathbf{a} - \boldsymbol{\phi}(\mathbf{x}) + \delta \mathbf{G}\mathbf{x},
\end{equation}
with $\boldsymbol{\phi}(\mathbf{x}) \triangleq (\phi_1(x_1), \dots, \phi_N(x_N))^\top$ and $\mathbf{B} \triangleq \operatorname{diag}(b_1, \dots, b_N)$. Consequently, the joint best-response mapping is $\boldsymbol{\Gamma}(\mathbf{x}; \mathbf{p}) = \Pi_{\mathcal{X}} \{ \boldsymbol{\phi}^{-1}(\mathbf{a} - \mathbf{B}\mathbf{p} + \delta \mathbf{G}\mathbf{x}) \}$. The consumer Nash equilibrium $\mathbf{x}^\star(\mathbf{p})$ is precisely the fixed point of $\boldsymbol{\Gamma}(\cdot; \mathbf{p})$, which is unique under appropriate conditions on $\delta$ and $\mathbf{G}$.
\end{remark}

\begin{theorem}[Uniqueness of Nash Equilibrium]
\label{thm:uniqueness_NE}
Under the hypotheses of Theorem~\ref{thm:general_competitive_NE_existence_monopolist} and Assumption~\ref{ass:strong-convexity}, let $|\mathbf{G}|$ denote the entry-wise absolute value of $\mathbf{G}$, and define $\mathbf{D} \triangleq \operatorname{diag}(\mu_1, \dots, \mu_N)$. For every $\mathbf{p} \in \mathcal{P}$, the game admits a unique Nash equilibrium $\mathbf{x}^\star(\mathbf{p})$ if either of the following holds:
\[
    \textbf{Contraction Condition:} \quad \rho(\delta \mathbf{D}^{-1} |\mathbf{G}|) < 1 \qquad \text{or} \qquad \textbf{Variational Condition:} \quad \mathbf{D} - \delta \tfrac{\mathbf{G} + \mathbf{G}^\top}{2} \succ 0.
\]
\end{theorem}


This terminology is deliberate: the \emph{contraction condition} applies the Banach contraction theorem to the best-response mapping, while the \emph{variational condition} stems from the variational inequality formulation of the consumer game.

Theorem~\ref{thm:uniqueness_NE} unifies two distinct uniqueness criteria: Theorem~\ref{thm:general_competitive_NE_uniqueness_monopolist} under the contraction condition $\rho(\delta\mathbf{D}^{-1}|\mathbf{G}|)<1$, and Theorem~\ref{thm:general_competitive_NE_uniqueness_VI} under the variational condition $\mathbf{D}-\delta\tfrac{\mathbf{G}+\mathbf{G}^\top}{2}\succ \mathbf{O}_N$. Formal derivations are deferred to Appendix~\ref{appendix:uniqueness}, allowing us to focus our discussion on their interpretation.

Theorem~\ref{thm:uniqueness_NE} combines two distinct proof strategies. On the one hand, we analyze the joint best-response map and prove uniqueness via Banach's contraction theorem, overcoming the mismatch between the spectral-radius condition and standard norm bounds through an ad-hoc weighted sup norm construction. On the other hand, we reformulate the consumer game as a variational inequality and obtain uniqueness from the strict monotonicity of the associated operator, which yields a condition involving only the symmetric part of~$\mathbf G$.

\begin{remark}[Economic Interpretation]
Per the first-order condition \eqref{FOC:consumer}, consumer choice balances individual utility curvature ($\phi_i(x_i)$) against aggregate network effects ($\delta\sum_j g_{ij}x_j$). Both uniqueness criteria require this individual curvature to strictly dominate network effects, albeit in different ways. \textbf{Contraction Condition:} ensures network effects can not trigger amplifying feedback loops. Because it relies on the absolute magnitude $|\mathbf{G}|$, it is a conservative bound that ignores potential cancellations between positive and negative peer effects. \textbf{Variational Condition:} depends exclusively on the symmetric component $\frac{\mathbf{G}+\mathbf{G}^\top}{2}$. It bounds only the mutually reinforcing interactions rather than the worst-case absolute spillovers, ensuring network effects shift baseline consumption without fracturing the game into multiple equilibria. 

Section~\ref{remark:Gershgorin} derives practical sufficient conditions for both bounds via the Gershgorin circle theorem.
\end{remark}

As demonstrated in Examples~\ref{example_complete_Linear} (linear-quadratic utility model)
and~\ref{example_complete_Discrete} (discrete choice model) in Appendix \ref{subsec:contraction-variational-coincide}, the contraction and variational conditions perfectly coincide when $\mathbf{G}$ is symmetric and entry-wise non-negative. In this restricted setting, our dual criteria collapse into the precise spectral thresholds established by \cite{ballester2006s} and \cite{li2025price} (Assumption 2). Consequently, because Theorem~\ref{thm:uniqueness_NE} accommodates arbitrary asymmetry and negative peer effects, it strictly generalizes these classical uniqueness results.

\begin{theorem}[Symmetric Nonnegative Networks]\label{thm:comparison}
If $\mathbf{G}$ is symmetric and entry-wise non-negative, the contraction and variational conditions are equivalent, with both reducing to $\delta \rho(\mathbf{D}^{-1}\mathbf{G}) < 1$.
\end{theorem}

\begin{remark}[Contraction vs. Variational Condition]
\label{rem:VI_condition_importance_monopolist}
While Theorem~\ref{thm:comparison} establishes equivalence for symmetric, entry-wise non-negative networks, the two criteria are generally incomparable otherwise. Assuming $\mathbf{D} = \mathbf{I}_2$ and $V>0$, the directed network $\mathbf{G} = \left(\begin{smallmatrix} 0 & V\\ 0 & 0 \end{smallmatrix}\right)$ yields $\rho(|\mathbf{G}|) = 0$ (satisfying contraction for any $\delta>0$), but restricts the variational condition to $\delta < 2/V$. Conversely, the skew-symmetric network $\mathbf{G} = \left(\begin{smallmatrix} 0 & V\\ -V & 0 \end{smallmatrix}\right)$ yields $\tfrac{\mathbf{G}+\mathbf{G}^\top}{2} = \mathbf{O}_2$ (satisfying the variational condition for any $\delta>0$) but restricts contraction to $\delta < 1/V$. Thus, neither bound universally dominates the other. Further examples are provided in Appendix~\ref{sec:contactio_variational_examples}.
\end{remark}

\section{Price discrimination across communities and influencers} \label{sec:influencers}

In this section, we examine two topological settings where network effects naturally induce price discrimination: community structures and influencers.

\textbf{Communities.}
Suppose that, up to node reordering, the interaction matrix is block-diagonal: $\mathbf{G} = \operatorname{diag}(\mathbf{G}^{(1)}, \dots, \mathbf{G}^{(K)})$, where each block $\mathbf{G}^{(k)}$ represents an isolated community. Because there are no cross-community spillovers, both the consumer equilibrium and the overarching monopoly pricing problem decompose additively (as formally proven in Appendix~\ref{app:examples_of_networks}). Consequently, the optimal price $p_i$ for any consumer $i$ in community $k$ depends exclusively on the local network structure $\mathbf{G}^{(k)}$ and the local parameters $(a_j, b_j, h_j)$ of that specific block. This implies the monopolist will optimally price discriminate across communities whenever they differ in size, density, or spectral properties. Notably, two consumers with identical individual parameters $(a_i, b_i, h_i)$ may face entirely different prices simply by residing in different communities.



\textbf{Influencers.}
As demonstrated in Appendix~\ref{app:more_examples}, a monopolist optimally assigns lower prices to consumers who exert stronger influence on the network. To formalize this, we derive a structural measure of influence. Assuming an interior equilibrium $\mathbf{x}^\star(\mathbf{p}) \in \mathring{\mathcal{X}}$ (assumption widely used in the literature \citep{li2025price,chen2021duopoly,candogan2012optimal}), implicitly differentiating \eqref{FOC:consumer-global} yields the price sensitivity of demand:
\[
    \nabla_{\mathbf{p}}\mathbf{x}^\star(\mathbf{p}) = -\mathbf{J}(\mathbf{x}^\star(\mathbf{p}))^{-1} \mathbf{B}, \qquad \text{where} \quad \mathbf{J}(\mathbf{x}) \triangleq D_{\boldsymbol{\phi}}(\mathbf{x}) - \delta \mathbf{G}.
\]
Because the Jacobian $\mathbf{J}$ depends exclusively on the network strength $\delta\mathbf{G}$ and utility curvature $\boldsymbol{\phi}$, and \textit{not on prices}, the sum of the $i$-th column of $\mathbf{J}^{-1}$ naturally captures consumer $i$'s intrinsic power to influence the broader network. 

\begin{definition}[Intrinsic Influential Value (IIV)]
Given a network $\mathbf{G}$ with strength $\delta$ and utility curvature $\boldsymbol{\phi}$, the \emph{intrinsic influential value} (or \emph{generalized Katz--Bonacich centrality}) of consumer $i \in \mathcal{N}$ is defined as:
$V_i \triangleq \sup_{\mathbf{x} \in \mathcal{X}} [\mathbf{1}^\top \mathbf{J}(\mathbf{x})^{-1}]_i$.
This isolates consumer $i$'s structural influence across the network, independent of specific consumption levels or prices.
\end{definition}

When $\mathbf{G}$ is entry-wise nonnegative ($\mathbf{G}=|\mathbf{G}|$) and the contraction condition $\rho(\delta \mathbf{D}^{-1}\mathbf{G}) < 1$ holds, $\mathbf{J}(\mathbf{x})^{-1}$ is nonnegative, meaning $V_i \ge 0$. A higher $V_i$ indicates greater network influence. Proposition~\ref{prop:v_i-equality} provides a clean analytical bound for this value utilizing the utility curvature parameter $\mathbf{D} = \operatorname{diag}(\mu_1, \dots, \mu_N)$.

\begin{proposition}\label{prop:v_i-equality}
Suppose Assumption~\ref{ass:strong-convexity} holds, $\mathbf{G}=|\mathbf{G}|$, and $\rho(\delta \mathbf{D}^{-1}\mathbf{G}) < 1$. Then, $V_i \le [\mathbf{1}^\top(\mathbf{D}-\delta \mathbf{G})^{-1}]_i$. Further, equality holds if there exists $\mathbf{x}^\dagger \in \mathcal{X}$ such that $D_{\boldsymbol{\phi}}(\mathbf{x}^\dagger) = \mathbf{D}$.
\end{proposition}

\begin{remark}
In the linear--quadratic model ($\mathbf{D} = \mathbf{I}_N$), the Jacobian is constant ($\mathbf{J} = \mathbf{I}_N - \delta\mathbf{G}$), and $V_i$ collapses exactly to the standard Katz--Bonacich centrality \citep{kohlberg1983equilibrium, li2025pricing, candogan2012optimal}. Thus, IIV serves as the natural generalization of Katz--Bonacich centrality for nonlinear utility models. We defer the technical discussion regarding signed networks and dual-space interpretations to Appendix~\ref{sec:influencer_appendix}.
\end{remark}
We can now cleanly define a network influencer using this structural metric.
\begin{definition}[Influencer]
A consumer $i \in \mathcal{N}$ is an \emph{influencer} if: (1) they are uninfluenced by others (the $i$-th row of $\mathbf{G}$ is zero), and (2) their IIV ($V_i$) is large.
\end{definition}

This topological definition carries direct pricing implications. In Proposition~\ref{prop:general_pure_influencer}, we demonstrate that if an influencer is powerful enough to warrant an optimal price near zero, forcing $p_i=0$ strictly causes only a second-order loss in monopoly revenue. Conversely, Corollary~\ref{cor:structural_influencer_zero_price} proves that if an uninfluenced consumer has a low $V_i$, their optimal price $p_i^\star$ is strictly bounded away from zero by a constant, independent of the overarching equilibrium price vector $\mathbf{p}^\star$. For details, see Appendix \ref{subsec:influencer-pricing}.

\section{Seller optimization with unknown customer utility} \label{sec:seller-optimization}

This section develops an online learning algorithm to estimate the optimal price $\mathbf{p}^\star$ of \eqref{eq:Opt-prob} when the monotone utility curvature $\boldsymbol{\phi}$ is unknown, assuming the network $\delta \mathbf{G}$ and sensitivities $\mathbf{B}$ remain known. Since the intercepts $\mathbf{a}$ are additive, we absorb them into the unknown mapping by defining $\psi_i(x_i) \triangleq \phi_i(x_i)-a_i, \ (\boldsymbol{\psi}(\mathbf{x}) \triangleq \boldsymbol{\phi}(\mathbf{x}) - \mathbf{a})$.

\begin{assumption}\label{ass:strong-concavity-estim}
There exists a constant $m>0$ such that $\mathbf{D}-\delta\,\tfrac{\mathbf{G}+\mathbf{G}^\top}{2} \succ m\mathbf{I}_N$.
\end{assumption}

Assumption~\ref{ass:strong-concavity-estim} (together with Assumption~\ref{ass:strong-convexity}) implies that the map $\mathbf F(\mathbf x):=\boldsymbol{\psi}(\mathbf x)-\delta \mathbf G\mathbf x$ is $m$-strongly monotone (see Appendix~\ref{app:proof-monotonicity-F}), and implies the variational condition. Hence, by Theorem~\ref{thm:general_competitive_NE_uniqueness_VI}, for every $\mathbf{p} \in \mathcal{P}$, there exists a unique $\mathbf{x}^{\star}(\mathbf{p})$ satisfying the equilibrium condition, which, assuming that $\mathcal{P}\subseteq \mathbf{Q}(\mathring{\mathcal{X}})$, where $\mathbf{Q}=-\mathbf{B}^{-1} \circ \mathbf{F}$ (i.e. the consumption equilibrium falls inside $\mathcal{X}$, assumption widely used in the literature \citep{li2025price,chen2021duopoly,candogan2012optimal}) uniquely satisfies $\mathbf F(\mathbf x^\star(\mathbf p))=-\mathbf B\mathbf p$. Our algorithm (summarized in Algorithm~\ref{alg:isotonic_plugin_pricing}) partitions the horizon $T$ into an \textbf{exploration phase} $E=\{1,\dots, T_0\}$ and an \textbf{exploitation phase} $E'=\{T_0+1,\dots, T\}$.
In $E$, the seller estimates the unknown monotone functions $\psi_i$. At each $t \in E$, the seller posts a randomized $\mathbf{p}^t \in \mathcal{P}$ and observes the resulting consumption $\mathbf{x}^t \in \mathcal{X}$, satisfying:
\begin{equation}\label{eq:coordinate_regression_isotonic}
    \mathbf{F}(\mathbf{x}^t) = -\mathbf{B}\mathbf{p}^t + \boldsymbol{\xi}^t \quad \Longleftrightarrow \quad y_i^t = \psi_i(x_i^t) + \xi_i^t, \quad \forall i \in \mathcal{N},
\end{equation}
where $\boldsymbol{\xi}^t$ is a zero-mean noise vector and $y_i^t \triangleq -b_i p_i^t + \delta(\mathbf{G}\mathbf{x}^t)_i$. Because $\delta \mathbf{G}$ and $\mathbf{B}$ are known, $y_i^t$ is fully observable, reducing the estimation to $N$ \textit{independent one-dimensional} monotone regressions. For each coordinate $i$, we sort the samples such that $x_i^{(1)} \le \cdots \le x_i^{(T_0)}$ and compute the isotonic least-squares estimate via the pool-adjacent-violators algorithm (PAVA) \citep{robertson1988order}:
\begin{equation}\label{eq:isotonic_estimator}
    \textstyle (\widehat{m}_i^{(1)},\dots,\widehat{m}_i^{(T_0)}) \in \arg\min_{m_1 \le \cdots \le m_{T_0}} \sum_{t=1}^{T_0} (y_i^{(t)} - m_t)^2.
\end{equation}
We then define $\widehat{\psi}_i$ on $\mathcal{X}_i$ as the piecewise-linear interpolation of these fitted values, stacking them to form the global estimator $\widehat{\boldsymbol{\psi}}(\mathbf{x}) \triangleq (\widehat{\psi}_1(x_1),\dots,\widehat{\psi}_N(x_N))^\top$. During exploitation ($t \in E'$), we define the empirical operator $\widehat{\mathbf{F}}(\mathbf{x}) \triangleq \widehat{\boldsymbol{\psi}}(\mathbf{x}) - \delta \mathbf{G}\mathbf{x}$ and compute the estimated equilibrium $\widehat{\mathbf{x}}^\star(\mathbf{p})$ by solving $\widehat{\mathbf{F}}(\widehat{\mathbf{x}}^\star(\mathbf{p})) = -\mathbf{B}\mathbf{p}$. The seller then repeatedly posts the price vector that maximizes the empirical revenue: $\mathbf{p}^t \in \arg\max_{\mathbf{p} \in \mathcal{P}} \mathbf{p}^\top \widehat{\mathbf{x}}^\star(\mathbf{p})$. We compare our algorithm against an oracle that knows an optimal solution $\mathbf{p}^{\star}$ to \eqref{eq:Opt-prob} in advance. The cumulative regret is: $R(T) \triangleq \sum_{t=1}^T \bigl( r(\mathbf{p}^\star) - r(\mathbf{p}^t) \bigr)$.

\begin{algorithm}[h]
\caption{}
\label{alg:isotonic_plugin_pricing}
\begin{algorithmic}[1]
\Require network $\mathbf G$, price sensitivities $\mathbf B$, total horizon $T$, exploration horizon $T_0$, feasible set $\mathcal P$.
\For{$t=1,\dots,T_0 \text{ (\textbf{Exploration phase})}$}
    \State $\blacktriangleright$ Sample $\mathbf p^t\sim \operatorname{Unif}(\mathcal P)$ and observe $\mathbf x^t\in\mathcal X$, which is sampled as in \eqref{eq:coordinate_regression_isotonic}.
    \State $\blacktriangleright$ $\forall i\in\mathcal N$, form $y_i^t:=-b_i p_i^t+\delta(\mathbf G\mathbf x^t)_i$.
\EndFor
\State $\blacktriangleright$ $\forall i\in\mathcal N$, estimate $\widehat{\psi}_i$ by isotonic regression using the exploration sample $\{(x_i^t,y_i^t)\}_{t=1}^{T_0}$
\State $\blacktriangleright$ Let $\mathbf{p}\mapsto\widehat{\mathbf{x}}^{\star}(\mathbf{p})$ such that
$\widehat{\mathbf{F}}(\widehat{\mathbf{x}}^{\star}(\mathbf{p}))
=
-\mathbf B\mathbf p$, where $\widehat{\mathbf{F}}(\mathbf x):=\widehat{\boldsymbol{\psi}}(\mathbf x)-\delta \mathbf G\mathbf x$.
\For{$t=T_0+1,\dots,T\textbf{ (Exploitation phase)}$}
    \State $\blacktriangleright$ Return any value  $\mathbf{p}^t \in \arg\max_{\mathbf p\in\mathcal P}\left\{\widehat{r}(\mathbf p)=\mathbf{p}^\top \widehat{\mathbf{x}}^{\star}(\mathbf{p})\right\}$
\EndFor
\end{algorithmic}
\end{algorithm}

\begin{theorem}
\label{thm:dynamic_regret_isotonic_network}
Suppose Assumptions~\ref{ass:strong-convexity} and \ref{ass:strong-concavity-estim} hold. Let the joint action space $\mathcal{X} = \prod_{i=1}^N \mathcal{X}_i$ be compact and $\mathcal{P} \subseteq \mathbf{Q}(\mathring{\mathcal{X}})$, where $\mathbf{Q} \triangleq -\mathbf{B}^{-1} \circ \mathbf{F}$. Assume the unknown functions $\psi_i: \mathcal{X}_i \to \mathbb{R}$ are monotone non-decreasing and $\alpha_i$-H\"older continuous: $|\psi_i(u) - \psi_i(v)| \le L_i|u-v|^{\alpha_i}$ for all $u, v \in \mathcal{X}_i$, with $L_i > 0$ and $\alpha_i \in (0,1]$. Further, let the noise vectors $\boldsymbol{\xi}^t \in \mathbb{R}^N$ have independent, zero-mean sub-Gaussian coordinates with variance proxies $\sigma_i^2$. 

If the exploration length is $T_0 = \lceil T^\beta \rceil$ with $\beta = \frac{2\alpha+1}{3\alpha+1}$ and $\alpha = \min_i \alpha_i$, then for all sufficiently large $T$, with probability at least $1 - \frac{N}{T^{\beta(\gamma-2)}}$, we have:
\[
    \textstyle R(T) \lesssim \sqrt{N}\,T^{\frac{2\alpha+1}{3\alpha+1}} \ln^{\frac{\alpha}{2\alpha+1}}(T), \quad \text{and} \quad \sup_{\mathbf{p} \in \mathcal{P}} \|\widehat{\mathbf{x}}^{\star}(\mathbf{p}) - \mathbf{x}^\star(\mathbf{p})\|_2 \lesssim \sqrt{N} (\nicefrac{\ln T_0}{T_0})^{\frac{\alpha}{2\alpha +1}}.
\]
If the revenue $r$ is additionally strongly concave, then $\|\mathbf{p}^t - \mathbf{p}^\star\|_2^2 \lesssim \sqrt{N} (\nicefrac{\ln T_0}{T_0})^{\frac{\alpha}{2\alpha+1}}$ holds with the same probability.
\end{theorem}

\begin{remark}[The Low-Dimensional Regime]\label{rem:low-dim-N}
Theorem~\ref{thm:dynamic_regret_isotonic_network} operates in a \emph{low-dimensional} regime where the number of consumers $N$ is much smaller than the time horizon ($N \ll T$). Because the seller must estimate $N$ distinct coordinate-wise response functions from $T_0$ exploration samples, both the equilibrium estimation error and the cumulative regret scale as $\mathcal{O}(\sqrt{N})$. Consequently, the high-probability guarantee of $1-\frac{N}{T^{\beta(\gamma-2)}}$ remains meaningful only if $N$ grows sufficiently slowly relative to $T$.
\end{remark}

\begin{remark}[Market Segmentation and Shared Utilities]
\label{rem:segmentation}
Market segmentation drastically mitigates the $\mathcal{O}(\sqrt{N})$ dependence of Remark~\ref{rem:low-dim-N}. In practice, consumers are often grouped by observable context into $K \ll N$ types sharing identical utilities ($h_i = h_k$). Such structural clustering is fairly standard in dynamic pricing \citep{ban2021personalized, ferreira2016analytics} and network learning \citep{cesa2013gang, gentile2014online}. By pooling exploration samples within each segment, the seller estimates only $K$ scalar functions. This reduces the effective dimensionality, allowing regret and estimation errors to scale gracefully as $\mathcal{O}(\sqrt{K})$. Consequently, our guarantees extend to large-scale networks ($N \gg T$), provided $K$ remains small.
\end{remark}

\begin{remark}[Computational Complexity]
The total complexity of the plug-in procedure is:
$$
\underbrace{\mathcal{O} \bigl(T_0\,\mathrm{nnz}(\mathbf{G})\bigr)}_{\text{constructing responses } y_i^t}
\;+\;
\underbrace{\mathcal{O} \bigl(NT_0\ln T_0\bigr)}_{\text{isotonic estimation of } \widehat{\boldsymbol{\psi}}}
\;+\;
\underbrace{\mathcal{O}\bigl(N+\mathrm{nnz}(\mathbf{G})\bigr)}_{\text{evaluating } \widehat{\mathbf{x}}^{\star} \text{ and optimizing}},
$$
where $\mathrm{nnz}(\mathbf{G})$ denotes the number of nonzero entries in the network matrix $\mathbf{G}$. We defer a full discussion of the computational methods for Algorithm~\ref{alg:isotonic_plugin_pricing} to Appendix~\ref{app:comp_comp}.
\end{remark}

\paragraph{Experiments.} We evaluate the finite-sample performance of Algorithm~\ref{alg:isotonic_plugin_pricing} on synthetic network markets. Consistently with Remark~\ref{rem:low-dim-N}, we focus on a low-dimensional regime in which the number of consumers is fixed at $N=20$, while the time horizon varies over $\mathcal{T}=\{25,50,75,100,125\}$. A detailed description of the experimental setup is deferred to Appendix~\ref{app:experiments}. We consider several specifications of the convex regularizer $h_i=h$, equivalently of its derivative $\phi=h'$:
\begin{itemize}[leftmargin=1em,noitemsep,topsep=0pt]
    \item \textbf{Linear--Quadratic utility:} $h(x)=\tfrac{1}{2}x^2$. In this case, $\phi(x)=x$, so the marginal response map is linear and belongs to the H\"older class with exponent $\alpha=1$.
    
    \item \textbf{Power utility:} $h(x)=\tfrac{1}{\gamma+1}x^{\gamma+1}$, with $\gamma\in\{0.2,0.4,0.6,0.8\}$. Here $\phi(x)=x^{\gamma}$, and the corresponding regularity $\alpha$ is indexed by the exponent $\gamma$, allowing us to test the algorithm under progressively less regular nonlinear responses.
    
    \item \textbf{Discrete-choice utility:} $h(x)=x\ln x+(1-x)\ln(1-x)$. In this case, $\phi(x)=\ln \bigl(\frac{x}{1-x}\bigr)$ diverges at the boundary of $[0,1]$, and therefore does not satisfy a global H\"older condition on the whole domain. We nevertheless include this specification in the experiments to assess the robustness of the proposed procedure beyond the exact scope of the theory.
\end{itemize}

We report the results for two different network structures $\mathbf G$. More network structures experiments validating our results in Theorem~\ref{thm:dynamic_regret_isotonic_network} can be found in Appendix~\ref{app:experiments}.

\noindent
\begin{minipage}[t]{0.33\textwidth}
\vspace{0pt}
\textbf{1. Sparse circular $\mathbf{G}$.} We generate $\mathbf G$ as a sparse directed circular network. Specifically, for each $i$ we set $g_{i,(i+1)\pmod N}=w$, $g_{i,(i+2)   \pmod N}=\tfrac{w}{2}$, $g_{ii}=0$, with $w=0.08$ and all remaining entries equal to zero. Hence, each node is influenced by its first successor on the circle and, more weakly, by its second successor. To introduce signed interactions, we flip the sign of $10\%$ of the nonzero entries. This yields a sparse directed network with an underlying circular topology and a few negative spillovers. Figure~\ref{fig:tau_2} illustrates the convergence of the total regret, validating our results in Theorem~\ref{thm:dynamic_regret_isotonic_network}.

\end{minipage}
\hspace{0.005\textwidth}
\begin{minipage}[t]{0.68\textwidth}
\vspace{0pt}
\centering
\includegraphics[width=\linewidth]{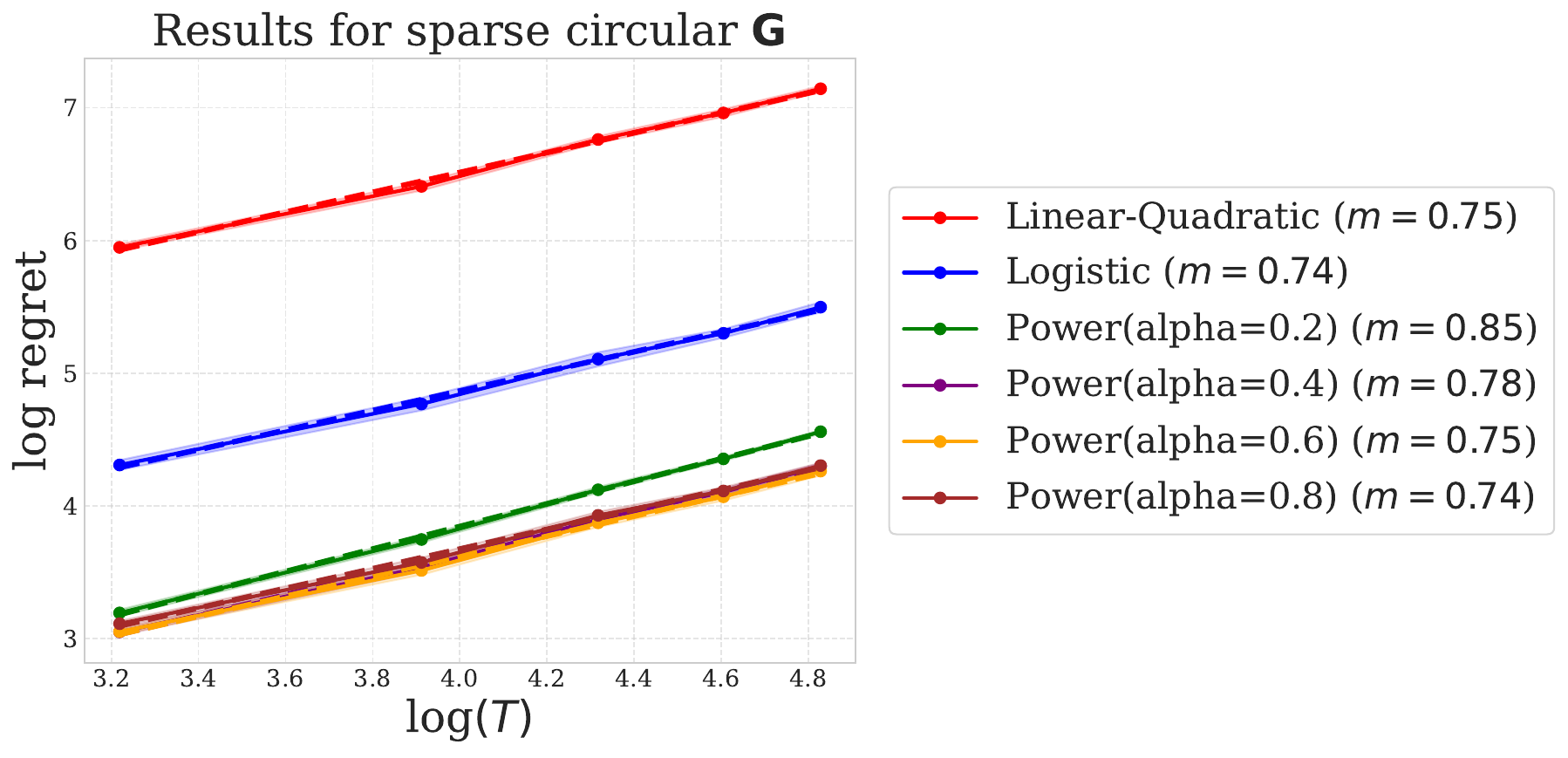}
\captionof{figure}{Average across the $N$ sellers of the cumulative regret as a function of the horizon $T\in\mathcal T$, displayed on a log--log scale. For each value of $T$, we run $10$ independent repetitions and report $95\%$ confidence intervals. The empirical slopes $m$'s are consistent with the theoretical regret rates predicted by our Theorem~\ref{thm:dynamic_regret_isotonic_network} (same or lower rate), which are respectively (following the legend from top to bottom): $0.75$, NA, $0.87$, $0.81$, $0.78$, and $0.76$.}
\label{fig:tau_2}
\end{minipage}

\textbf{2. $\mathbf{G}$ with an influencer.} We also consider a network with the presence of an influencer, indexed by $i=0$. This node is characterized by a null $i$-th row of $\mathbf G$ and large positive entries in the $i$-th column, so that it affects the rest of the network while remaining unaffected by others. Figure~\ref{fig:influencer_prices} displays the resulting optimal price distribution. Consistently with the theory in Section~\ref{sec:influencers}, the influencer receives a substantially lower optimal price than the rest of the network, in some cases close to zero.

\begin{figure}[h]
\centering
\includegraphics[width=1\linewidth]{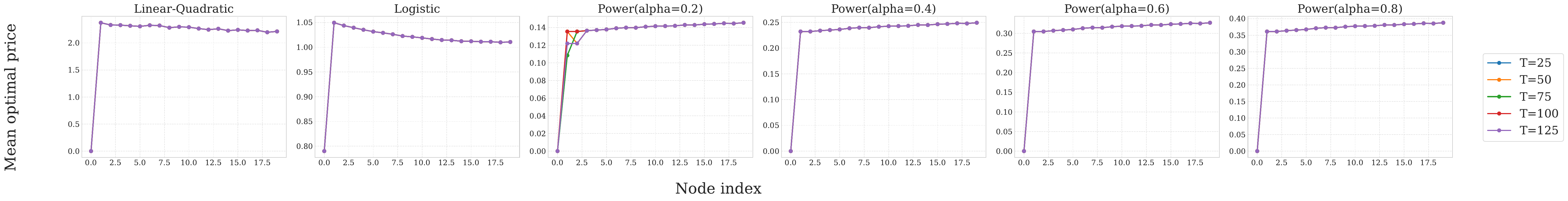}
\caption{Average optimal prices across nodes in the influencer network. The influencer, indexed by $i=0$, is characterized by a null $i$-th row of $\mathbf G$ and large positive entries in the $i$-th column.}
\label{fig:influencer_prices}
\end{figure}




\section{Future directions: addressing limitations}\label{sec:future_directions}

Several directions remain open. First, in line with Remark~\ref{rem:low-dim-N}, it would be important to study regimes in which the number of consumers $N$ is of the same order as the horizon $T$, or even larger. In such settings, additional structure (such as sparsity, clustering, or low-rank parameterizations of the functions $\psi_i$) will be necessary to obtain meaningful statistical and computational guarantees. In particular, identifying latent clustering structure appears especially promising, since it could reduce the pricing problem to a smaller number of effective communities, as discussed in Remark \ref{rem:segmentation} and also suggested by the decomposition results in Section~\ref{sec:influencers}. A second direction concerns the learning problem itself. Our regret analysis relies on H\"older regularity of the unknown marginal response functions, whereas certain economically relevant models such as the discrete-choice specification fall outside this class because of boundary singularities. Extending the theory to such settings would considerably broaden the scope of the algorithm. Finally, an ambitious theoretical challenge is to study seller-side optimality without imposing the interiority condition that the equilibrium consumption lies in $\mathring{\mathcal X}$. This assumption is convenient because it enables a clean first-order characterization of the seller problem, but removing it would require a more general treatment, likely involving variational-inequality methods directly at the seller level.

\newpage

\section{Broader Impact}

This paper studies monopoly pricing over consumer networks with nonlinear utilities, with a focus on equilibrium analysis, price discrimination, and online learning. On the positive side, the framework may improve the modeling of markets in which consumers interact through peer effects, congestion, or adoption spillovers, thereby enabling pricing decisions that better reflect the underlying economic structure. More broadly, the methodological contributions on equilibrium characterization and shape-constrained learning may be useful beyond pricing, for example in networked decision systems and structural demand estimation.

At the same time, the tools developed here could be used to support increasingly personalized pricing strategies. While such strategies may improve efficiency or targeting, they also raise concerns about fairness, transparency, and the potential exploitation of vulnerable or highly influential consumers. In particular, models that identify influential nodes in a network may be used to selectively discount or surcharge specific individuals in ways that are difficult to detect or regulate. Moreover, if deployed with real consumer data, errors in network estimation or utility learning could amplify existing biases and lead to systematically unequal treatment across groups.

For these reasons, we view the present work primarily as a theoretical and methodological contribution. Any real-world deployment should be accompanied by appropriate safeguards, including transparency about pricing policies, robustness checks, and, where relevant, fairness or regulatory constraints. We hope this work contributes to a more principled understanding of network-based pricing while also encouraging careful consideration of its societal implications.

\section{Details of LLM usage}
We used generative AI tools when preparing the manuscript to polish our sentences and correct potential typos; we remain responsible for all opinions, findings, and conclusions or recommendations expressed in the paper.

\bibliographystyle{apalike}
\bibliography{main}


\newpage
\section*{\centering \Large Appendix of \\
Equilibrium and Pricing in Consumer Networks with
Nonlinear Utilities: An Online Shape-Constrained
Learning Approach
}
\appendix

\startcontents[appendix]

\section*{Appendix Contents}
\printcontents[appendix]{}{1}{}

\section*{Appendix Roadmap}

For the reader’s convenience, we briefly summarize the content of each Appendix.

\begin{itemize}[leftmargin=1.5em,itemsep=0.25em,topsep=0.25em]
    \item \textbf{Appendix A: Uniqueness of NE via Contraction and via Variational Inequalities.}
    Details the two analytical routes establishing Theorem~\ref{thm:uniqueness_NE}: the Banach contraction and variational inequality arguments. It explores the role of the weighted sup-norm, provides counterexamples showing the two conditions are generally incomparable, and derives practical sufficient conditions via Gershgorin-type bounds.

    \item \textbf{Appendix B: Missing Proofs.}
    Collects all mathematical proofs omitted from the main text. This includes Theorem~\ref{thm:comparison}, the invertibility properties of the Jacobian $\mathbf{J}(\mathbf{x})$, Proposition~\ref{prop:v_i-equality}, the influencer-pricing results, and the dynamic regret bounds of Theorem~\ref{thm:dynamic_regret_isotonic_network}.

    \item \textbf{Appendix C: Optimal Price for Symmetric and Non-Negative Networks.}
    Contrasts benchmark network structures under linear--quadratic and nonlinear settings, characterizing precisely when the optimal monopoly price is invariant to the underlying network topology and when this property breaks down.
    
    \item \textbf{Appendix D: Network Communities and their Pricing Policies.}
    Analyzes block-diagonal network structures, demonstrating how both the consumer equilibrium and the seller’s pricing optimization perfectly decompose across independent network communities.

    \item \textbf{Appendix E: Supplementary Details on Network Influencers.}
    Expands on the structural definition of network influencers introduced in the main text. It examines the Jacobian's invertibility, interprets intrinsic influential value across specific models, and elaborates on the pricing implications for highly connected consumers.

    \item \textbf{Appendix F: Examples of Utility Functions Covered by our Model.}
    Catalogues concrete utility specifications accommodated by our general framework, including linear--quadratic, Stone--Geary, discrete-choice, exponential, and isoelastic models.

    \item \textbf{Appendix G: Examples of Monopolistic Euclidean Network Topologies.}
    Constructs explicit examples of follower and influencer network structures in the Euclidean setting, detailing their implied optimal pricing strategies.

    \item \textbf{Appendix H: Finite-Sample Convergence of the Isotonic Estimates.}
    Develops the finite-sample uniform convergence guarantees for isotonic regression with continuous responses. These concentration bounds form the technical foundation for the main regret analysis.

    \item \textbf{Appendix I: Computational Complexity.}
    Breaks down the computational cost of Algorithm~\ref{alg:isotonic_plugin_pricing}, detailing the overhead for response construction, isotonic estimation, and the final plug-in pricing optimization step.

    \item \textbf{Appendix J: Experiments.}
    Provides comprehensive details on the empirical evaluation, including hyperparameter selections, network generation, oracle baselines, and supplementary numerical experiments that complement the main text.
\end{itemize}

\newpage

\section{Uniqueness of NE via Contraction and via Variational Inequalities}\label{appendix:uniqueness}

\subsection{Uniqueness of Nash equilibrium via contraction}\label{appendix:uniqueness-contration}
To prove the uniqueness via the Banach contraction theorem, we first prove that the best-response map is Lipschitz, then impose a contraction condition. 

\begin{lemma}[Lipschitz continuity of the regularized best response]
\label{lem:BR_Lipschitz_general_monopolist}
Let Assumption~\ref{ass:strong-convexity} hold. For each $z\in\mathbb R$, define
\[
\mathcal T_i(z)\triangleq \arg\max_{x_i\in\mathcal X_i}\{z x_i-h_i(x_i)\}.
\]
Then $\mathcal T_i(z)$ is single-valued, and for all $z,\tilde z\in\mathbb R$,
\[
\textstyle \bigl|\mathcal T_i(z)-\mathcal T_i(\tilde z)\bigr|
\le \frac{1}{\mu_i}|z-\tilde z|.
\]
\end{lemma}

\begin{proof}[Proof of Lemma~\ref{lem:BR_Lipschitz_general_monopolist}]
Let
\[
x=\mathcal T_i(z),
\qquad
\tilde x=\mathcal T_i(\tilde z).
\]
Since $x$ maximizes $u\mapsto z \cdot u-h_i(u)$ over $\mathcal{X}_i$, the first-order optimality condition gives
\[
(z-g)\cdot (u-x) \le 0
\qquad
\forall u\in\mathcal{X}_i,
\]
for some $g\in \partial h_i(x)$. Choosing $u=\tilde x$, we obtain
\[
(z-g) \cdot (\tilde x-x ) \le 0.
\]
Similarly, since $\tilde x$ maximizes $u\mapsto \tilde z \cdot u-h_i(u)$, there exists $\tilde g\in\partial h_i(\tilde x)$ such that
\[
(\tilde z-\tilde g) \cdot (x-\tilde x ) \le 0.
\]
Adding the two inequalities yields
\[
( z-\tilde z) \cdot (x-\tilde x )
\ge
(g-\tilde g) \cdot (x-\tilde x).
\]
By strong convexity, the subdifferential of $h_i$ is strongly monotone:
\[
(g-\tilde g) \cdot (x-\tilde x ) \ge \mu_i |x-\tilde x |^2.
\]
Therefore
\[
(z-\tilde z ) \cdot (x-\tilde x )
\ge
\mu_i |x-\tilde x |^2.
\]
Using Hölder's inequality,
\[
|z-\tilde z | \cdot  |x-\tilde x |
\ge
(z-\tilde z) \cdot (x-\tilde x),
\]
hence
\[
|z-\tilde z | \cdot |x-\tilde x |
\ge
\mu_i |x-\tilde x |^2.
\]
If $x\neq \tilde x$, dividing by $ |x-\tilde x |$ gives
\[
|x-\tilde x | \le \frac{1}{\mu_i} |z-\tilde z |.
\]
The same inequality is trivial if $x=\tilde x$. This proves the claim.
\end{proof}

\begin{theorem}[Uniqueness of Nash equilibrium via contraction]
\label{thm:general_competitive_NE_uniqueness_monopolist}
Assume the hypotheses of Theorem~\ref{thm:general_competitive_NE_existence_monopolist} and let Assumption~\ref{ass:strong-convexity} hold. Let
\[
\mathbf{M}\triangleq \delta\,\mathbf{D}^{-1}\,\mathbf{|G|},
\]
where $|\mathbf{G}|=(|g_{ij}|)_{i,j\in\mathcal{N}}$ denotes the entrywise absolute value of $\mathbf{G}$ and $\mathbf{D}=  \operatorname{diag}(\mu_1,\mu_2,\dots,\mu_N)$. If
\[
\rho(\mathbf{M})<1,
\]
then, for every $\mathbf p \in \mathcal{P}$, the best-response map $\boldsymbol{\Gamma}(\cdot;\mathbf p):\mathcal X\to\mathcal X$ is a contraction (with respect to some weighted sup norm on $\mathbb R^N$, uniformly in $\mathbf p$). Consequently, the game admits a unique Nash equilibrium $\mathbf x^\star(\mathbf p)$.
\end{theorem}

\begin{proof}
By applying Lemma~\ref{lem:BR_Lipschitz_general_monopolist}, with $\Gamma_i(\mathbf{x}_{-i};p_i)=\mathcal T_i\big(s_i(\mathbf{x}_{-i};p_i)\big)$, for every $i\in\mathcal{N}$ we have
\[
\textstyle |\Gamma_i(\mathbf{x}_{-i};p_i)-\Gamma_i(\tilde{\mathbf{x}}_{-i};p_i)|
\le
\frac{1}{\mu_i}\,
|s_i(\mathbf{x}_{-i};p_i)-s_i(\tilde{\mathbf{x}}_{-i};p_i)|.
\]
Since
\[
\textstyle s_i(\mathbf{x}_{-i};p_i)-s_i(\tilde{\mathbf{x}}_{-i};p_i)
=
\delta \sum_{j=1}^N g_{ij}(x_j-\tilde x_j),
\]
it follows that
\[
\textstyle |\Gamma_i(\mathbf{x}_{-i};p_i)-\Gamma_i(\tilde{\mathbf{x}}_{-i};p_i)|
\le
\frac{\delta}{\mu_i}\sum_{j=1}^N |g_{ij}|\,|x_j-\tilde x_j|.
\]
Define the nonnegative matrix
\[
\mathbf{M}\triangleq \delta\,\mathbf{D}^{-1}\,|\mathbf{G}|=\delta\,\mathbf{D}^{-1}\,|\mathbf{G}|,
\]
where $|\mathbf{G}|=(|g_{ij}|)_{i,j\in\mathcal{N}}$ denotes the entrywise absolute value of $\mathbf{G}$ and $\mathbf{D}=\operatorname{diag}(\mu_1,\mu_2,\dots,\mu_N)$. Then the previous estimate can be written component-wise as
\begin{equation}\label{ineq:M_matrix}
|\boldsymbol{\Gamma}(\mathbf{x};\mathbf p)-\boldsymbol{\Gamma}(\tilde{\mathbf{x}};\mathbf p)|
\le
\mathbf{M}\,|\mathbf{x}-\tilde{\mathbf{x}}|,
\end{equation}
where $|\cdot|$ applied to a vector is also understood componentwise.

Now assume that $\rho(\mathbf{M})<1$. Fix any $\alpha\in(\rho(\mathbf{M}),1)$. Then
$\rho(\alpha^{-1}\mathbf M)=\alpha^{-1}\rho(\mathbf M)<1$, so
$\alpha \mathbf I_N-\mathbf M$ is invertible and
\[
\textstyle (\alpha \mathbf I_N-\mathbf M)^{-1}
=
\alpha^{-1}(\mathbf I_N-\alpha^{-1}\mathbf M)^{-1}
=
\sum_{k=0}^\infty \alpha^{-(k+1)}\mathbf M^k .
\]
Since $\mathbf M\ge 0$, each term in the series is componentwise nonnegative, hence
\[
\textstyle (\alpha \mathbf I_N-\mathbf M)^{-1}
=
\alpha^{-1}\mathbf I_N+\sum_{k=1}^\infty \alpha^{-(k+1)}\mathbf M^k
\ge \alpha^{-1}\mathbf I_N.
\]
Define
\[
\mathbf w\triangleq (\alpha \mathbf I_N-\mathbf M)^{-1}\mathbf 1.
\]
Then $\mathbf w\ge \alpha^{-1}\mathbf 1>0$, so $\mathbf w\in(0,\infty)^N$. Moreover, $(\alpha \mathbf I_N-\mathbf M)\mathbf w=\mathbf 1$, and therefore $\mathbf M\mathbf w=\alpha \mathbf w-\mathbf 1$. Since $\mathbf w\ge \alpha^{-1}\mathbf 1$, we have $\mathbf M\mathbf w=\alpha \mathbf w-\mathbf 1\ge 0$, and also $\mathbf M\mathbf w=\alpha \mathbf w-\mathbf 1<\alpha \mathbf w$. Hence $\mathbf 0 \le \mathbf M\mathbf w\le \alpha \mathbf w$. For this $\mathbf w$, define the weighted sup norm
\[
\textstyle \|\mathbf x\|_{\mathbf{w},\infty}:=\max_{i\in\mathcal{N}}\frac{|x_i|}{w_i},
\]
and the induced matrix norm
\[
\textstyle \|\mathbf M\|_{\mathbf{w},\infty}
:=
\max_{i\in\mathcal{N}}\frac{1}{w_i}\sum_{j=1}^N M_{ij}w_j
=
\max_{i\in\mathcal{N}}\frac{[\mathbf M\mathbf w]_i}{w_i}
\le
\max_{i\in\mathcal{N}}\frac{\alpha w_i}{w_i}
=
\alpha<1.
\]
Hence, from \eqref{ineq:M_matrix},
\[
|\boldsymbol{\Gamma}(\mathbf x;\mathbf p)-\boldsymbol{\Gamma}(\tilde{\mathbf x};\mathbf p)|_i
\le
(\mathbf M|\mathbf x-\tilde{\mathbf x}|)_i,
\qquad \forall i\in\mathcal{N}.
\]
Dividing by $w_i>0$, we get for every $i\in\mathcal{N}$,
\[
\textstyle \frac{|\Gamma_i(\mathbf x;\mathbf p)-\Gamma_i(\tilde{\mathbf x};\mathbf p)|}{w_i}
\le
\frac{1}{w_i}\sum_{j=1}^N M_{ij}|x_j-\tilde x_j|.
\]
Now, by definition of $\|\cdot\|_{\mathbf{w},\infty}$,
\[
|x_j-\tilde x_j|
\le
\|\mathbf x-\tilde{\mathbf x}\|_{\mathbf{w},\infty}\, w_j,
\qquad \forall j\in\mathcal{N}.
\]
Substituting this bound into the previous display yields
\[
\textstyle \frac{|\Gamma_i(\mathbf x;\mathbf p)-\Gamma_i(\tilde{\mathbf x};\mathbf p)|}{w_i}
\le
\left(\frac{1}{w_i}\sum_{j=1}^N M_{ij} w_j\right)
\|\mathbf x-\tilde{\mathbf x}\|_{\mathbf{w},\infty}.
\]
Taking the maximum over $i\in\mathcal{N}$, we obtain
\[
\textstyle \|\boldsymbol{\Gamma}(\mathbf x;\mathbf p)-\boldsymbol{\Gamma}(\tilde{\mathbf x};\mathbf p)\|_{\mathbf{w},\infty}
\le
\left(\max_{i\in\mathcal{N}}\frac{1}{w_i}\sum_{j=1}^N M_{ij}w_j\right)
\|\mathbf x-\tilde{\mathbf x}\|_{\mathbf{w},\infty}.
\]
By the definition of the induced weighted sup norm,
\[
\textstyle \max_{i\in\mathcal{N}}\frac{1}{w_i}\sum_{j=1}^N M_{ij}w_j
=
\|\mathbf M\|_{\mathbf{w},\infty},
\]
and therefore
\[
\|\boldsymbol{\Gamma}(\mathbf x;\mathbf p)-\boldsymbol{\Gamma}(\tilde{\mathbf x};\mathbf p)\|_{\mathbf{w},\infty}
\le
\|\mathbf M\|_{\mathbf{w},\infty}\,
\|\mathbf x-\tilde{\mathbf x}\|_{\mathbf{w},\infty}.
\]
Therefore,
\[
\|\boldsymbol{\Gamma}(\mathbf x;\mathbf p)-\boldsymbol{\Gamma}(\tilde{\mathbf x};\mathbf p)\|_{\mathbf{w},\infty}
\le
\alpha\|\mathbf x-\tilde{\mathbf x}\|_{\mathbf{w},\infty},
\]
so $\boldsymbol{\Gamma}(\cdot;\mathbf p)$ is a contraction.Thus $\boldsymbol{\Gamma}(\cdot;\mathbf p)$ is a contraction on the complete metric space $(\mathcal X,\|\cdot\|_{\mathbf{w},\infty})$, with contraction constant independent of $\mathbf p$. By Banach's fixed-point theorem, for every $\mathbf p \in \mathcal{P}$, $\boldsymbol{\Gamma}(\cdot;\mathbf p)$ admits a unique fixed point $\mathbf{x}^\star(\mathbf p)\in\mathcal X$. Since fixed points of $\boldsymbol{\Gamma}(\cdot;\mathbf p)$ are exactly profiles of mutual best responses, $\mathbf{x}^\star(\mathbf p)$ is the unique Nash equilibrium.
\end{proof}

\begin{remark}[Why the weighted sup norm is important]
\label{rem:weighted_norm_importance_monopolist}
The contraction condition $\rho(\mathbf{M})<1$ is the intrinsic condition in our setting, in the sense that it does not depend on the choice of a norm. Indeed, \textit{\textbf{the spectral radius is the quantity that lies below every sub-multiplicative matrix norm}}, in the sense that
\[
\rho(\mathbf{M})\leq \|\mathbf{M}\|
\]
for every sub-multiplicative matrix norm $\|\cdot\|$. To see this, fix any norm $\|\cdot \|$ such that $\| \mathbf{M} \mathbf{v}\|\leq \| \mathbf{M}\| \| \mathbf{v}\|$ for all $\mathbf{M} \in \mathbb{R}^{N \times N}$ and $\mathbf{v} \in \mathbb{R}^N$ (i.e. $\| \cdot \|$ is sub-multiplicative) and let $\lambda \in \sigma (\mathbf{M})$ be any eigenvalue of $\mathbf{M}$ with eigenvector $\mathbf{v}$, then
$$
|\lambda|\|\mathbf{v}\|=\|\lambda \mathbf{v}\|=\|\mathbf{M} \mathbf{v}\| \leq\|\mathbf{M}\|\|\mathbf{v}\|,
$$
so $|\lambda| \leq \| \mathbf{M} \|$, and taking the maximum over all the eigenvalues we get 
$$
\rho ( \mathbf{M} ) \triangleq \max \{ |\lambda|: \lambda \in \sigma( \mathbf{M})\} \leq \| \mathbf{M} \|.
$$
Thus, if one fixes a norm from the beginning and applies Banach's fixed-point theorem directly, one is naturally led to a sufficient condition of the form
\[
\|\mathbf{M}\|<1,
\]
which may be strictly stronger than $\rho(\mathbf{M})<1$. The difficulty is that Banach's fixed-point theorem is formulated in terms of a norm, whereas the spectral radius itself does not define a norm. One is therefore tempted to choose a standard norm, such as $\|\cdot\|_\infty$, $\|\cdot\|_1$, or $\|\cdot\|_2$, and prove contraction under the corresponding stronger requirement
\[
\|\mathbf{M}\|_\infty<1,\qquad
\|\mathbf{M}\|_1<1,\qquad
\text{or}\qquad
\|\mathbf{M}\|_2<1.
\]
However, these are only sufficient conditions, and can be substantially more restrictive than the contraction condition, as we show in the following examples. For instance, fix $\mathbf{D} = \mathbf{I}_N$ as in the Linear-Quadratic Utility Model in Example~\ref{example_complete_Linear}, and consider
\[
\mathbf G=
\begin{pmatrix}
0 & V\\
0 & 0
\end{pmatrix},
\qquad V>0.
\]
Then $\rho(|\mathbf G|)=\rho(\mathbf G)=0$, so the contraction condition $\delta \rho(|\mathbf G|)<1$ is automatically satisfied for every $\delta>0$. 
However, $\|\mathbf G\|_1=V$, which implies that a norm-based condition such as $\delta\|\mathbf G\|_1<1$ holds for $\delta< 1/V$, which is more restrictive than necessary. As a symmetric example, consider
\[
\mathbf G=
\begin{pmatrix}
0 & 1 & 1\\
1 & 0 & 0\\
1 & 0 & 0
\end{pmatrix}.
\]
Then $\mathbf G$ is symmetric, so $\rho(|\mathbf G|)=\rho(\mathbf G)=\lambda_{\max}(\mathbf G)=\sqrt{2}$, whereas $\|\mathbf G\|_1=2$. Thus, the norm-based condition $\delta<1/2$ is strictly stronger than the sharp contraction condition (or variational condition, since they coincide when $\mathbf{G}$ is symmetric and entry-wise positive), i.e., $\delta<\frac{1}{\sqrt{2}}$. This illustrates again why the weighted norm construction is essential to obtain the least conservative contraction criterion.

The role of the weighted sup norm $\|\cdot\|_{\mathbf{w},\infty}$ is precisely to provide the correct bridge between the spectral-radius assumption and Banach's theorem. Starting from $\rho(\mathbf{M})<1$, one can choose $\alpha\in(\rho(\mathbf{M}),1)$ and construct a vector $\mathbf{w}\in(0,\infty)^N$ such that
\[
\|\mathbf{M}\|_{\mathbf{w},\infty}\leq \alpha<1.
\]
In this way, the contraction argument is carried out in a normed space, as required by Banach's theorem, but without replacing the intrinsic condition $\rho(\mathbf{M})<1$ by a stronger norm bound. In this sense, the construction of $\|\cdot\|_{\mathbf{w},\infty}$ allows one to recover a condition that is uniform over all possible choices of norm, rather than tied to a particular one.
\end{remark}

\subsection{Uniqueness of Nash equilibrium via variational inequality}\label{appendix:uniqueness-variational}

We now provide a uniqueness result that characterizes the Nash equilibrium as the solution of a variational inequality and prove uniqueness by strict monotonicity of the corresponding operator, under Assumption~\ref{ass:strong-convexity}, i.e., strong convexity of $h_i$ for all $i\in \mathcal{N}$.

For every $\mathbf{x} \in \mathcal{X}$ and $\mathbf{p} \in \mathcal{P}$, define
\[
\mathbf F(\mathbf x;\mathbf p)
\triangleq
\boldsymbol{\phi}(\mathbf x)-\delta \mathbf G\mathbf x-(\mathbf a-\mathbf B\mathbf p),
\]
and recall from Remark~\ref{rem:BR_fixed_point_characterization} that, under strong convexity of $h_i$ for all $i\in \mathcal{N}$, that for every $\mathbf{p} \in \mathcal{P}$, a NE can be found by solving $F(\mathbf x;\mathbf p)= \mathbf{0}$ with respect to $\mathbf{x}$, and projecting the solution into $\mathcal{X}$.

\begin{remark}[Variational characterization of Nash equilibrium]
\label{rem:VI_characterization_NE}
Fix $\mathbf p \in \mathcal{P}$. Since $h_i$ is strictly convex, the function
\[
x_i\mapsto s_i(\mathbf x_{-i};p_i)x_i-h_i(x_i)
\]
is strictly concave on $\mathcal X_i$, and therefore player $i$'s best response is unique. Moreover, following the definition~\ref{def:NE} of NE via variational inequality, for every fixed $\mathbf{p}$, a point $\mathbf x^\star= \mathbf x^{\star}(\mathbf{p})\in\mathcal X$ is a Nash equilibrium if and only if, for every $i\in\mathcal{N}$,
\[
\bigl(\phi_i(x_i^\star)-a_i+b_ip_i-\delta(\mathbf G\mathbf x^\star)_i\bigr)(y_i-x_i^\star)\ge 0,
\qquad \forall y_i\in\mathcal X_i.
\]
Stacking these conditions across $i$, we see that $\mathbf x^\star$ is a Nash equilibrium if and only if
\begin{equation}\label{eq:VI_NE}
\langle \mathbf F(\mathbf x^\star;\mathbf p),\mathbf y-\mathbf x^\star\rangle \ge 0,
\qquad \forall \mathbf y\in\mathcal X.
\end{equation}
Thus, a Nash equilibrium is exactly a solution of the variational inequality \eqref{eq:VI_NE}.

\end{remark}

The next theorem gives a sufficient condition for uniqueness.

\begin{theorem}[Uniqueness of Nash equilibrium via strict monotonicity]
\label{thm:general_competitive_NE_uniqueness_VI}
Assume the hypotheses of Theorem~\ref{thm:general_competitive_NE_existence_monopolist} and let Assumption~\ref{ass:strong-convexity} hold. Define $\mathbf D\triangleq \operatorname{diag}(\mu_1,\dots, \mu_N)$. If
\begin{equation}\label{eq:strict_monotonicity_condition}
\textstyle \mathbf z^\top
\left(
\mathbf D-\delta\,\frac{\mathbf G+\mathbf G^\top}{2}
\right)
\mathbf z>0,
\qquad \forall \mathbf z\neq 0,
\end{equation}
then, for every $\mathbf p \in \mathcal{P}$, the game admits a unique Nash equilibrium.
\end{theorem}

\begin{proof}
Fix $\mathbf p \in \mathcal{P}$, and let $\mathbf x,\mathbf y\in\mathcal X$. Then
\[
\mathbf F(\mathbf x;\mathbf p)-\mathbf F(\mathbf y;\mathbf p)
=
\boldsymbol{\phi}(\mathbf x)-\boldsymbol{\phi}(\mathbf y)-\delta \mathbf G(\mathbf x-\mathbf y).
\]
Therefore,
\begin{align*}
\textstyle \langle \mathbf F(\mathbf x;\mathbf p)-\mathbf F(\mathbf y;\mathbf p),\mathbf x-\mathbf y\rangle
=
\sum_{i=1}^N
\bigl(\phi_i(x_i)-\phi_i(y_i)\bigr)(x_i-y_i)
-\delta (\mathbf x-\mathbf y)^\top \mathbf G(\mathbf x-\mathbf y).
\end{align*}
Since
\[
\textstyle \mathbf z^\top \mathbf G \mathbf z
=
\mathbf z^\top \frac{\mathbf G+\mathbf G^\top}{2}\mathbf z,
\qquad \forall \mathbf z\in\mathbb R^N,
\]
and by Assumption~\ref{ass:strong-convexity},
\[
\bigl(\phi_i(x_i)-\phi_i(y_i)\bigr)(x_i-y_i)\ge \mu_i|x_i-y_i|^2,
\]
we obtain
\begin{align*}
\textstyle \langle \mathbf F(\mathbf x;\mathbf p)-\mathbf F(\mathbf y;\mathbf p),\mathbf x-\mathbf y\rangle
\ge
(\mathbf x-\mathbf y)^\top
\left(
\mathbf D-\delta\,\frac{\mathbf G+\mathbf G^\top}{2}
\right)
(\mathbf x-\mathbf y).
\end{align*}
Hence, by \eqref{eq:strict_monotonicity_condition},
\[
\langle \mathbf F(\mathbf x;\mathbf p)-\mathbf F(\mathbf y;\mathbf p),\mathbf x-\mathbf y\rangle>0,
\qquad \forall \mathbf x\neq \mathbf y,
\]
so $\mathbf F(\cdot;\mathbf p)$ is strictly monotone on $\mathcal X$. Now let $\mathbf x^\star,\tilde{\mathbf x}^\star\in\mathcal X$ be two solutions of $\mathrm{VI}(\mathcal X,\mathbf F(\cdot;\mathbf p))$. Then
\[
\langle \mathbf F(\mathbf x^\star;\mathbf p),\tilde{\mathbf x}^\star-\mathbf x^\star\rangle\ge 0,
\qquad
\langle \mathbf F(\tilde{\mathbf x}^\star;\mathbf p),\mathbf x^\star-\tilde{\mathbf x}^\star\rangle\ge 0.
\]
Adding the two inequalities yields
\[
\langle \mathbf F(\mathbf x^\star;\mathbf p)-\mathbf F(\tilde{\mathbf x}^\star;\mathbf p),\mathbf x^\star-\tilde{\mathbf x}^\star\rangle\le 0.
\]
By strict monotonicity, this is possible only if $\mathbf x^\star=\tilde{\mathbf x}^\star$. Therefore, the variational inequality admits at most one solution. Since existence was established in Theorem~\ref{thm:general_competitive_NE_existence_monopolist}, the Nash equilibrium exists and is unique.
\end{proof}

\subsection{Contraction and variational conditions coincide for $\mathbf G$ symmetric and non-negative}
\label{subsec:contraction-variational-coincide}

Using the linear--quadratic and discrete-choice models, the following examples illustrate that both uniqueness criteria coincide whenever the network is symmetric and entry-wise non-negative.

\begin{example}[The Linear-Quadratic Utility Model.]\label{example_complete_Linear}
We continue from the Linear-Quadratic Utility Model in Example~\ref{example_1}. In this case, all consumers share the same regularization function $h_i(x)=h(x)=\frac{1}{2}x^2$, which is $1$-strongly convex on $\mathbb{R}$, therefore, the monotonicity constant is $\mu_i=1$ for every $i$, and $\mathbf D=\mathbf I_N$. Consequently $\phi_i = (\nabla h_i)^{-1}=\mathrm{Id}$, so, conditional on $\mathbf p \in \mathcal{P}$, the equilibrium condition in \eqref{FOC:consumer-global} solves $\mathbf a-\mathbf B\mathbf p+\delta \mathbf G\mathbf x-\mathbf x=\mathbf 0$. Therefore, whenever $\mathbf I_N-\delta \mathbf G$ is invertible, the consumer Nash equilibrium satisfies $\mathbf x^*(\mathbf p)
=
\Pi_{\mathcal X} ((\mathbf I_N-\delta \mathbf G)^{-1}(\mathbf a-\mathbf B\mathbf p))$. Since $h$ is $1$-strongly convex, Theorem~\ref{thm:uniqueness_NE} implies that, for every $\mathbf p \in \mathcal{P}$, the consumer game admits a unique Nash equilibrium whenever $\delta\rho(\,\mathbf{|G|})<1$ or $\delta\,\lambda_{\max} (\tfrac{\mathbf G+\mathbf G^\top}{2})<1$. In particular, if $\mathbf G$ is symmetric and with non-negative entries, $\mathbf{|G|}= \tfrac{\mathbf{G}+\mathbf{G^{\top}}}{2}= \mathbf{G}$, then $\rho(\mathbf{G}) = \lambda_{\max}(\mathbf{G})$ and two conditions reduce to $\delta\,\lambda_{\max}(\mathbf G)<1$.
\end{example}

\begin{example}[Discrete Choice Model.]\label{example_complete_Discrete}
We continue from the Discrete Choice Model in Example~\ref{example_1}. Similar to Example~\ref{example_complete_Linear}, using that $ h\left(x\right)=x \ln x+\left(1-x\right) \ln \left(1-x\right)$, $x \in [0,1]$ satisfies $h^{\prime \prime}(x)=\frac{1}{x(1-x)}\geq 4$, it is $4$-strongly convex in $[0,1]$, and then $\mathbf D=4\mathbf I_N$. Hence the contraction condition becomes $\delta\rho(\mathbf{|G|})<4$ and the variational condition $\delta\,\lambda_{\max}(\frac{\mathbf G+\mathbf G^\top}{2})<4$. Similarly to the Example~\ref{example_complete_Linear}, we have that, if $\mathbf G$ is symmetric and with non-negative entries, then the two conditions reduce to $\delta\,\lambda_{\max}(\mathbf G)<4$.
\end{example}

\subsection{More examples showing that the contraction and variational conditions are not comparable}\label{sec:contactio_variational_examples}

\textbf{Example 1: contraction condition weaker than variational condition.} Fix $\mathbf D=\mathbf I$. Consider first

\begin{figure}[h]
\centering
\begin{minipage}[c]{0.38\textwidth}
\centering
\[
\mathbf G=
\begin{pmatrix}
0 & V & V & V\\
0 & 0 & 0 & 0\\
0 & 0 & 0 & 0\\
0 & 0 & 0 & 0
\end{pmatrix},
\qquad V>0
\]
\end{minipage}
\hfill
\begin{minipage}[c]{0.48\textwidth}
\centering
\begin{tikzpicture}[
  >=Latex,
  node/.style={circle, draw=black, thick, fill=black!8, font=\scriptsize, inner sep=1.2pt},
  edge/.style={->, draw=black!60, line width=0.8pt},
  elabel/.style={font=\tiny, fill=white, inner sep=0.8pt}
]
\node[node, minimum size=9mm]   (1) at (0,0.2) {$1$};
\node[node, minimum size=6.5mm] (2) at (-1.7,1.1) {$2$};
\node[node, minimum size=6.5mm] (3) at (-1.8,-0.9) {$3$};
\node[node, minimum size=6.5mm] (4) at (1.8,0.2) {$4$};

\draw[edge] (2) -- node[elabel, above left] {$V$} (1);
\draw[edge] (3) -- node[elabel, below left] {$V$} (1);
\draw[edge] (4) -- node[elabel, above] {$V$} (1);
\end{tikzpicture}
\end{minipage}
\caption{Network associated with the matrix $\mathbf G$.}
\label{fig:network_one_influencer}
\end{figure}
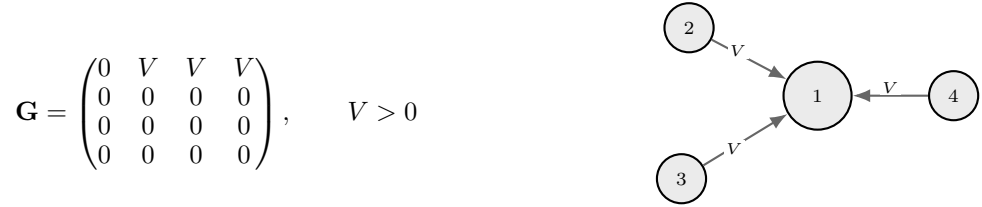

which network is represented in Figure~\ref{fig:network_one_influencer}. Since $\mathbf G$ is nilpotent, $\rho(|\mathbf G|)=0$, so the contraction condition is automatically satisfied for every $\delta>0$. However,
\[
\frac{\mathbf G+\mathbf G^\top}{2}
=
\begin{pmatrix}
0 & V/2 & V/2 & V/2\\
V/2 & 0 & 0 & 0\\
V/2 & 0 & 0 & 0\\
V/2 & 0 & 0 & 0
\end{pmatrix},
\]
whose largest eigenvalue is $\frac{\sqrt{3}}{2}V$. Hence, the variational condition requires
\[
\delta<\frac{2}{\sqrt{3}\,V},
\]
and is therefore strictly stronger.

\textbf{Example 2: variational condition weaker than contraction condition.} Conversely, consider

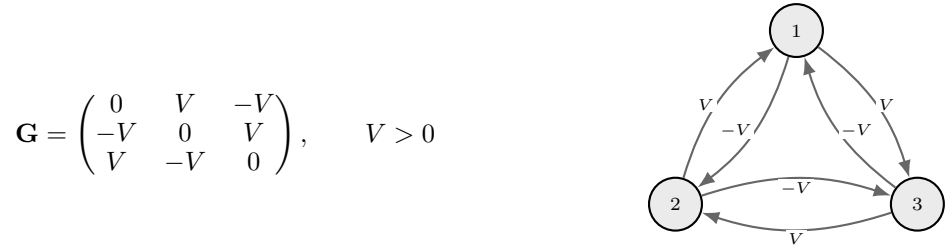
\begin{figure}[ht]
\centering
\begin{minipage}[c]{0.34\textwidth}
\centering
\[
\mathbf G=
\begin{pmatrix}
0 & V & -V\\
-V & 0 & V\\
V & -V & 0
\end{pmatrix},
\qquad V>0
\]
\end{minipage}
\hfill
\begin{minipage}[c]{0.52\textwidth}
\centering
\begin{tikzpicture}[
  >=Latex,
  node/.style={circle, draw=black, thick, fill=black!8, font=\scriptsize, inner sep=1.2pt},
  edge/.style={->, draw=black!60, line width=0.8pt},
  elabel/.style={font=\tiny, fill=white, inner sep=0.8pt}
]
\node[node, minimum size=7mm] (1) at (0,1.5) {$1$};
\node[node, minimum size=7mm] (2) at (-1.6,-0.8) {$2$};
\node[node, minimum size=7mm] (3) at (1.6,-0.8) {$3$};

\draw[edge, bend left=18] (2) to node[elabel, left] {$V$} (1);
\draw[edge, bend left=18] (1) to node[elabel, left] {$-V$} (2);

\draw[edge, bend left=18] (3) to node[elabel, below] {$V$} (2);
\draw[edge, bend left=18] (2) to node[elabel, below] {$-V$} (3);

\draw[edge, bend left=18] (1) to node[elabel, right] {$V$} (3);
\draw[edge, bend left=18] (3) to node[elabel, right] {$-V$} (1);
\end{tikzpicture}
\end{minipage}
\caption{Signed skew-symmetric triangle associated with the matrix $\mathbf G$.}
\label{fig:network_circle}
\end{figure}

which network is represented in Figure~\ref{fig:network_circle}. This matrix is skew-symmetric, so
\[
\frac{\mathbf G+\mathbf G^\top}{2}=0,
\]
and the variational condition is automatically satisfied for every $\delta>0$. On the other hand,
\[
|\mathbf G|=
\begin{pmatrix}
0 & V & V\\
V & 0 & V\\
V & V & 0
\end{pmatrix},
\]
whose spectral radius is $2V$. Therefore the contraction condition requires
\[
\delta<\frac{1}{2V},
\]
and is thus strictly stronger.

These examples show that, in the absence of symmetry and sign restrictions on $\mathbf G$, neither condition dominates the other in general.

\begin{remark}
The contraction condition and the varitional condition emphasize different aspects of the network. The contraction condition controls the magnitude of all spillovers, regardless of their sign, and is therefore particularly well suited to networks with weak feedback loops, such as triangular or nilpotent structures. By contrast, the variational condition depends only on the symmetric part of $G$, and is therefore insensitive to purely directional or antisymmetric interactions. As a result, it is especially favorable in networks where influence is largely one-way or sign-changing, since those features disappear after symmetrization. 
\end{remark}

\subsection{Sufficient Conditions via Gershgorin disc theorem.}\label{remark:Gershgorin}

The contraction condition and the variational conditions can also be interpreted through the Gershgorin disc theorem. Recall that, for any square matrix $\mathbf M=(m_{ij})$, every eigenvalue of $\mathbf M$ lies in at least one disc
\[
\textstyle \left\{z\in\mathbb C:\ |z-m_{ii}|\le \sum_{j\neq i}|m_{ij}|\right\}.
\]
Therefore, Gershgorin's theorem allows one to replace global spectral conditions with simpler row-wise dominance conditions. In the present setting, this yields a transparent customer-level interpretation of both uniqueness assumptions.

For the contraction condition, let
\[
\mathbf M:=\delta \mathbf D^{-1}|\mathbf G|.
\]
Since $g_{ii}=0$, the diagonal entries of $\mathbf M$ are zero, and the $i$-th Gershgorin disc is centered at the origin with radius
\[
\textstyle  \sum_{j\neq i}\frac{\delta |g_{ij}|}{\mu_i}.
\]
Hence a sufficient condition for all eigenvalues of $\mathbf M$ to lie strictly inside the unit disc, and therefore for $\rho(\mathbf M)<1$, is
\[
\textstyle  \frac{\delta}{\mu_i}\sum_{j\neq i}|g_{ij}|<1,
\qquad \forall i\in\mathcal N,
\]
or equivalently
\[
\textstyle  \mu_i>\delta\sum_{j\neq i}|g_{ij}|,
\qquad \forall i\in\mathcal N.
\]
This shows that the contraction condition holds whenever, for each consumer $i$, the stabilizing force generated by diminishing returns exceeds the total absolute influence exerted by that consumer's neighbors. In this sense, the best-response effect of the network is locally damped at each node, so that a shock to one consumer's demand cannot be amplified indefinitely through repeated interactions.

For the variational condition, define
\[
\textstyle  \mathbf H:=\mathbf D-\delta\frac{\mathbf G+\mathbf G^\top}{2}.
\]
Because $\mathbf H$ is symmetric, a convenient sufficient condition for positive definiteness is strict diagonal dominance with positive diagonal entries. By Gershgorin's theorem, this is ensured if
\[
\textstyle  \mu_i
>
\delta \sum_{j\neq i}\left|\frac{g_{ij}+g_{ji}}{2}\right|,
\qquad \forall i\in\mathcal N.
\]
Thus, for every consumer $i$, the individual curvature $\mu_i$ must dominate the total magnitude of the \emph{symmetric} interaction with the rest of the network. Economically, this means that the mutually reinforcing part of the local network effects around consumer $i$ is not strong enough to offset the stabilizing effect of diminishing returns.

Therefore, Gershgorin's theorem gives a simple node-level interpretation of both conditions: uniqueness is guaranteed if, at every consumer, the stabilizing force coming from the curvature of utility is stronger than the aggregate social influence coming from the network. The difference between the two conditions is that the contraction condition uses the full absolute row sum $\sum_j |g_{ij}|$, while the variational condition uses only the symmetric component $\sum_j |(g_{ij}+g_{ji})/2|$. Hence the contraction condition is generally more conservative, whereas the variational condition ignores purely directional asymmetries that do not contribute to mutual reinforcement.
\newpage

\section{Missing Proofs}\label{app:missing-proof}

\subsection{Proof of Theorem~\ref{thm:comparison}}\label{app:missing-proof-comparison}
\begin{proof}[Proof of Theorem~\ref{thm:comparison}]
Since $\mathbf G$ is symmetric and entry-wise non-negative, we have
\[
|\mathbf G|=\mathbf G,
\qquad
\frac{\mathbf G+\mathbf G^\top}{2}=\mathbf G.
\]
Therefore, the contraction condition becomes $\rho(\delta \mathbf D^{-1}\mathbf G)<1$, which is equivalent to $\delta\,\rho(\mathbf D^{-1}\mathbf G)<1$. On the other hand, the variational condition becomes $\mathbf D-\delta \mathbf G\succ \mathbf{O}_N$. Since $\mathbf D$ is diagonal with positive entries, this is equivalent to $\mathbf I-\delta \mathbf D^{-1/2}\mathbf G\mathbf D^{-1/2}\succ \mathbf{O}_N$. Because $\mathbf D^{-1/2}\mathbf G\mathbf D^{-1/2}$ is symmetric, the latter holds if and only if $\lambda_{\max} \left(\delta \mathbf D^{-1/2}\mathbf G\mathbf D^{-1/2}\right)<1$. Since $\mathbf D^{-1/2}\mathbf G\mathbf D^{-1/2}$ is a real value symmetric matrix, the eigenvalues are real, and its largest eigenvalue equals its spectral radius, that is $\lambda_{\max} \left(\mathbf D^{-1/2}\mathbf G\mathbf D^{-1/2}\right) = \rho \left(\mathbf D^{-1/2}\mathbf G\mathbf D^{-1/2}\right)$. Because $\mathbf D^{-1/2}\mathbf G\mathbf D^{-1/2}$ is similar to $\mathbf D^{-1}\mathbf G$ (with orthogonal matrix $\mathbf D^{-1/2}$), we obtain $\rho\left(\mathbf D^{-1/2}\mathbf G\mathbf D^{-1/2}\right)
=
\rho(\mathbf D^{-1}\mathbf G)$. Hence $\lambda_{\max} \left(\mathbf D^{-1/2}\mathbf G\mathbf D^{-1/2}\right)
=
\rho(\mathbf D^{-1}\mathbf G)$, and therefore the variational condition is also equivalent to $\delta\,\rho(\mathbf D^{-1}\mathbf G)<1$. Thus, in the symmetric and entry-wise non-negative case, the contraction condition and the variational condition coincide with the condition $\delta \,\rho(\mathbf D^{-1}\mathbf G)<1$.
\end{proof}

\newpage

\subsection{On the invertibility of $\mathbf J(\mathbf x)$}\label{sec:invertibility_of_J}

\subsubsection{On the invertibility of $\mathbf J(\mathbf x)$ via variational condition}\label{sec:invertibility_of_J-variational}

\begin{proposition}
\label{prop:invertibility_jacobian_phi}
Let
\[
\boldsymbol{\phi}(\mathbf x)
=
\bigl(\phi_1(x_1),\dots,\phi_N(x_N)\bigr)^\top,
\qquad
D_{\boldsymbol{\phi}}(\mathbf x)
:=
\operatorname{diag}\bigl(\phi_1'(x_1),\dots,\phi_N'(x_N)\bigr),
\]
and fix $\mathbf x\in\mathbb R^N$. If
\begin{equation}\label{eq:local_variational_condition}
D_{\boldsymbol{\phi}}(\mathbf x)-\delta\,\frac{\mathbf G+\mathbf G^\top}{2}\succ \mathbf{O}_N,
\end{equation}
then the matrix
\[
D_{\boldsymbol{\phi}}(\mathbf x)-\delta \mathbf G
\]
is invertible.
\end{proposition}

\begin{proof}
Define
\[
\mathbf J(\mathbf x):=D_{\boldsymbol{\phi}}(\mathbf x)-\delta \mathbf G.
\]
We prove that $\mathbf J(\mathbf x)$ has trivial kernel. Suppose by contradiction that there exists $\mathbf v\neq \mathbf 0$ such that
\[
\mathbf J(\mathbf x)\mathbf v=\mathbf 0.
\]
Multiplying on the left by $\mathbf v^\top$, we get
\[
0
=
\mathbf v^\top \mathbf J(\mathbf x)\mathbf v
=
\mathbf v^\top D_{\boldsymbol{\phi}}(\mathbf x)\mathbf v
-
\delta\,\mathbf v^\top \mathbf G \mathbf v.
\]
Now observe that
\[
\mathbf v^\top \mathbf G \mathbf v
=
\mathbf v^\top \frac{\mathbf G+\mathbf G^\top}{2}\mathbf v,
\]
because the anti-symmetric part of $\mathbf G$ vanishes in the quadratic form. Therefore,
\[
0
=
\mathbf v^\top
\left(
D_{\boldsymbol{\phi}}(\mathbf x)-\delta\,\frac{\mathbf G+\mathbf G^\top}{2}
\right)\mathbf v.
\]
This contradicts \eqref{eq:local_variational_condition}, since the latter implies
\[
\mathbf v^\top
\left(
D_{\boldsymbol{\phi}}(\mathbf x)-\delta\,\frac{\mathbf G+\mathbf G^\top}{2}
\right)\mathbf v
>0
\qquad \forall \mathbf v\neq \mathbf 0.
\]
Hence $\ker(\mathbf J(\mathbf x))=\{\mathbf 0\}$, and since $\mathbf J(\mathbf x)$ is square, it is invertible.
\end{proof}

\begin{corollary}\label{cor:comparison_local_jacobian_variational}
If $\mathbf D-\delta\,\frac{\mathbf G+\mathbf G^\top}{2}\succ \mathbf{O}_N$ then $\boldsymbol{J}(\mathbf{x})=D_{\boldsymbol{\phi}}(\mathbf x)-\delta \mathbf G$ is invertible for every $\mathbf{x}$.
\end{corollary}
\begin{proof}
Proposition~\ref{prop:invertibility_jacobian_phi} gives a \emph{pointwise} sufficient condition for invertibility of the Jacobian
\[
D_{\boldsymbol{\phi}}(\mathbf x)-\delta \mathbf G
\]
at a fixed point $\mathbf x$. This condition is the local analogue of the variational condition in Theorem~\ref{thm:uniqueness_NE}. Indeed, if each $h_i$ is $\mu_i$-strongly convex and twice differentiable, then
\[
\phi_i'(x_i)=h_i''(x_i)\ge \mu_i,
\qquad \forall x_i\in\mathcal X_i,
\]
so that
\[
D_{\boldsymbol{\phi}}(\mathbf x)\succeq \mathbf D,
\qquad
\mathbf D:=\operatorname{diag}(\mu_1,\dots,\mu_N).
\]
Therefore, if the variational condition
\[
\mathbf D-\delta\,\frac{\mathbf G+\mathbf G^\top}{2}\succ \mathbf{O}_N
\]
holds, then for every $\mathbf x\in\mathcal X$,
\[
D_{\boldsymbol{\phi}}(\mathbf x)-\delta\,\frac{\mathbf G+\mathbf G^\top}{2}\succ \mathbf{O}_N.
\]
By Proposition~\ref{prop:invertibility_jacobian_phi}, this implies that
\[
D_{\boldsymbol{\phi}}(\mathbf x)-\delta \mathbf G
\]
is invertible for every $\mathbf x\in\mathcal X$.

Hence, the variational condition is a \emph{uniform} sufficient condition for invertibility of the Jacobian over the whole domain, whereas \eqref{eq:local_variational_condition} is only a \emph{local} sufficient condition at a given point $\mathbf x$. In particular, the pointwise condition in Proposition~\ref{prop:invertibility_jacobian_phi} is generally weaker, because it depends on the realized local slopes $\phi_i'(x_i)$, which may be strictly larger than the global lower bounds $\mu_i$.
\end{proof}



\subsubsection{On the invertibility and non-negativity of $\mathbf J(\mathbf x)$ via the contraction condition}\label{sec:invertibility_of_J-contraction}

\begin{proposition}\label{prop:comparison_local_jacobian_contraction}
If $\mathbf{G}=|\mathbf{G}|$ and the contraction condition $\rho(\delta \mathbf{G})<1$ holds, then the inverse of $\mathbf J(\mathbf x)
=
D_{\boldsymbol{\phi}}(\mathbf x)-\delta \mathbf G$ exists for every $\mathbf{x}$ and has non-negative entries. Moreover,
$$
\mathbf J(\mathbf x)^{-1}
=
\sum_{k=0}^\infty
\delta^k D_{\boldsymbol{\phi}}(\mathbf x)^{-k}\mathbf G^k
D_{\boldsymbol{\phi}}(\mathbf x)^{-1}.
$$
\end{proposition}

\begin{proof}
Since $\phi_i'(x_i)\ge \mu_i$ for all $i \in \mathcal{N}$, we have
\[
D_{\boldsymbol{\phi}}(\mathbf x)\succeq \mathbf D=\operatorname{diag}(\mu_1,\mu_2,\dots, \mu_N)
\qquad\text{diagonally},
\]
and therefore
\[
D_{\boldsymbol{\phi}}(\mathbf x)^{-1}\preceq \mathbf D^{-1}
\qquad\text{diagonally}.
\]
Because $\mathbf G\ge 0$ entrywise, it follows that
\begin{equation}\label{ineq:entri-wise-nonneg}
\mathbf O_N\le \delta D_{\boldsymbol{\phi}}(\mathbf x)^{-1}\mathbf G
\le \delta \mathbf D^{-1}\mathbf G
\qquad\text{entrywise}.
\end{equation}

\begin{lemma}[Monotonicity of the spectral radius for nonnegative matrices]\label{lemma:Monotonicity-spectral-radius}
Let $\mathbf A,\mathbf B\in\mathbb R^{N \times N}$ be entrywise nonnegative matrices. If $\mathbf O_N\le \mathbf A\le \mathbf B$
entrywise, then $\rho(\mathbf A)\le \rho(\mathbf B)$, where $\rho(\cdot)$ denotes the spectral radius.
\end{lemma}

\begin{proof}
Since $\mathbf O_N\le \mathbf A\le \mathbf B$ entry-wise and both matrices are nonnegative, we first show that
\[
\mathbf O_N \le \mathbf A^k\le \mathbf B^k
\qquad\text{entrywise for every } k\ge 1.
\]
The case $k=1$ is exactly the assumption. Suppose now that $0\le \mathbf A^k\le \mathbf B^k$. Since $\mathbf A\le \mathbf B$ and all matrices involved are entrywise nonnegative, multiplication preserves entrywise inequalities. Hence
\[
\mathbf A^{k+1}
=
\mathbf A^k \mathbf A
\le
\mathbf B^k \mathbf B
=
\mathbf B^{k+1}.
\]
Therefore, by induction, $\mathbf{O}_N\le \mathbf A^k\le \mathbf B^k$ for all $k\ge 1$. Now consider the matrix norm $\|\mathbf M\|_{\infty} := \max_{1\le i\le n}\sum_{j=1}^n |M_{ij}|$. Since $\mathbf O_N\le \mathbf A^k\le \mathbf B^k$ entrywise, we have $\|\mathbf A^k\|_{\infty}
\le
\|\mathbf B^k\|_{\infty}$ for every $k\ge 1$. Taking $k$-th roots gives
\[
\|\mathbf A^k\|_{\infty}^{1/k}
\le
\|\mathbf B^k\|_{\infty}^{1/k}.
\]
By Gelfand's formula,
\[
\rho(\mathbf A)
=
\lim_{k\to\infty}\|\mathbf A^k\|_{\infty}^{1/k},
\qquad
\rho(\mathbf B)
=
\lim_{k\to\infty}\|\mathbf B^k\|_{\infty}^{1/k}.
\]
Passing to the limit in the previous inequality yields $\rho(\mathbf A)\le \rho(\mathbf B)$.
\end{proof}

By \eqref{ineq:entri-wise-nonneg} and Lemma~\ref{lemma:Monotonicity-spectral-radius}, we have
\[
\rho \left(\delta D_{\boldsymbol{\phi}}(\mathbf x)^{-1}\mathbf G\right)
\le
\rho \left(\delta \mathbf D^{-1}\mathbf G\right)
<1.
\]
Therefore, $\left(\mathbf I_N-\delta D_{\boldsymbol{\phi}}(\mathbf x)^{-1}\mathbf G\right)^{-1}$ exists and can be written as the Neumann series:
\[
\left(\mathbf I_N-\delta D_{\boldsymbol{\phi}}(\mathbf x)^{-1}\mathbf G\right)^{-1}
=
\sum_{k=0}^\infty
\left(\delta D_{\boldsymbol{\phi}}(\mathbf x)^{-1}\mathbf G\right)^k,
\]
and every term in the series is entrywise nonnegative. Since
\[
\mathbf J(\mathbf x)
=
D_{\boldsymbol{\phi}}(\mathbf x)-\delta \mathbf G
=
D_{\boldsymbol{\phi}}(\mathbf x)(\mathbf I_N-\delta D_{\boldsymbol{\phi}}(\mathbf x)^{-1}\mathbf G),
\]
we obtain that $\mathbf J(\mathbf x)^{-1}$ exists and
\[
\mathbf J(\mathbf x)^{-1}
=
\Bigl(\mathbf I_N-\delta D_{\boldsymbol{\phi}}(\mathbf x)^{-1}\mathbf G\Bigr)^{-1}
D_{\boldsymbol{\phi}}(\mathbf x)^{-1}\ge \mathbf O_N 
\]
entrywise. Moreover 
\[
\mathbf J(\mathbf x)^{-1}
=
\sum_{k=0}^\infty
\delta^k D_{\boldsymbol{\phi}}(\mathbf x)^{-k}\mathbf G^k
D_{\boldsymbol{\phi}}(\mathbf x)^{-1}
\]
\end{proof}

\newpage

\subsection{Proof of Proposition~\ref{prop:v_i-equality}}\label{app:proof-prop-v_i-equality}

\begin{proof}[Proof of Proposition~\ref{prop:v_i-equality}]
By Assumption~\ref{ass:strong-convexity}, $\phi_i'(x_i)\ge \mu_i$ for all $i$, hence
\[
D_{\boldsymbol\phi}(\mathbf x)\succeq \mathbf D
\qquad\text{and}\qquad
D_{\boldsymbol\phi}(\mathbf x)^{-1}\preceq \mathbf D^{-1}.
\]
Therefore
\[
\mathbf O_N\le \delta D_{\boldsymbol\phi}(\mathbf x)^{-1}\mathbf G
\le \delta \mathbf D^{-1}\mathbf G
\qquad\text{entrywise}.
\]
Since $\rho(\delta \mathbf D^{-1}\mathbf G)<1$, the Neumann series is valid and monotone for nonnegative matrices, so
\[
\bigl(\mathbf I-\delta D_{\boldsymbol\phi}(\mathbf x)^{-1}\mathbf G\bigr)^{-1}
\preceq
\bigl(\mathbf I-\delta \mathbf D^{-1}\mathbf G\bigr)^{-1}.
\]
Using
\[
\mathbf J(\mathbf x)^{-1}
=
\bigl(\mathbf I-\delta D_{\boldsymbol\phi}(\mathbf x)^{-1}\mathbf G\bigr)^{-1}
D_{\boldsymbol\phi}(\mathbf x)^{-1},
\]
we obtain
\[
\mathbf J(\mathbf x)^{-1}
\preceq
\bigl(\mathbf I-\delta \mathbf D^{-1}\mathbf G\bigr)^{-1}\mathbf D^{-1}
=
(\mathbf D-\delta \mathbf G)^{-1}.
\]
Multiplying on the left by $\mathbf 1^\top$ and taking the $i$-th coordinate gives
\[
\bigl[\mathbf 1^\top \mathbf J(\mathbf x)^{-1}\bigr]_i
\le
\bigl[\mathbf 1^\top(\mathbf D-\delta \mathbf G)^{-1}\bigr]_i,
\qquad \forall \mathbf x\in\mathcal X.
\]
Taking the supremum over $\mathbf x\in\mathcal X$ proves
\[
V_i\le \bigl[\mathbf 1^\top(\mathbf D-\delta \mathbf G)^{-1}\bigr]_i.
\]

If there exists $\mathbf x^\dagger\in\mathcal X$ such that $D_{\boldsymbol\phi}(\mathbf x^\dagger)=\mathbf D$, then
\[
\mathbf J(\mathbf x^\dagger)=\mathbf D-\delta \mathbf G,
\]
hence
\[
\bigl[\mathbf 1^\top \mathbf J(\mathbf x^\dagger)^{-1}\bigr]_i
=
\bigl[\mathbf 1^\top(\mathbf D-\delta \mathbf G)^{-1}\bigr]_i.
\]
Combining this with the upper bound yields equality.
\end{proof}
\newpage

\subsection{Proof of Proposition \ref{prop:general_pure_influencer}}\label{app:proof-prop-general-pure-infl}

\begin{proof}[Proof of Proposition \ref{prop:general_pure_influencer}]
Since $\mathbf x^\star(\mathbf p)\in\mathring{\mathcal X}$, it satisfies
\[
\boldsymbol{\phi}(\mathbf x^\star(\mathbf p))-\delta \mathbf G\mathbf x^\star(\mathbf p)
=
\mathbf a-\mathbf B\mathbf p.
\]
\emph{Step 1.} Since the $i$-th row of $\mathbf G$ is zero, the $i$-th equilibrium equation becomes
\[
\phi_i(x_i^\star(\mathbf p))=a_i-b_i p_i.
\]
Because $\phi_i$ is strictly increasing, it is invertible, and therefore
\[
x_i^\star(\mathbf p)=\phi_i^{-1}(a_i-b_i p_i).
\]

\medskip
\noindent
\emph{Step 2.}
Differentiating
\[
\boldsymbol{\phi}(\mathbf x^\star(\mathbf p))-\delta \mathbf G\mathbf x^\star(\mathbf p)
=
\mathbf a-\mathbf B\mathbf p.
\]
with respect to $\mathbf p$, we obtain
\[
\bigl(D_{\boldsymbol{\phi}}(\mathbf x^\star(\mathbf p))-\delta \mathbf G\bigr)\,
D_{\mathbf p}\mathbf x^\star(\mathbf p)
=
-\mathbf B,
\]
where $\mathbf{J}(\mathbf{x})^{-1}$ exists for every $\mathbf{x} \in \mathring{\mathcal{X}}$ (see Appendix~\ref{sec:invertibility_of_J}). Hence
\begin{equation}\label{eq:implicit_dxdp_general}
D_{\mathbf p}\mathbf x^\star(\mathbf p)
=
-\mathbf J(\mathbf x^\star(\mathbf p))^{-1}\mathbf B.
\end{equation}
By the product rule and \eqref{eq:implicit_dxdp_general},
\[
\nabla r(\mathbf p)
=
\mathbf x^\star(\mathbf p)
+
\bigl(D_{\mathbf p}\mathbf x^\star(\mathbf p)\bigr)^\top \mathbf p
=
\mathbf x^\star(\mathbf p)
-
\mathbf B\,\mathbf J(\mathbf x^\star(\mathbf p))^{-\top}\mathbf p.
\]
Thus, the $i$-th first-order condition at the interior maximizer $\mathbf p^\star$ is
\[
0
=
x_i^\star(\mathbf p^\star)
-
b_i\bigl(\mathbf J(\mathbf x^\star(\mathbf p^\star))^{-\top}\mathbf p^\star\bigr)_i.
\]
Equivalently,
\[
x_i^\star(\mathbf p^\star)
=
b_i\sum_{j=1}^N
\bigl[\mathbf J(\mathbf x^\star(\mathbf p^\star))^{-1}\bigr]_{ji}p_j^\star.
\]

Using Step 1, we have $x_i^\star(\mathbf p^\star)=\phi_i^{-1}(a_i-b_i p_i^\star)$. Hence, we obtain
\[
\phi_i^{-1}(a_i-b_i p_i^\star)=b_i\sum_{j=1}^N
\bigl[\mathbf J(\mathbf x^\star(\mathbf p^\star))^{-1}\bigr]_{ji}p_j^\star.
\]
Applying $\phi_i$ to both sides gives
\[
a_i-b_i p_i^\star=\phi_i\left(b_i\sum_{j=1}^N
\bigl[\mathbf J(\mathbf x^\star(\mathbf p^\star))^{-1}\bigr]_{ji}p_j^\star\right),
\]
hence
\[
p_i^\star
=
\frac{a_i-\phi_i\left(b_i\sum_{j=1}^N
\bigl[\mathbf J(\mathbf x^\star(\mathbf p^\star))^{-1}\bigr]_{ji}p_j^\star\right)}{b_i}.
\]
which has to be strictly positive because we are assuming that $\mathbf{p}^{\star} \in \mathring{\mathcal{P}}$.

\medskip
\noindent
\emph{Step 3.} Wlof let assume that $i=1$. Let $\widetilde{\mathcal{P}}=\{\mathbf{p} \in \mathcal{P}:p_1=0\}$, and let $\widetilde{\mathbf p}$ be the value in $\widetilde{\mathcal{P}}$ that maximizes $r$. Since $\mathbf p^\star \in \mathcal{P}$ is an interior local maximizer and is twice differentiable around $\mathbf{p}^{\star}$, the Taylor expansion around $\mathbf p^\star$ gives
\[
r(\mathbf p)
=
r(\mathbf p^\star)
+\frac12(\mathbf p-\mathbf p^\star)^\top \mathbf Q(\mathbf p-\mathbf p^\star)
+
o(\|\mathbf p-\mathbf p^\star\|_2^2),
\qquad \mathbf Q:=\nabla^2 r(\mathbf p^\star)\prec 0,
\]
so for $\mathbf{p}=\widetilde{\mathbf p}$ we have
\[
r(\mathbf p^\star)-r(\widetilde{\mathbf p})
=
-\frac12(\widetilde{\mathbf p}-\mathbf p^\star)^\top \mathbf Q(\widetilde{\mathbf p}-\mathbf p^\star)
+
o(\|\widetilde{\mathbf p}-\mathbf p^\star\|_2^2),
\]
Now $\widetilde{\mathbf p}$ is constrained by
$$
\mathbf e_1^{\top} \widetilde{\mathbf p}=0 \quad \Longrightarrow \quad \mathbf e_1^{\top}\left(\widetilde{\mathbf p}-\mathbf p^{\star}\right)=-p_1^*,
$$

where $\mathbf e_1=(1,0, \ldots, 0)^{\top}$.
To second order, $\widetilde{\mathbf p}$ is obtained by maximizing the quadratic approximation of $r$, i.e. by solving

$$
\max _{\Delta \in \mathcal{P}} \frac{1}{2} \Delta^{\top} \mathbf Q \Delta \quad \text { s.t. } \quad \mathbf e_1^{\top} \Delta=-p_1^*,
$$

with $\Delta=\widetilde{\mathbf p}-\mathbf{p}^{\star}$.
which concludes the proof. Since $\mathbf Q \prec 0$, the solution is $\Delta=-\frac{p_1^*}{\mathbf e_1^{\top} \mathbf Q^{-1} \mathbf e_1} \mathbf Q^{-1} \mathbf e_1 \quad$ to leading order. Plugging this into the quadratic term,
$$
\Delta^{\top} \mathbf Q \Delta=\frac{\left(p_1^*\right)^2}{\mathbf e_1^{\top} \mathbf Q^{-1} \mathbf e_1} .
$$
\[
r(\mathbf p^\star)-r(\widetilde{\mathbf p})
=
\underset{=:\theta_i>0}{\underbrace{-\frac12 \frac{1}{\mathbf e_1^{\top} \mathbf Q^{-1} \mathbf e_1}}}\left(p_1^*\right)^2
+
o((p_1^{\star})^2),.
\]
This concludes the proof.
\end{proof}

\newpage

\subsection{Proof of Corollary \ref{cor:structural_influencer_zero_price}}\label{proof-cor-structural-influencer}

\begin{proof}[Proof of Corollary \ref{cor:structural_influencer_zero_price}]
First, note that by Proposition \ref{prop:comparison_local_jacobian_contraction}, $\mathbf{J}(\mathbf{x})^{-1}$ exists for every $\mathbf{x}$ and has non-negative entries. Now, from Proposition~\ref{prop:general_pure_influencer},
\[
p_i^\star
=
\frac{a_i-\phi_i(\alpha_i^\star p_i^\star+\beta_i^\star)}{b_i},
\]
where
\[
\alpha_i^\star p_i^\star+\beta_i^\star
=
b_i\sum_{j=1}^N
\bigl[\mathbf J(\mathbf x^\star(\mathbf p^\star))^{-1}\bigr]_{ji}p_j^\star.
\]
Since $p_j^\star\le \overline p$ for all $j$,
\[
\alpha_i^\star p_i^\star+\beta_i^\star
\le
\overline p\, b_i
\sum_{j=1}^N
\bigl[\mathbf J(\mathbf x^\star(\mathbf p^\star))^{-1}\bigr]_{ji}
\le
\overline p\, b_i\, V_i.
\]
Hence, if
\[
\overline p\, b_i\, V_i<\phi_i^{-1}(a_i),
\]
then, since $\phi_i$ is increasing
\[
\phi(\alpha_i^\star p_i^\star+\beta_i^\star)\leq \phi(\bar{p}b_iV_i)<a_i,
\]
Therefore
\[
p_i^{\star}=\frac{a_i-\phi_i(\alpha_i^\star p_i^\star+\beta_i^\star)}{b_i}\geq \frac{a_i-\phi_i(\bar{p}b_iV_i)}{b_i}>0.
\]
This completes the proof.
\end{proof}

\newpage

\subsection{Strong-monotonicity of $\mathbf{F}$}\label{app:proof-monotonicity-F}

We prove that if Assumption~\ref{ass:strong-convexity} and Assumption~\ref{ass:strong-concavity-estim} hold, then the map $\mathbf F(\mathbf x):=\boldsymbol{\psi}(\mathbf x)-\delta \mathbf G\mathbf x$ is $m$-strongly monotone.
\begin{proof}
For any $\mathbf x,\mathbf y\in\mathcal X$, we have 
\[
\mathbf F(\mathbf x)-\mathbf F(\mathbf y)
=
\boldsymbol{\phi}(\mathbf x)-\boldsymbol{\phi}(\mathbf y)-\delta \mathbf G(\mathbf x-\mathbf y).
\]
Therefore,
\begin{align*}
\textstyle \langle \mathbf F(\mathbf x)-\mathbf F(\mathbf y),\mathbf x-\mathbf y\rangle
=
\sum_{i=1}^N
\bigl(\phi_i(x_i)-\phi_i(y_i)\bigr)(x_i-y_i)
-\delta (\mathbf x-\mathbf y)^\top \mathbf G(\mathbf x-\mathbf y).
\end{align*}
Since
\[
\textstyle \mathbf z^\top \mathbf G \mathbf z
=
\mathbf z^\top \frac{\mathbf G+\mathbf G^\top}{2}\mathbf z,
\qquad \forall \mathbf z\in\mathbb R^N,
\]
and by Assumption~\ref{ass:strong-convexity},
\[
\bigl(\phi_i(x_i)-\phi_i(y_i)\bigr)(x_i-y_i)\ge \mu_i|x_i-y_i|^2,
\]
we obtain
\begin{align*}
\textstyle \langle \mathbf F(\mathbf x;\mathbf p)-\mathbf F(\mathbf y;\mathbf p),\mathbf x-\mathbf y\rangle
\ge
(\mathbf x-\mathbf y)^\top
\left(
\mathbf D-\delta\,\frac{\mathbf G+\mathbf G^\top}{2}
\right)
(\mathbf x-\mathbf y),
\end{align*}
and by Assumption~\ref{ass:strong-concavity-estim}
\begin{align*}
\textstyle \langle \mathbf F(\mathbf x;\mathbf p)-\mathbf F(\mathbf y;\mathbf p),\mathbf x-\mathbf y\rangle
\ge
(\mathbf x-\mathbf y)^\top
m \mathbf{I}_N
(\mathbf x-\mathbf y) = m \|\mathbf x-\mathbf y\|_2.
\end{align*}
This completes the proof.
\end{proof}

\newpage

\subsection{Proof of Theorem~\ref{thm:dynamic_regret_isotonic_network}}\label{app:thm:dynamic_regret}

\begin{proof}[Proof of Theorem~\ref{thm:dynamic_regret_isotonic_network}]
Let $E = \{1,2,\dots, n\}$ be the exploration phase and $E' = \{n+1,n+2,\dots, T\}$ the exploitation phase. Let $n:=T_0=\lceil T^\beta\rceil$, where $\beta\in(0,1)$ will be chosen later.

\textbf{Step 1: Exploration regret.} Define 
\[
R_{\max}:=\sup_{\mathbf p\in\mathcal P,\mathbf x\in\mathcal X} |\mathbf p^\top \mathbf x|<\infty,
\]
which exists because $\mathcal{X}\times \mathcal{P}$ is bounded. Since $|\mathbf p^\top \mathbf x|\le R_{\max}$ for all $(\mathbf p,\mathbf x)\in\mathcal P\times\mathcal X$, the exploration contribution is bounded by
\[
\sum_{t\in E}
\bigl(
r(\mathbf p^\star)-r(\mathbf p^t)
\bigr)
\le
2R_{\max}\,n.
\]

\medskip
\noindent
\textbf{Step 2: Stability of the equilibrium map.}

Fix $t \in E'$. Recall that
$$
\mathbf p^t\in\arg\max_{\mathbf p\in\mathcal P} \, \{\widehat{r}(\mathbf p)=\mathbf p^\top \widehat{\mathbf x}^\star(\mathbf p)\}
,\qquad \mathbf p^{\star}\in\arg\max_{\mathbf p\in\mathcal P} \, \{r(\mathbf p)=\mathbf p^\top \mathbf x^\star(\mathbf p)\}.
$$

$$
\begin{aligned}
r(\mathbf p^{\star})-r(\mathbf p^t) &=\left\{r(\mathbf p^{\star})-\widehat{r}(\mathbf p^{\star}) \right\}+\underset{\leq 0}{\underbrace{\left\{\widehat{r}(\mathbf p^{\star})-\widehat{r}(\mathbf p^t)\right\}}}+ \left\{r(\mathbf p^t)-\widehat{r}(\mathbf p^t)\right\} \nonumber\\
&\leq 2\sup_{\mathbf p \in \mathcal{P}} |\widehat{r}(\mathbf p)-r(\mathbf p) | = 2\sup_{\mathbf p \in \mathcal{P}} |\mathbf{p}^{\top} (\widehat{\mathbf{x}}(\mathbf{p})-\mathbf{x}^{\star}(\mathbf{p}))|\nonumber\\
&\leq 2 \sup_{\mathbf p \in \mathcal{P}} \|\mathbf{p}\|_2\|\widehat{\mathbf{x}}(\mathbf{p})-\mathbf{x}^{\star}(\mathbf{p}))\|_2 \nonumber\\
&\leq 2P_{\max}\sup_{\mathbf p \in \mathcal{P}} \|\widehat{\mathbf{x}}(\mathbf{p})-\mathbf{x}^{\star}(\mathbf{p}))\|_2,
\end{aligned}
$$
where \[
P_{\max}:=\sup_{\mathbf p\in\mathcal P}\|\mathbf p\|_2<\infty.
\]

Now, by assumptions, for all $i\in \mathcal{N}$, $\psi_i$ is $\alpha_i$-H\"older continuous in $\mathcal{X}_i$ for some $\alpha_i \in (0,1]$. Hence, since $\mathcal{X_i}$ are compact, each $\psi_i$ is $\alpha = \min_i \alpha_i$ H\"older continuous in $\mathcal{X}_i$.

\begin{lemma}\label{lem:J2J3}
Let assumptions of Theorem~\ref{thm:dynamic_regret_isotonic_network} hold. Then for every $\kappa >0$ and $\gamma>2$ there exists $n_0=n_0(\gamma,\kappa,\alpha) \in \mathbb{N}$ and $C_i= C_i(L_i,|\mathcal{X}|,\kappa,\alpha)>0$ such that 
$$
\mathbb{P}\left\{\sup_{x_i \in \mathcal{X}_{i,n}}|\widehat{\psi}_i(x_i)-\psi_i(x_i))|\lesssim C_i \rho_n^{\nicefrac{\alpha}{2\alpha+1}}\right\}\geq 1-\tfrac{1}{n^{\gamma -2}},\quad n \geq n_0,
$$
where $\rho_n=\ln (n) / n$ and $\mathcal{X}_{i,n} = \{x_i \in \mathcal{X}_i:\left[x_i \pm \delta_n\right] \subset \mathcal{X}_i\}$, with $\delta_n=\kappa  \rho_n^{1 /(2\alpha +1)}$. Specifically, $n_0$ is the smallest integer $n$ that satisfies $\epsilon_n<1$, where
$$
\epsilon_n\triangleq\max \left(c_n / \delta_n, \sqrt{2 c_n / \delta_n}\right)+\left(n \delta_n\right)^{-1},\qquad c_n\triangleq \gamma \ln (n+2) /(n+1).
$$
\end{lemma}
The proof can be found in Section~\ref{proof:lemma-convergence}, and is an application of the more general Theorem~\ref{thm:Dumbgen}.

Now, define 
\[
\mathcal E_n:=
\left\{
\max_{i\in \mathcal{N}}
\sup_{x_i\in\mathcal X_{i,n}}
\bigl|
\widehat{\psi}_{i}(x_i)-\psi_i(x_i)
\bigr|
\le C_i \rho_n^{\nicefrac{\alpha}{2\alpha+1}}
\right\}=\left\{
\sup_{\mathbf{x}\in\mathcal X_{n}}
\|
\widehat{\boldsymbol{\psi}}(\mathbf{x})-\boldsymbol{\psi}(\mathbf{x})
\|_{\infty}
\lesssim \rho_n^{\nicefrac{\alpha}{2\alpha+1}}
\right\},
\]
where $\mathcal{X}_n = \prod_{i \in \mathcal{N}}\mathcal{X}_{i,n}$ and $\lesssim$ absorbs the maximum of the $C_i$ over $i\in \mathcal{N}$. Using Lemma~\ref{lem:J2J3} an union bound over $i \in \mathcal{N}$ we get that 
$$
\mathbb{P} (\mathcal{E}_n)\geq 1-\frac{N}{n^{\gamma-2}}, \qquad n \geq n_0.
$$

Now, on the event $\mathcal E_n$ fix any $\mathbf p\in\mathcal P$, and let
\[
\mathbf x^\star(\mathbf p)
\quad\text{and}\quad
\widehat{\mathbf x}^{\star}(\mathbf p)
\]
be the true and plug-in equilibria. Since $\mathbf F(\mathbf x^\star(\mathbf p))=\mathbf 0$ and
\[
\mathbf{F}(\widehat{\mathbf x}^{\star}(\mathbf p))=\widehat{\boldsymbol{\psi}}(\widehat{\mathbf x}^{\star}(\mathbf p))-\delta \mathbf G\widehat{\mathbf x}^{\star}(\mathbf p)
=
-\mathbf B\mathbf p,
\]
we have
\[
\mathbf F(\widehat{\mathbf x}^{\star}(\mathbf p))
= \boldsymbol{\psi}(\widehat{\mathbf x}^{\star}(\mathbf p))-\delta \mathbf G\widehat{\mathbf x}^{\star}(\mathbf p)+\mathbf B\mathbf p = 
\boldsymbol{\psi}(\widehat{\mathbf x}^{\star}(\mathbf p))
-
\widehat{\boldsymbol{\psi}}(\widehat{\mathbf x}^{\star}(\mathbf p)).
\]
The assumption 
    \[
\textstyle \mathbf D-\delta\,\frac{\mathbf G+\mathbf G^\top}{2}\succ m\mathbf{I}_N.
\]
for some $m>0$ implies that for every $\mathbf p\in\mathcal P$, the operator
    \[
    \mathbf F(\mathbf x):=
    \boldsymbol{\psi}(\mathbf x)-\delta \mathbf G\mathbf x+\mathbf B\mathbf p
    \]
    is $m$-strongly monotone on $\mathcal X$ (see Appendix~\ref{app:proof-monotonicity-F}), that is,
    \[
    \langle \mathbf F(\mathbf x)-\mathbf F(\mathbf y),\mathbf x-\mathbf y\rangle
    \ge m\|\mathbf x-\mathbf y\|_2^2,
    \qquad \forall \mathbf x,\mathbf y\in\mathcal X.
    \]
    
Therefore, by strong monotonicity of $\mathbf F$,
\begin{align*}
m\|\widehat{\mathbf x}^{\star}(\mathbf p)-\mathbf x^\star(\mathbf p)\|_2^2
&\le
\left\langle
\mathbf F(\widehat{\mathbf x}^{\star}(\mathbf p))
-
\mathbf F(\mathbf x^\star(\mathbf p)),
\widehat{\mathbf x}^{\star}(\mathbf p)-\mathbf x^\star(\mathbf p)
\right\rangle \\
&=
\left\langle
\boldsymbol{\psi}(\widehat{\mathbf x}^{\star}(\mathbf p))
-
\widehat{\boldsymbol{\psi}}(\widehat{\mathbf x}^{\star}(\mathbf p)),
\widehat{\mathbf x}^{\star}(\mathbf p)-\mathbf x^\star(\mathbf p)
\right\rangle \\
&\leq \|
\boldsymbol{\psi}(\widehat{\mathbf x}^{\star}(\mathbf p))
-
\widehat{\boldsymbol{\psi}}(\widehat{\mathbf x}^{\star}(\mathbf p))\|_{\infty}
\|\widehat{\mathbf x}^{\star}(\mathbf p)-\mathbf x^\star(\mathbf p)
\|_1 \\
& \leq \|
\boldsymbol{\psi}(\widehat{\mathbf x}^{\star}(\mathbf p))
-
\widehat{\boldsymbol{\psi}}(\widehat{\mathbf x}^{\star}(\mathbf p))\|_{\infty}\sqrt N
\|\widehat{\mathbf x}^{\star}(\mathbf p)-\mathbf x^\star(\mathbf p)\|_2.
\end{align*}

then, on $\mathcal{E}_n$ we have
\begin{equation}\label{eq:equilibrium_stability_bound}
\sup_{\mathbf p\in\mathcal P}
\|\widehat{\mathbf x}^{\star}(\mathbf p)-\mathbf x^\star(\mathbf p)\|_2
\le
\sqrt N\,\rho_n^{\nicefrac{\alpha}{2\alpha +1}}
\end{equation}
where in the second inequality we used that $\mathcal{P}\subset \mathbf{Q}(\mathring{\mathcal{X}})$, where $\mathbf{Q}=-\mathbf{B}^{-1} \circ \mathbf{F}$, hence for every $\mathbf{p} \in \mathcal{P}$, $\mathbf{x}^{\star}(\mathbf{p})$ is inside $\mathring{\mathcal{X}}$ and bounded away from the boundary of $\mathcal{X}$, hence for $n$ large enough 

$$
\sup_{\mathbf p \in \mathcal{P}}\|
\boldsymbol{\psi}(\widehat{\mathbf x}^{\star}(\mathbf p))
-
\widehat{\boldsymbol{\psi}}(\widehat{\mathbf x}^{\star}(\mathbf p))\|_{\infty} \leq \sup_{\mathbf x \in \mathring{\mathcal{X}}}\|
\boldsymbol{\psi}(\mathbf x)
-
\widehat{\boldsymbol{\psi}}(\mathbf x)\|_{\infty}  \lesssim \rho_n^{\nicefrac{\alpha}{2\alpha +1}}
$$

Hence, on $\mathcal E_n$, we have that for every $t \in E'$,
\begin{equation}\label{eq:conv_single_rev}
r(\mathbf p^{\star})-r(\mathbf p^t) \le
\frac{2P_{\max}\sqrt N}{m}\rho_n^{\nicefrac{\alpha}{2\alpha +1}}\cong \sqrt{N}\left(\frac{\ln n}{n}\right)^{\frac{\alpha}{2\alpha+1}} ,
\end{equation}
and the total exploitation contribution is bounded, on $\mathcal{E}_n$, by
\[
\sum_{t\in E'}
r(\mathbf p^\star)-r(\mathbf p^t)
\lesssim
|E'|\,\sqrt{N}\left(\frac{\ln n}{n}\right)^{\frac{\alpha}{2\alpha+1}}\leq T \sqrt{N}\left(\frac{\ln n}{n}\right)^{\frac{\alpha}{2\alpha+1}}
\]
\medskip
\noindent
\textbf{Step 3: Uniform error of the revenue function.}

Combining Steps 1 and 2, we get that, on $\mathcal{E}_n$, the total regret is
\[
\sum_{t\in [T]}
r(\mathbf p^\star)-r(\mathbf p^t)
\leq 2R_{\max}\,n
+T\left(\frac{\ln n}{n}\right)^{\frac{\alpha}{2\alpha+1}} \lesssim n+ \sqrt{N}T n^{-\frac{\alpha}{2\alpha+1}} \ln^{\frac{\alpha}{2\alpha+1}}(n)
\]

Now we choose $n = \lceil T^{\beta}\rceil$ for some $\beta\in (0,1)$. Then the previous display becomes
\[
\sum_{t\in [T]}
r(\mathbf p^\star)-r(\mathbf p^t)
\lesssim
T^\beta
+
\sqrt{N}\,T^{\,1-\beta\frac{\alpha}{2\alpha+1}}
\ln^{\frac{\alpha}{2\alpha+1}}(T).
\]
Balancing the two powers of $T$ gives
\[
\beta
=
1-\beta\frac{\alpha}{2\alpha+1},
\qquad\text{hence}\qquad
\beta=\frac{2\alpha+1}{3\alpha+1}.
\]
With this choice, and considering that our calculations are true on $\mathcal{E}_n$, where 
$$
\mathbb{P} (\mathcal{E}_n)\geq 1-\frac{N}{n^{\gamma-2}}, \qquad n \geq n_0.
$$
i.e., 
$$
\mathbb{P} (\mathcal{E}_n)\geq 1-\frac{N}{T^{\beta(\gamma-2)}}, \qquad T \geq \frac{\ln(n_0)}{\beta},\qquad n_0 \text{ defined in Lemma~\ref{lem:J2J3}},
$$
we conclude that, for $T$ sufficiently large ($T \geq \nicefrac{\ln(n_0)}{\beta}$), then
\[
R(T)=\sum_{t\in [T]}
r(\mathbf p^\star)-r(\mathbf p^t)
\lesssim \sqrt{N}\,
T^{\frac{2\alpha+1}{3\alpha+1}}
\ln^{\frac{\alpha}{2\alpha+1}}(T)
\]
with probability at least $1-\frac{N}{T^{\beta(\gamma-2)}}$, for any $\gamma>2$. This proves the first part of the Theorem. 

\medskip
\noindent
\textbf{Ste4: convergence of $\|\mathbf{p}^t-\mathbf{p}^{\star}\|_2$.}

We fix any $t \in E'$ and assume that $r$ is $\mu_r$-strongly concave. Then, since $\mathbf{p}^{\star}$ is the minimizer of $r$ over $\mathcal{P}$, we have

$$
\|\mathbf{p}^{t}-\mathbf{p}^{\star}\|_2^2 \lesssim r(\mathbf{p}^t) -
r(\mathbf{p}^{\star})
$$
and by \eqref{eq:conv_single_rev} we get 

$$
\|\mathbf{p}^{t}-\mathbf{p}^{\star}\|_2^2 \lesssim r(\mathbf{p}^t) -
r(\mathbf{p}^{\star})\lesssim \sqrt{N}\left(\frac{\ln n}{n}\right)^{\frac{\alpha}{2\alpha+1}}. 
$$

\end{proof}

\newpage

\section{Optimal Price for symmetric and non-negative networks}\label{app:linear-logistic}

In this section, we compare two extreme benchmark networks:
\[
\mathbf G=\mathbf O_N
\qquad\text{and}\qquad
\mathbf G=\mathbf 1\mathbf 1^\top-\mathbf I_N,
\]
that is, the no-network case and the complete network with unit off-diagonal interactions. We show that, in the Linear--Quadratic model with homogeneous consumers, the equilibrium consumption changes but the optimal monopoly price does not. By contrast, this invariance generally fails beyond the Linear--Quadratic case.

Throughout this subsection, assume homogeneous consumers:
\[
a_i=a,\qquad b_i=b,\qquad h_i=h,
\qquad \forall i\in\mathcal N.
\]
By symmetry, in both benchmark networks any interior equilibrium and any optimal price vector are necessarily symmetric, so that
\[
x_i^\star=x^\star,\qquad p_i^\star=p^\star,
\qquad \forall i\in\mathcal N.
\]

\begin{proposition}[Linear--Quadratic case]
\label{prop:complete_vs_null_linear}
Consider the Linear--Quadratic model
\[
\textstyle h(x)=\frac12 x^2,
\qquad \phi(x)=h'(x)=x.
\]
Then the following hold.

\begin{enumerate}[ leftmargin=1.2em]
    \item If $\mathbf G=\mathbf O_N$, then
    \[
    \textstyle x^\star=a-bp^\star,
    \qquad
    p^\star=\frac{a}{2b},
    \qquad
    x^\star=\frac{a}{2}.
    \]

    \item If $\mathbf G=\mathbf 1\mathbf 1^\top-\mathbf I_N$, and $\delta(N-1)<1$, then
    \[
    \textstyle x^\star=a-bp^\star+\delta (N-1)x^\star,
    \]
    hence
    \[
    \textstyle x^\star=\frac{a-bp^\star}{1-\delta(N-1)}.
    \]
    The seller's optimal price is still
    \[
    \textstyle p^\star=\frac{a}{2b},
    \]
    while the equilibrium consumption becomes
    \[
    \textstyle x^\star=\frac{a}{2(1-\delta(N-1))}.
    \]
\end{enumerate}

Therefore, in the Linear--Quadratic case, the complete network changes the equilibrium consumption level but not the optimal monopoly price.
\end{proposition}

\begin{proof}
If $\mathbf G=\mathbf O_N$, then the equilibrium condition reduces to
\[
x=a-bp.
\]
The seller's per-consumer revenue is
\[
r(p)=px=p(a-bp),
\]
whose first-order condition is
\[
a-2bp=0.
\]
Hence
\[
\textstyle p^\star=\frac{a}{2b},
\qquad
x^\star=a-bp^\star=\frac{a}{2}.
\]

Now let $\mathbf G=\mathbf 1\mathbf 1^\top-\mathbf I_N$. Since each row sum is $N-1$, the symmetric equilibrium condition is
\[
x=a-bp+\delta (N-1)x,
\]
that is,
\[
\textstyle x=\frac{a-bp}{1-\delta(N-1)}.
\]
The per-consumer revenue is therefore
\[
\textstyle r(p)=px
=
\frac{p(a-bp)}{1-\delta(N-1)}.
\]
Since the multiplicative factor $(1-\delta(N-1))^{-1}$ is constant in $p$, maximizing $r(p)$ is equivalent to maximizing $p(a-bp)$. Hence the first-order condition is again
\[
a-2bp=0,
\]
which gives
\[
\textstyle p^\star=\frac{a}{2b}.
\]
Substituting into the equilibrium expression yields
\[
\textstyle x^\star=\frac{a-bp^\star}{1-\delta(N-1)}
=
\frac{a}{2(1-\delta(N-1))}.
\]
\end{proof}

The previous proposition shows that the invariance of the optimal price is a special feature of the Linear--Quadratic model. In general, once $h$ is nonlinear, the complete network changes not only the equilibrium consumption, but also the seller's optimal price.

\begin{proposition}[General nonlinear case]
\label{prop:complete_vs_null_general}
Let $h$ be any continuously differentiable, strictly convex function, and let $\phi=h'$. Under the same homogeneity assumptions as above, the optimal monopoly price under $\mathbf G=\mathbf O_N$ is characterized by
\[
\textstyle a-\phi(x)-x\phi'(x)=0,
\qquad
p=\frac{a-\phi(x)}{b}.
\]
By contrast, under $\mathbf G=\mathbf 1\mathbf 1^\top-\mathbf I_N$, the optimal monopoly price is characterized by
\[
\textstyle a-\phi(x)-x\phi'(x)+2\delta (N-1)x=0,
\qquad
p=\frac{a-\phi(x)+\delta (N-1)x}{b}.
\]
Hence, except in special cases, the optimal price differs between the no-network and the complete-network benchmarks.
\end{proposition}

\begin{proof}
If $\mathbf G=\mathbf O_N$, then $x=\phi^{-1}(a-bp)$, equivalently
\[
\textstyle p=\frac{a-\phi(x)}{b}.
\]
Thus the per-consumer revenue can be written as
\[
\textstyle r(x)=px=\frac{x(a-\phi(x))}{b}.
\]
Differentiating with respect to $x$ gives
\[
\textstyle r'(x)=\frac{1}{b}\bigl(a-\phi(x)-x\phi'(x)\bigr),
\]
which proves the first characterization.

If instead $\mathbf G=\mathbf 1\mathbf 1^\top-\mathbf I_N$, then by symmetry the equilibrium condition becomes
\[
\phi(x)-\delta (N-1)x=a-bp,
\]
that is,
\[
\textstyle p=\frac{a-\phi(x)+\delta (N-1)x}{b}.
\]
Hence
\[
\textstyle r(x)=\frac{x(a-\phi(x)+\delta (N-1)x)}{b}.
\]
Differentiating gives
\[
\textstyle r'(x)
=
\frac{1}{b}\bigl(a-\phi(x)-x\phi'(x)+2\delta (N-1)x\bigr),
\]
which proves the second characterization.
\end{proof}

Proposition~\ref{prop:complete_vs_null_general} shows that the equality of the optimal prices in Proposition~\ref{prop:complete_vs_null_linear} is not a generic network-invariance property, but rather a consequence of the linear form $\phi(x)=x$. Indeed, when $\phi(x)=x$, the extra term $2\delta (N-1)x$ introduced by the complete network is exactly offset when one translates the optimal $x^\star$ back into the price $p^\star$. This cancellation no longer holds for a general nonlinear $\phi$.

\begin{example}[Discrete-choice/logistic model]
\label{ex:complete_vs_null_logistic}
Consider
\[
h(x)=x\ln x+(1-x)\ln(1-x),
\qquad x\in(0,1),
\]
so that
\[
\textstyle \phi(x)=\ln \left(\frac{x}{1-x}\right),
\qquad
\phi'(x)=\frac{1}{x(1-x)}.
\]

If $\mathbf G=\mathbf O_N$, the optimal $x^\star\in(0,1)$ is characterized by
\[
\textstyle a-\ln \left(\frac{x}{1-x}\right)-\frac{1}{1-x}=0,
\]
and the corresponding price is
\[
\textstyle p^\star=\frac{a-\ln \left(\frac{x^\star}{1-x^\star}\right)}{b}.
\]

If $\mathbf G=\mathbf 1\mathbf 1^\top-\mathbf I_N$, the optimal $x^\star\in(0,1)$ is characterized by
\[
\textstyle a-\ln \left(\frac{x}{1-x}\right)-\frac{1}{1-x}+2\delta (N-1)x=0,
\]
and the corresponding price is
\[
\textstyle p^\star
=
\frac{a-\ln \left(\frac{x^\star}{1-x^\star}\right)+\delta (N-1)x^\star}{b}.
\]

Thus, unlike the Linear--Quadratic case, the optimal monopoly price in the discrete-choice model is generally different in the no-network and complete-network benchmarks.
\end{example}

\newpage

\section{Network Communities and their pricing policies}\label{app:examples_of_networks}

Suppose that, after reordering the consumers, the interaction matrix is block diagonal:
\[
\mathbf G
=
\operatorname{diag} \bigl(\mathbf G^{(1)},\dots,\mathbf G^{(K)}\bigr),
\]
where each block $\mathbf G^{(k)}$ represents one community.  Then there are no spillovers across different communities, and the consumer equilibrium problem decomposes block by block. Note that this model has to be interpreted in expectation, in the sense that the observed network is 
$$
\widetilde{\mathbf{G}} = \mathbf{G} + \mathbf{E}
$$
where $\mathbf{E}$ is a $N \times N$ random matrix with mean zero. To see why the consumer equilibrium problem decomposes block by block, consider the partition
\[
\mathbf x=(\mathbf x^{(1)},\dots,\mathbf x^{(K)}),\qquad
\mathbf p=(\mathbf p^{(1)},\dots,\mathbf p^{(K)}),\qquad
\mathbf a=(\mathbf a^{(1)},\dots,\mathbf a^{(K)}),
\]
and similarly write
\[
\mathbf B=\operatorname{diag} \bigl(\mathbf B^{(1)},\dots,\mathbf B^{(K)}\bigr),
\qquad
\mathbf D=\operatorname{diag} \bigl(\mathbf D^{(1)},\dots,\mathbf D^{(K)}\bigr),
\]
then the equilibrium condition \eqref{FOC:consumer-global} becomes, for each block $k$,
\begin{equation}\label{eq:equil-communities}
\mathbf a^{(k)}-\mathbf B^{(k)}\mathbf p^{(k)}
+\delta \mathbf G^{(k)}\mathbf x^{(k)}
-\boldsymbol\phi^{(k)}(\mathbf x^{(k)})=\mathbf 0.
\end{equation}
Hence, each community behaves as an independent submarket. This immediately implies that both uniqueness conditions decompose by block. The contraction condition
\[
\rho \left(\delta \mathbf D^{-1}|\mathbf G|\right)<1
\]
is equivalent to
\[
\rho \left(\delta (\mathbf D^{(k)})^{-1}|\mathbf G^{(k)}|\right)<1,
\qquad \forall k=1,\dots,K,
\]
because the spectral radius of a block-diagonal matrix is the maximum of the spectral radii of its blocks. Likewise, the variational condition
\[
\mathbf D-\delta\frac{\mathbf G+\mathbf G^\top}{2}\succ \mathbf{O}_N
\]
holds if and only if
\[
\mathbf D^{(k)}-\delta\frac{\mathbf G^{(k)}+(\mathbf G^{(k)})^\top}{2}\succ \mathbf{O}_N,
\qquad \forall k=1,\dots,K.
\]

Therefore, the eigenvalues of each block describe the internal strength of reinforcement within that community. In the symmetric entry-wise nonnegative case, the two uniqueness conditions coincide on each block and reduce to
\[
\delta\,\rho \left((\mathbf D^{(k)})^{-1}\mathbf G^{(k)}\right)<1,
\qquad \forall k=1,\dots,K.
\]
Thus, the leading eigenvalue of block $k$ measures how close community $k$ is to the instability threshold: larger eigenvalues correspond to stronger within-community amplification.

From \eqref{eq:equil-communities} we have that for each block $k$,
$$
 \mathbf x^{\star,(k)}(\mathbf p^{(k)}),
$$
that is, the demand for block $k$ depends only on the price of that block. Hence, the seller's revenue decomposes as
\[
r(\mathbf p)
=
\sum_{k=1}^K
(\mathbf p^{(k)})^\top \mathbf x^{\star,(k)}(\mathbf p^{(k)}),
\]
hence, the monopoli problem separates
$$
\max _{\mathbf{p}} r(\mathbf{p})=\sum_{k=1}^K \max _{\mathbf{p}^{(k)}}\mathbf{p}^{(k)\top} \mathbf{x}^{\star,(k)}(\mathbf{p}^{(k)}) .
$$
mare precisely
$$
\mathbf{p}^{\star}=(\mathbf{p}^{\star,(1)}, \ldots, \mathbf{p}^{\star,(K)}),
$$
where, every $\mathbf{p}^{\star,(k)}$ solves

$$
\mathbf{p}^{\star,(k)} \in \arg \max _{\mathbf{p}^{(k)} \in \mathcal{P}^{(k)}}\mathbf{p}^{(k)\top} \mathbf{x}^{\star,(k)}(\mathbf{p}^{(k)}) .
$$
Once the network is block diagonal, each community can be solved independently, so the price for consumer $i$ only depends on the data inside consumer $i$’s own block. 

More concretely, suppose $\mathbf{G}=\operatorname{diag}\left(\mathbf{G}^{(1)}, \ldots, \mathbf{G}^{(K)}\right)$, and consumer $i$ belongs to block $k$. Then there are no links between block $k$ and the other blocks, so: \begin{itemize}
    \item consumers in other blocks do not affect $x_i^{\star}$,
    \item changing prices in other blocks does not affect the demand in block $k$,
    \item and therefore the seller's optimal choice for $p_i$ depends only on: the local network $\mathbf{G}^{(k)}$, the local parameters $a_j, b_j, h_j$ for $j$ in block $k$, and the prices inside block $k$.
\end{itemize}

It does not depend on what happens in the other communities. Thus, the seller generally does discriminate across blocks whenever the blocks differ in size, density, or spectral structure. Two consumers with identical primitive parameters $(a_i,b_i,h_i)$ may receive different prices solely because they belong to different communities. In the Linear--Quadratic case, this is completely governed by the block-specific resolvents $(\mathbf I-\delta \mathbf G^{(k)})^{-1}$, equivalently by the Bonacich-type centralities computed within each community.

\newpage
\section{Supplementary Details on Network Influencers}
\label{sec:influencer_appendix}

\subsection{Invertibility of the Jacobian Matrix}
\label{sec:invertibility_of_J_supp}
The formal definition of the intrinsic influential value (IIV) relies on the invertibility of the game's Jacobian matrix $\mathbf{J}(\mathbf{x}) = D_{\boldsymbol{\phi}}(\mathbf{x}) - \delta \mathbf{G}$. This matrix is guaranteed to be invertible for every $\mathbf{x} \in \mathcal{X}$ under either of our two uniqueness criteria from Theorem~\ref{thm:uniqueness_NE}. Specifically, $\mathbf{J}(\mathbf{x})^{-1}$ is well-defined if the variational condition holds, or if the network is nonnegative ($\mathbf{G} = |\mathbf{G}|$) and satisfies the contraction condition $\rho(\delta \mathbf{D}^{-1}\mathbf{G}) < 1$ (see Appedix~\ref{sec:invertibility_of_J}).

\subsection{IIV in the Discrete-Choice Model}
Under the assumptions of Proposition~\ref{prop:v_i-equality}, we can explicitly evaluate the intrinsic influential value for specific utility forms. Consider the Discrete-Choice model, where $h_i(x) = x \ln x + (1-x) \ln(1-x)$ for $x \in [0,1]$. Here, we have:
\[
    h_i''(x) = \phi_i'(x) = \frac{1}{x(1-x)} \ge 4.
\]
This inequality holds with equality if and only if $x = \tfrac{1}{2}$. Thus, the strong convexity parameter evaluates to $\mathbf{D} = 4\mathbf{I}_N$. Applying Proposition~\ref{prop:v_i-equality}, the IIV reduces exactly to:
\[
    V_i = \bigl[\mathbf{1}^\top(4\mathbf{I}_N - \delta \mathbf{G})^{-1}\bigr]_i.
\]

\subsection{Signed Networks and Dual-Space Interpretations}
For general signed or non-symmetric networks, an explicit analytical bound for $V_i$ mapping directly to $(\mathbf{D} - \delta \mathbf{G})^{-1}$ is not generally available. The proof of Proposition~\ref{prop:v_i-equality} relies on the entry-wise monotonicity of the Neumann series, which breaks down once $\mathbf{G}$ contains mixed signs. Consequently, for nonlinear models, the assumption that $\mathbf{G} = |\mathbf{G}|$ is typically required to leverage the analytic bound $V_i \le \bigl[\mathbf{1}^\top(\mathbf{D} - \delta \mathbf{G})^{-1}\bigr]_i$.

However, in the baseline Linear--Quadratic case where $h_i(x) = \frac{1}{2}x^2$, the assumption that $\mathbf{G} = |\mathbf{G}|$ is unnecessary. Because $D_{\boldsymbol{\phi}}(\mathbf{x}) = \mathbf{I}_N$, the Jacobian $\mathbf{J}(\mathbf{x}) = \mathbf{I}_N - \delta \mathbf{G}$ is independent of $\mathbf{x}$, recovering the standard Katz--Bonacich centralities with parameter $\mathbf{b}(\delta\mathbf{G}) = \mathbf{1}^\top(\mathbf{I}_N - \delta \mathbf{G})^{-1}$. 

More broadly, our formalization of $V_i$ with a general $\boldsymbol{\phi}$ serves as a generalization of Katz--Bonacich centrality evaluated in the dual space. This space is induced by the mapping $\boldsymbol{\phi}(\mathbf{x}) = \nabla \boldsymbol{h}(\mathbf{x}) = (\nabla \boldsymbol{h}^*(\mathbf{x}))^{-1}$, where $\boldsymbol{h}^*$ denotes the convex conjugate (dual function) of $\mathbf{x} \mapsto \boldsymbol{h}(\mathbf{x}) = (h_1(x_1), \dots, h_N(x_N))$.

\subsection{Influencer Pricing and the Room for Discounting}
\label{subsec:influencer-pricing}

\begin{proposition}[Pure influencer: implicit pricing equation and quadratic revenue loss]
\label{prop:general_pure_influencer}
Assume that $\mathcal P\subseteq \mathbb [0,\infty)^{N}$ contains $\{0\}$, is nonempty, closed, and convex. Assume the hypotheses of Theorem~\ref{thm:general_competitive_NE_existence_monopolist}, let Assumption~\ref{ass:strong-convexity} hold, and suppose that
\[
\textstyle \mathbf D-\delta\,\frac{\mathbf G+\mathbf G^\top}{2}\succ \mathbf{O}_N\qquad \text{ or } \qquad \{\mathbf{G}=|\mathbf{G}| \text{, and } \rho(\delta \mathbf{G})<1\}
\]
is satisfied (hence for every $\mathbf p\in\mathcal P$,  $\mathbf x^\star(\mathbf p)\in\mathcal X$ exists and is unique by Theorem~\ref{thm:uniqueness_NE}). Moreover, assume that $\mathbf x^\star(\mathbf p)\in\mathring{\mathcal X}$ for all $\mathbf p \in \mathcal{P}$, and that $\mathbf p^\star\in\mathcal P$ is an interior global maximizer of the seller's revenue $r(\mathbf p):=\mathbf p^\top \mathbf x^\star(\mathbf p)$. Now let $i\in\mathcal N$, such that the $i$-th row of $\mathbf{G}$ is zero. Then the following hold:

\begin{enumerate}[ leftmargin=1.2em]
    \item $p_i^\star
    =
     b_i^{-1}[ a_i-\phi_i(b_i\sum_{j=1}^N
\bigl[\mathbf J(\mathbf x^\star(\mathbf p^\star))^{-1}\bigr]_{ji}p_j^\star)]>0$.
    
    \item Let $\widetilde{\mathbf p}\in\arg\max_{\mathbf p\in \widetilde{\mathcal P}_i} r(\mathbf p)$ with $\widetilde{\mathcal P}_i:=\{\mathbf p\in\mathcal P:\ p_i=0\}$. Then, if exists $\mathbf Q:=-\nabla^2 r(\mathbf p^\star)\succ \mathbf{O}_N$, there exists $\theta_i>0$ such that $r(\mathbf p^\star)-r(\widetilde{\mathbf p})
    =
    \frac12\,\theta_i\,(p_i^\star)^2
    +
    o ((p_i^\star)^2)$, as $p_i^\star\to 0$.
\end{enumerate}
\end{proposition}

\begin{remark}[Connection with the external demand-drop index]
If $p_i^\star$ is close to zero, then forcing the constraint $p_i=0$ changes total revenue only by $r(\mathbf p^\star)-r(\widetilde{\mathbf p}) = \frac12\,\theta_i\,(p_i^\star)^2+o ((p_i^\star)^2)$. Hence, for a strongly influential node, the seller may optimally choose a very small price, and once that price is close to zero, setting it exactly to zero has only a second-order effect on revenue.
\end{remark}

\begin{corollary}[A sufficient condition for a strictly positive optimal price]\label{cor:structural_influencer_zero_price}
Assume that $\mathcal P= \times_{i \in \mathcal{N}}\mathbb [0,\bar{p}_i]^{N}$ for some $0<\bar{p}_i<\infty$ and let $\bar{p}:=\max_{i \in \mathcal{N}} \bar{p}_i$. Assume the hypotheses of Theorem~\ref{thm:general_competitive_NE_existence_monopolist}, let Assumption~\ref{ass:strong-convexity} hold. Additionally, assume that $\mathbf{G} = |\mathbf{G}|$ and the contraction condition holds: $\rho(\delta \mathbf G)<1$ (hence for every $\mathbf p\in\mathcal P$,  $\mathbf x^\star(\mathbf p)\in\mathcal X$ exists and is unique by Theorem~\ref{thm:uniqueness_NE}). Moreover, assume that $\mathbf x^\star(\mathbf p)\in\mathring{\mathcal X}$ for all $\mathbf p \in \mathcal{P}$, and that $\mathbf p^\star\in\mathcal P$ is an interior local maximizer of the seller's revenue $r(\mathbf p):=\mathbf p^\top \mathbf x^\star(\mathbf p)$. Now let $i\in\mathcal N$, such that the $i$-th row of $\mathbf{G}$ is zero. Then 
$$
\textstyle V_i<\frac{\phi_i^{-1}(a_i)}{\overline p\, b_i} \quad \implies \quad p_i^\star \ge \frac{a_i-\phi_i(\bar{p}b_iV_i)}{b_i}>0.
$$
\end{corollary}

\begin{remark}
Corollary~\ref{cor:structural_influencer_zero_price} does not imply a direct ordering of the optimal prices $p_i^\star$ across consumers, since it only provides a \emph{lower bound} on each $p_i^\star$ for the $i$ with $i$-th row of $\mathbf G$ is null. However, it does allow a meaningful pairwise comparison of the seller's \emph{structural room for discounting} to different consumers. To see this, suppose that two consumers $i,k\in\mathcal N$ satisfy the assumptions of Corollary~\ref{cor:structural_influencer_zero_price}, and in addition share the same primitives: $a_i=a_k=:a$, $b_i=b_k=:b$, $\phi_i=\phi_k=:\phi$. Then the corollary yields the common lower-bound formula $p_\ell^\star \ge L(V_\ell)$, $L(V):=\frac{a-\phi(\bar p\,bV)}{b}$, $\ell\in\{i,k\}$. Since $\phi$ is increasing, the function $L(V)$ is decreasing in $V$. Therefore,
\[
\textstyle V_i>V_k
\qquad\Longrightarrow\qquad
L(V_i)<L(V_k).
\]
This means that influencers have the lowest lower bound on the price: the model allows the seller to discount that consumer more aggressively without violating the structural lower bound implied by optimality. In this sense, $V_i$ does not rank the optimal prices themselves, but it does rank how much \emph{room} the seller has to assign lower prices. Thus, more influential consumers are precisely those for whom the theory permits more aggressive discounting.
\end{remark}

\newpage

\section{Examples of utility functions covered by our model}\label{ex:utilities-covered}

\begin{example}\label{example_1}
By substituting different functional forms for $h_i\left(x_i\right)$, this framework encapsulates four canonical economic models. 

\noindent
\textbf{(1) The Linear-Quadratic Utility Model (Standard Goods)}: is defined with
$$
\begin{aligned}
h\left(x\right)=\tfrac{1}{2} x^2, \quad x \in \mathbb R,
\end{aligned}
$$
and is used heavily in standard network pricing literature, such as: the monopolistic model by \cite{ballester2006s}, which considers the case $b_i=0$ and $\mathbf{G}$; the monopolistic model by \cite{li2025price} which considers the case $b_i=1$ and $\mathbf{G}$ symmetric, and similarly, the multiseller model by \cite{chen2018competitive}, considers the case $b_i=1$ and $\mathbf{G}$ symmetric as well. 
\noindent

\textbf{(2) Stone-Geary Utility Model (Subsistence/Essential Goods):} requires a baseline minimum consumption $\gamma>0$ and parameter strenght $\beta>0$:
$$
\begin{aligned}
h\left(x\right)=-\beta \ln \left(x-\gamma\right),\quad x \in (\gamma,\infty).
\end{aligned}
$$

\noindent
\textbf{(3) Discrete Choice Utility Model (Standard Logit)}: Models market share probabilities $x \in [0,1]$ by defining intrinsic utility via a negative Shannon Entropy penalty:
$$
\begin{aligned}
h\left(x\right)=x \ln x+\left(1-x\right) \ln \left(1-x\right), \quad x \in [0,1],  \quad 0\ln (0) = 0.
\end{aligned}
$$
This model was used in \cite{chen2021duopoly} and  \cite{du2016optimal}.

\noindent
\textbf{(4) Exponential Utility Model}:  used a lot in the theory of risk aversion \citep{menezes1970theory,meyer2014theory}
$$
h(x) = e^{\gamma x} -1,\quad x \geq 0
$$
for some $\gamma >0$.

\noindent
\textbf{(5) Isoelastic Utility Model}: is used to express utility in terms of consumption or some other economic variable that a decision-maker is concerned with \cite{ljungqvist2018recursive, kale2009growth}. In statistics is also called Box-Cox transformation: for $x >0$
$$
h(x)= \begin{cases}-\frac{x^{1-\gamma}-1}{1-\gamma} & ,\gamma \in [0,1) \\ -\ln (x) & ,\gamma=1\end{cases}.
$$
\end{example}

\newpage

\section{Some examples of monopolistic Euclidean network topologies: follower and influencer}\label{app:more_examples}

We study special monopolistic Euclidean models: directed star networks and explicit optimal prices. In the directed star cases below, we have $\mathbf{G}^2=\mathbf{O}_N$, hence
\[
\mathbf{H}=(\mathbf{I}_N-\delta \mathbf{G})^{-1}=\mathbf{I}_N+\delta \mathbf{G}.
\]
For simplicity of exposition, in the following two examples we assume the price sensitivities are $\mathbf{B}=\mathbf{I}_N$.
\textbf{Case 1. Follower.} Suppose $g_{1j}=1$ for $j\neq 1$ and $g_{ij}=0$ otherwise.

\begin{figure}[h]
\centering
\begin{minipage}[c]{0.44\textwidth}
\centering
\[
\mathbf G=
\begin{pmatrix}
0 & 1 & 1 & 1 & 1\\
0 & 0 & 0 & 0 & 0\\
0 & 0 & 0 & 0 & 0\\
0 & 0 & 0 & 0 & 0\\
0 & 0 & 0 & 0 & 0\\
\end{pmatrix},
\]
\end{minipage}
\hfill
\begin{minipage}[c]{0.52\textwidth}
\centering
\begin{tikzpicture}[
  >=Latex,
  node/.style={circle, draw=black, thick, fill=black!8, font=\scriptsize, inner sep=1.2pt},
  edge/.style={->, draw=black!60, line width=0.8pt},
  elabel/.style={font=\tiny, fill=white, inner sep=0.8pt}
]
\node[node, minimum size=6.5mm] (1) at (-1.6,1.2) {$3$};
\node[node, minimum size=6.5mm] (2) at (-1.8,-0.6) {$2$};
\node[node, minimum size=9mm]   (3) at (0,0.2) {$1$};
\node[node, minimum size=6.5mm] (4) at (1.8,1.1) {$4$};
\node[node, minimum size=6.5mm] (5) at (1.7,-0.9) {$5$};

\draw[edge] (3) -- node[elabel, above left] {$1$} (1);
\draw[edge] (3) -- node[elabel, below left] {$1$} (2);
\draw[edge] (3) -- node[elabel, above right] {$1$} (4);
\draw[edge] (3) -- node[elabel, below right] {$1$} (5);

\node[draw=none, font=\tiny, align=left] at (0,-1.8)
{Node $1$ influences all the others: $g_{i1}>0$ for $i\neq 1$.};
\end{tikzpicture}
\end{minipage}
\end{figure}
Then
\[
\mathbf{G}=\mathbf{e}_1 s^\top,
\qquad
\mathbf{e}_1=(1,0,\dots,0)^\top,
\qquad
\mathbf{s}=(0,1,\dots,1)^\top.
\]
Therefore
\[
\mathbf{H}=\mathbf{I}_N+\delta \mathbf{e}_1 \mathbf{s}^\top,
\qquad
\mathbf{H}\mathbf{a}=\mathbf{a}+\delta\Big(\sum_{j=2}^n a_j\Big)\mathbf{e}_1.
\]
Let $A:=\sum_{j=2}^n a_j$. The first-order condition $(\mathbf{H}+\mathbf{H}^\top)\mathbf{p}=\mathbf{H}\mathbf{a}$ becomes
\[
\bigl(2\mathbf{I}_N+\delta(\mathbf{e}_1\mathbf{a}^\top+\mathbf{s}\mathbf{e}_1^\top)\bigr)\mathbf{p}=\mathbf{a}+\delta A \mathbf{e}_1.
\]
Coordinatewise,
\[
2p_1+\delta\sum_{j=2}^n p_j=a_1+\delta A,
\]
and, for $i\ge 2$,
\[
2p_i+\delta p_1=a_i.
\]
Thus, for $i\ge 2$, we have
\[
p_i^\star=\frac{a_i-\delta p_1^\star}{2},
\qquad
p_1^\star=\frac{2a_1+\delta A}{4-(n-1)\delta^2},
\qquad
A:=\sum_{j=2}^n a_j.
\]

Hence
\[
p_1^\star>p_i^\star
\iff
p_1^\star>\frac{a_i-\delta p_1^\star}{2}
\iff
(2+\delta)p_1^\star>a_i.
\]
Substituting the explicit formula for $p_1^\star$, we obtain
\[
p_1^\star>p_i^\star
\iff
(2+\delta)\frac{2a_1+\delta A}{4-(n-1)\delta^2}>a_i,
\]
that is,
\begin{equation}\label{eq:star_exact_condition}
p_1^\star>p_i^\star
\iff
(2+\delta)\bigl(2a_1+\delta A\bigr)
>
a_i\bigl(4-(n-1)\delta^2\bigr).
\end{equation}
Without further assumptions on the intercepts $a_i$, the inequality $p_1^\star>p_i^\star$ does \emph{not} hold in general. A simple sufficient condition is
\[
a_1\ge a_i,
\qquad \forall i\ge 2,
\]
assuming also $a_j\ge 0$ for all $j$. Indeed, since $A=\sum_{j=2}^n a_j\ge 0$, we have
\[
(2+\delta)\bigl(2a_1+\delta A\bigr)\ge 2(2+\delta)a_1.
\]
Moreover, if $a_i\le a_1$, then
\[
a_i\bigl(4-(n-1)\delta^2\bigr)
\le
a_1\bigl(4-(n-1)\delta^2\bigr).
\]
Thus it is enough to verify that
\[
2(2+\delta)a_1>a_1\bigl(4-(n-1)\delta^2\bigr),
\]
or equivalently,
\[
4+2\delta>4-(n-1)\delta^2.
\]
This is the same as
\[
2\delta+(n-1)\delta^2>0,
\]
which always holds for $\delta>0$. Hence, under the sufficient condition $a_1\ge a_i$ for all $i\ge 2$, we obtain
\[
p_1^\star>p_i^\star,
\qquad \forall i\ge 2.
\]

In particular, when $a_i=a$ for all $i$, condition \eqref{eq:star_exact_condition} is automatically satisfied, recovering the previous conclusion.

\textbf{Case 2: Influeencer.} Suppose instead $g_{i1}=1$ for $i\neq 1$, and $g_{ij}=0$ otherwise.

\begin{figure}[h]
\centering
\begin{minipage}[c]{0.44\textwidth}
\centering
\[
\mathbf G=
\begin{pmatrix}
0 & 0 & 0 & 0 & 0\\
1 & 0 & 0 & 0 & 0\\
1 & 0 & 0 & 0 & 0\\
1 & 0 & 0 & 0 & 0\\
1 & 0 & 0 & 0 & 0\\
\end{pmatrix},
\]
\end{minipage}
\hfill
\begin{minipage}[c]{0.52\textwidth}
\centering
\begin{tikzpicture}[
  >=Latex,
  node/.style={circle, draw=black, thick, fill=black!8, font=\scriptsize, inner sep=1.2pt},
  edge/.style={->, draw=black!60, line width=0.8pt},
  elabel/.style={font=\tiny, fill=white, inner sep=0.8pt}
]
\node[node, minimum size=6.5mm] (1) at (-1.6,1.2) {$2$};
\node[node, minimum size=6.5mm] (2) at (-1.8,-0.6) {$2$};
\node[node, minimum size=9mm]   (3) at (0,0.2) {$1$};
\node[node, minimum size=6.5mm] (4) at (1.8,1.1) {$4$};
\node[node, minimum size=6.5mm] (5) at (1.7,-0.9) {$5$};

\draw[edge] (1) -- node[elabel, above left] {$1$} (3);
\draw[edge] (2) -- node[elabel, below left] {$1$} (3);
\draw[edge] (4) -- node[elabel, above right] {$1$} (3);
\draw[edge] (5) -- node[elabel, below right] {$1$} (3);

\node[draw=none, font=\tiny, align=left] at (0,-1.8)
{Node $1$ is influenced by all the others: $g_{1j}=1$ for $j\neq 1$.};
\end{tikzpicture}
\end{minipage}
\end{figure}

Then
\[
\mathbf{G}=\mathbf{s}\mathbf{e}_1^\top, \qquad \mathbf{e}_1 = (1,0,\dots,0)^{\top}
\qquad\mathbf{s}=(0,1,\dots,1)^\top.
\]
Therefore
\[
\mathbf{H}=\mathbf{I}_N+\delta \mathbf{s}\mathbf{e}_1^\top,
\qquad
\mathbf{H}\mathbf{a}=\mathbf{a}+\delta a_1 \mathbf{s}.
\]
The first-order condition becomes
\[
\bigl(2\mathbf{I}_N+\delta(\mathbf{s}\mathbf{e}_1^\top+\mathbf{e}_1\mathbf{s}^\top)\bigr)\mathbf{p}=\mathbf{a}+\delta a_1 \mathbf{s}.
\]
Coordinatewise,
\[
2p_1+\delta\sum_{j=2}^n p_j=a_1,
\]
and, for $i\ge 2$,
\[
2p_i+\delta p_1=a_i+\delta a_1.
\]
Hence, for $i\ge 2$, we have
\[
p_i^\star=\frac{a_i+\delta a_1-\delta p_1^\star}{2},
\qquad
p_1^\star=\frac{2a_1-\delta A-(n-1)\delta^2 a_1}{4-(n-1)\delta^2},
\qquad
A:=\sum_{j=2}^n a_j.
\]
Hence
\[
p_1^\star<p_i^\star
\iff
p_1^\star<\frac{a_i+\delta a_1-\delta p_1^\star}{2}
\iff
(2+\delta)p_1^\star<a_i+\delta a_1.
\]
Substituting the explicit formula for $p_1^\star$, we obtain
\[
p_1^\star<p_i^\star
\iff
(2+\delta)\frac{2a_1-\delta A-(n-1)\delta^2 a_1}{4-(n-1)\delta^2}
<
a_i+\delta a_1,
\]
that is,
\begin{equation}\label{eq:star_column_exact_condition}
p_1^\star<p_i^\star
\iff
(2+\delta)\bigl(2a_1-\delta A-(n-1)\delta^2 a_1\bigr)
<
(a_i+\delta a_1)\bigl(4-(n-1)\delta^2\bigr).
\end{equation}
Therefore, without further assumptions on the intercepts $a_i$, the comparison between $p_1^\star$ and $p_i^\star$ is not universal.

A simple sufficient condition ensuring that the center is priced \emph{below} every leaf is
\[
a_i\ge a_1,
\qquad \forall i\ge 2,
\]
assuming also $a_j\ge 0$ for all $j$. Indeed, since $A=\sum_{j=2}^n a_j\ge 0$, we have
\[
(2+\delta)\bigl(2a_1-\delta A-(n-1)\delta^2 a_1\bigr)
\le
(2+\delta)\bigl(2a_1-(n-1)\delta^2 a_1\bigr).
\]
Moreover, if $a_i\ge a_1$, then
\[
(a_i+\delta a_1)\bigl(4-(n-1)\delta^2\bigr)
\ge
(1+\delta)a_1\bigl(4-(n-1)\delta^2\bigr).
\]
Thus it is enough to verify that
\[
(2+\delta)\bigl(2a_1-(n-1)\delta^2 a_1\bigr)
<
(1+\delta)a_1\bigl(4-(n-1)\delta^2\bigr),
\]
or equivalently,
\[
(2+\delta)\bigl(2-(n-1)\delta^2\bigr)
<
(1+\delta)\bigl(4-(n-1)\delta^2\bigr).
\]
Expanding both sides, this reduces to
\[
2\delta+(n-1)\delta^2>0,
\]
which always holds for $\delta>0$. Hence, under the sufficient condition $a_i\ge a_1$ for all $i\ge 2$, we obtain
\[
p_1^\star<p_i^\star,
\qquad \forall i\ge 2.
\]

In particular, when $a_i=a$ for all $i$, condition \eqref{eq:star_column_exact_condition} is automatically satisfied, and therefore
\[
p_1^\star<p_i^\star,
\qquad \forall i\ge 2.
\]

\newpage

\section{Finite-sample convergence of the isotonic estimates with continuous response}\label{app:finite-sample}

Our Theorem~\ref{thm:dynamic_regret_isotonic_network} relies on a key finite-sample convergence result, namely Lemma~\ref{lem:J2J3}. In turn, Lemma~\ref{lem:J2J3} follows as a direct application of the more general finite-sample isotonic convergence result stated in this section, namely Theorem~\ref{thm:Dumbgen} below.

Let $\mathcal{U}$ be a bounded interval in $\mathbb{R}$. In this section, we demonstrate the uniform convergence of the isotonic estimates $\widehat{\psi}$ of a non-decreasing function $\psi:\mathcal{U}\to \mathbb{R}$, from an i.i.d. sample $\{w_t, y_t\}_{t \in \mathcal{T}}$ where $\mathcal{T} = \{1,2,\dots,n\}$, for some finite $n$, $w_t \in \mathcal{U}$ and
\begin{equation}\label{eq:y_generate}
y_t = \psi(w_t)+\varepsilon_t
\end{equation}
where $\varepsilon_t$ is sub-gaussian with variance proxy $\sigma^2$. The isotonic estimate $\widehat{\psi}$ is defined as 
\begin{equation}\label{eq:likelihood_F}
\textstyle\widehat{\psi} \triangleq \operatorname{argmin}_{\psi \in \mathcal{S}} \sum_{t \in \mathcal{T}} (y_t-\psi(w_t))^2,
\end{equation}
where $\mathcal{S}$ is the set of non-decreasing functions in $\mathbb{R}$ \citep{robertson1988order, mosching2020monotone}.

The minimizer $\widehat{\psi}$ is a piecewise constant function with jumps at a subset of $\{w_t: t \in \mathcal{T}\}$. The order statistics on which $\widehat{\psi}$ is based are the order statistics of the values $w_t$ and the values of the corresponding $y_t$. To be more specific, let $u_1<u_2<\dots<u_m$ the different value of the observed $\{w_t\}_{t\in \mathcal{T}}$. For $j=1,\dots,m$ set 
$$
\textstyle o_j= \#\{t: w_t=u_j\},\quad \widehat{y}_j = \frac{1}{o_j}\sum_{i:w_{i}=u_j}y_i.
$$
For every $1\leq r \leq s \leq m$ let
$$
\textstyle o_{r s}\triangleq\sum_{j=r}^s o_j=\#\left\{t: u_r \leq w_t \leq u_s\right\}, \quad \widehat{y}_{rs} =
\frac{1}{o_{r s}} \sum_{j=r}^s o_j \widehat{y}_j.
$$
It is well known that $\widehat{\psi}= (\widehat{\psi}(u_1),\widehat{\psi}(u_2),\dots,\widehat{\psi}(u_m))$ may be represented by the following minimax and maximin formulae \citep{robertson1988order,dai2020bias}: for $1\leq j \leq m$
\begin{align*}
\widehat{\psi}(u_j) = \max_{r\leq j}\min_{j\leq s} \,\widehat{y}_{rs} = \min_{j\leq s}\max_{r\leq j} \,\widehat{y}_{rs}.
\end{align*}
The $\widehat{\psi}$ is also known as the isotonic regression on data $\{\left(w_t,y_t\right)\}_{t\in \mathcal{T}}$, and we will denote it as
$$
\widehat{\psi} = \text{Isotonic} \{\left(w_t,y_t\right)\}_{t\in \mathcal{T}}.
$$
We are now prepared to demonstrate the main result of this section, that is, the convergence of $\widehat{\psi}$ to $\psi$ in the supremum norm.

\begin{theorem}\label{thm:Dumbgen}
Let $\psi$ is $\alpha$-H\"older for some $\alpha \in (0,1]$: $|\psi(u)-\psi(v)|\leq C_1|u-v|^{\alpha}$ for all $u,v \in \mathcal{U}$, for some $C_1>0$. Let $w_1,w_2,\dots,w_n \in \mathcal{U}$ i.i.d. points with density $f_w$ that satisfies $\inf_{u \in \mathcal{U}} f_w (u) \geq C_2$ for some universal constant $C_2>0$, and let $y_i$ be generated as in \eqref{eq:y_generate}. Then for every $\kappa >0$ and $\gamma>2$ there exists $n_0=n_0(\gamma,\kappa,\alpha) \in \mathbb{N}$ and $C= C(C_1,\mathcal{U},\sigma,\kappa,\alpha)>0$ such that
$$
\mathbb{P}\left\{\underset{u \in \mathcal{U}_n}{\sup} |\widehat{\psi}(u)-\psi(u)| \leq C \rho_n^{\alpha /(2\alpha +1)} \right\} \geq 1-\tfrac{1}{n^{\gamma -2}},\quad n \geq n_0,
$$
where $\rho_n=\ln (n) / n$ and $\mathcal{U}_n = \{u \in \mathcal{U}:\left[u \pm \delta_n\right] \subset \mathcal{U}\}$, with $\delta_n=\kappa  \rho_n^{1 /(2\alpha +1)}$. Specifically, $n_0$ is the smallest integer $n$ that satisfies $\epsilon_n<1$, where
$$
\epsilon_n\triangleq\max \left(c_n / \delta_n, \sqrt{2 c_n / \delta_n}\right)+\left(n \delta_n\right)^{-1},\qquad c_n\triangleq \gamma \ln (n+2) /(n+1).
$$
\end{theorem}

\begin{remark}
Theorem~\ref{thm:Dumbgen} is closely parallel to Theorem 3.3 in \cite{mosching2020monotone}. The main difference is that Theorem 3.3 in \cite{mosching2020monotone} is established for binary response variables, whereas our setting allows for continuous responses. In addition, while \cite{mosching2020monotone} provides a convergence rate, it does not make explicit the tail probability in the corresponding concentration inequality; by contrast, our result is stated directly in finite-sample form with an explicit high-probability bound. The proof is nonetheless strongly inspired by \cite{mosching2020monotone} and follows the same overall strategy. The proof is also close to Theorem 4.8 in \cite{bracale2025dynamic}, who work in the binary response setting, differently from our continuous case. 
\end{remark}

Theorem~\ref{thm:Dumbgen} requires two Lemmas: Lemma~\ref{lemma:Hoeffding} and Lemma~\ref{lemma:A_n}.

\begin{lemma}\label{lemma:Hoeffding}
Let \begin{align*}
\bar{\psi}_{rs} \triangleq
\frac{1}{o_{r s}} \sum_{j=r}^s o_j \psi(u_j), \qquad \text{and} \qquad M_n\triangleq\max _{1 \leq r \leq s \leq m} o_{r s}^{1 / 2}|\widehat{y}_{rs}-\bar{\psi}_{rs}|.
\end{align*}
Then for any constant $D>1$,
$$
 \mathbb{P}\left(M_n \leq2\sigma (D \ln n)^{1 / 2}\right)\leq 1-(\tfrac{n+1}{n^D})^2.
$$
\end{lemma}

\begin{proof}
First, note that 
$$
o_j \widehat{y}_j-o_j \psi(u_j)=\sum_{i:w_i=u_j}y_i-o_j \psi(u_j)=\sum_{i:w_i=u_j} (y_i-\psi(w_i))= \sum_{i:w_i=u_j} \varepsilon_i
$$
is subgaussian with variance proxy $o_j\sigma^2$, hence, using that $o_{rs}=\sum_{j=r}^s o_j$, we have that for any $\eta>0$

$$
\mathbb{P}\left[ \left|\sum_{j=r}^s o_j \widehat{y}_j- \sum_{j=r}^s o_j \psi(u_j) \right| \geq \eta\right] \leq 2\exp\left\{-\frac{\eta^2}{2o_{rs}\sigma^2}\right\},
$$
or, replacing $\eta \leftarrow \sqrt{o_{rs}} \eta$,
$$
\mathbb{P}\left[ \sqrt{o_{rs}}\left|\widehat{y}_{rs}- \bar{\psi}_{rs} \right| \geq \eta\right] \leq 2\exp\left\{-\frac{\eta^2}{2\sigma^2}\right\}.
$$

Note that $M_n$ is the maximum of the $\binom{m+1}{2}$ quantities
$$
  o_{r s}^{1 / 2}|\widehat{y}_{rs}-\bar{\psi}_{rs} |.
$$
Consequently,
$$
\begin{aligned}
\mathbb{P}\left(M_n \geq \eta_n\right) & \leq \sum_{1 \leq r \leq s \leq m} \mathbb{P}\left(o_{r s}^{1 / 2}|\widehat{y}_{rs}-\bar{\psi}_{rs} | \geq \eta_n\right) \\
  & \leq 2\binom{m+1}{2} \exp \left(- \eta_n^2/2\sigma^2\right) \\
  &= m(m+1) \exp \left(- \eta_n^2/2\sigma^2\right) \\
  & \leq \exp \left(2 \ln (n+1)- \eta_n^2/2\sigma^2\right)\\
  & \leq \exp \left(2 \ln ((n+1)/n^D)\right) = (\tfrac{n+1}{n^D})^2,
\end{aligned}
$$
where we chose $\eta_n=2\sigma( D \ln n)^{1 / 2}$ for some $D>1$.
\end{proof}

Before proceeding with the technical Lemma~\ref{lemma:A_n}, let's define
$$
\rho_n\triangleq \frac{\ln n}{n},
$$ 
and $\lambda(\cdot)$ the Lebesgue measure, and denote by $P_n(\cdot)$ the empirical measure of the design points $w_{t}$, that means
$$
P_n(B)\triangleq\frac{1}{n}\#\left\{t \in \mathcal{T}: w_{t}\in B\right\} \quad \text { for } B \subset \mathcal{U}.
$$

\begin{lemma}\label{lemma:A_n}
Let $w_1,w_2,\dots,w_n$ i.i.d. points with density $f_w$ that satisfies $\inf_{u \in \mathcal{U}} f_w (u) \geq C_2$ for some universal constant $C_2>0$, then for a given constant $\kappa >0$, and for any $\gamma >2$, there exists $n_0=n_0(\gamma,\kappa,\alpha) \in \mathbb{N}$ and a sequence $\epsilon_n = \epsilon_n(\gamma,\kappa,\alpha)>0$, $\epsilon_n \rightarrow 0$ such that
$$
\mathbb{P}\left(A_{n,\gamma}\right) > 1-\frac{1}{2(n+2)^{\gamma -2}}, \quad n \geq n_0
$$
where $A_{n,\gamma}$ is the event
$$
  \inf\left\{\frac{P_n\left(\mathcal{U}_n\right)}{\lambda\left(\mathcal{U}_n\right)}:\mathcal{U}_n \subset \mathcal{U}, \lambda\left(\mathcal{U}_n\right) \geq \delta_n\triangleq \kappa  \rho_n^{1 /(2\alpha +1)} \right\}\geq C_2(1-\epsilon_n).
$$
\end{lemma}

\begin{proof}
This result appears in \cite{bracale2025dynamic}; we provide a proof here for completeness. The proof is a direct application of the more general result by \citet[Section 4.3]{mosching2020monotone} which can be stated as follows: let $\delta_n>0$ such that $\delta_n \rightarrow 0$ while $n\delta_n /\ln(n) \rightarrow \infty$ (as $n\rightarrow \infty$). Then for every $\gamma>2$, there exists $n_0 = n_0 (\gamma, \delta_n)$ and $\epsilon_n = \epsilon_n(\gamma,\delta_n)>0$, $\epsilon_n \rightarrow 0$ such that 

$$  \mathbb{P}\left(\inf\left\{\frac{P_n\left(\mathcal{U}_n\right)}{P\left(\mathcal{U}_n\right)}:\mathcal{U}_n \subset \mathcal{U}, P\left(\mathcal{U}_n\right) \geq \delta_n \right\}\geq 1-\epsilon_n\right) > 1-\frac{1}{2(n+2)^{\gamma -2}}, \quad n \geq n_0,
$$
where $P(\cdot)$ is the probability measure of the design points $w_t$, that is
$$
P(B)\triangleq\int_B f_w(w)dw, \quad \text { for } B \subset \mathcal{U},
$$
and
$$
\epsilon_n\triangleq\max \left(c_n / \delta_n, \sqrt{2 c_n / \delta_n}\right)+\left(n \delta_n\right)^{-1} \rightarrow 0,
$$
where $c_n\triangleq \gamma \ln (n+2) /(n+1)$. The value $n_0$ is the smallest integer $n$ that satisfies $\epsilon_n<1$.
\end{proof}

Now we prove Theorem~\ref{thm:Dumbgen}.

\begin{proof}[Proof of Theorem~\ref{thm:Dumbgen}]
Let $n$ be sufficiently large so that $\mathcal{U}_n \neq \emptyset$ and such that the event $A_{n,\gamma}$ in Lemma~\ref{lemma:A_n} occurs. Since $f_w$ is the uniform distribution, the value $C_2$ defined in Lemma~\ref{lemma:A_n} corresponds to $1/|\mathcal{U}|$. For $u \in \mathcal{U}_n$ the indices
\begin{align*}
  r(u) & \triangleq\min \left\{j \in\{1, \ldots, m\}: u_j \geq u\right\}, \\
  j(u) & \triangleq\max \left\{j \in\{1, \ldots, m\}: u_j \leq u+\delta\right\},
\end{align*}
are well-defined, because $\left[u, u+\delta_n\right]$ is a subinterval of $I$ of length $\delta_n$. Note that by Lemma~\ref{lemma:A_n} this interval contains at least one observation $u_j$. Moreover,
\begin{align*}
& r(u) \leq j(u),\\
& u \leq u_{r(u)} \leq u_{j(u)}\leq u+\delta_n, \\
& o_{r(u) j(u)}=o_n\left(\left[u, u+\delta_n\right]\right) \geq C_2(1-\epsilon_n) n \delta_n,
\end{align*}
where $\epsilon_n$ is defined as in Lemma~\ref{lemma:A_n}. Consequently, with $M_n$ as in Lemma~\ref{lemma:Hoeffding}, we have 
\begin{align*}
\widehat{\psi}(u)-\psi(u)& \leq \widehat{\psi}(u_{r(u)})-\psi(u) \\
& =\min _{s\geq r(u)} \max _{r \leq r(u)} \widehat{y}_{r s}-\psi(u) \\
& \leq \max _{r \leq r(u)} \widehat{y}_{r j(u)}-\psi(u) \\
& \leq o_{r(u) j(u)}^{-1 / 2} M_n+\max _{r \leq r(u)} \bar{\psi}_{rj(u)}-\psi(u) \\
& \leq\left(C_2(1-\epsilon_n) n \delta_n\right)^{-1 / 2} M_n+\psi(u_{j(u)})-\psi(u) \\
& \leq\left(C_2(1-\epsilon_n) n \delta_n\right)^{-1 / 2} M_n+C_1 \delta_n^{\alpha}.
\end{align*}

In the first step, we used isotonicity of $u \mapsto \widehat{\psi}(u)$, and in the second last step we used isotonicity of $u \mapsto \psi(u)$, and the last step utilizes the H\"olderianity assumption. But on the event $\left\{M_n \leq 2\sigma (D \ln n)^{1 / 2}\right\}$, the previous considerations implies that
\begin{align*}
\underset{u \in \mathcal{U}_n}{\sup}(\widehat{\psi}(u)-\psi(u)) & \leq2\sigma\left(C_2(1-\epsilon_n) n \delta_n\right)^{-1 / 2}(D \ln n)^{1 / 2}+C_1 \delta_n^{\alpha}=C \rho_n^{\alpha /(2\alpha +1)},
\end{align*}
where $C=2\sigma \sqrt{\kappa D/C_2} + C_1\kappa^{\alpha}$, and $D$ is any real value strictly greater than $1$. But $\underset{u \in \mathcal{U}_n}{\sup}(\psi(u)-\widehat{\psi}(u))\leq C \rho_n^{\alpha /(2\alpha +1)}$ happens in $A_{n,\gamma} \cap \{M_n \leq 2 \sigma (D \ln n)^{1 / 2}\}$ which has probability
\begin{align*}
\mathbb{P}(A_{n,\gamma} \cap \{M_n \leq 2 \sigma(D \ln n)^{1 / 2}\}) &= 1-\mathbb{P}(A_{n,\gamma}^c \cup \{M_n \geq 2 \sigma(D \ln n)^{1 / 2}\}) \\
&\geq 1-\mathbb{P}(A_{n,\gamma}^c)-\mathbb{P}( M_n \geq 2 \sigma(D \ln n)^{1 / 2})\\
&= \mathbb{P}(A_{n,\gamma})+\mathbb{P}( M_n \leq 2 \sigma (D \ln n)^{1 / 2})-1\\
&\geq 1-\frac{1}{2(n+2)^{\gamma -2}}-\left(\frac{n+1}{n^D}\right)^2\\
&\geq 1-\frac{1}{(n+2)^{\gamma -2}} \geq 1-\frac{1}{n^{\gamma -2}},
\end{align*}
where we used that by Lemma~\ref{lemma:Hoeffding}, for any fixed $D>1$ we have $\mathbb{P}\left(M_n \leq 2 \sigma (D \ln n)^{1 / 2}\right) \geq 1-(\tfrac{n+1}{n^D})^2$ and by Lemma~\ref{lemma:A_n} for any $\gamma >2$ we have $\mathbb{P}\left(A_{n,\gamma}\right) > 1-\frac{1}{2(n+2)^{\gamma -2}}$. The last two inequalities come from choosing $\gamma=D$ sufficiently large. 

Analogously one can show that on $\left\{M_n \leq 2 \sigma(D \ln n)^{1 / 2}\right\}$,
\begin{align*}
\underset{u \in \mathcal{U}_n}{\sup}(\psi(u)-\widehat{\psi}(u))&\leq 2 \sigma\left(n \delta_n\right)^{-1 / 2}(D \ln n)^{1 / 2}+C_1 \delta_n^{\alpha}=C \rho_n^{\alpha /(2\alpha +1)},
\end{align*}
with the same constant $C$ and with the same probability tail.
\end{proof}

\subsection{Proof of Lemma~\ref{lem:J2J3}}\label{proof:lemma-convergence}
Lemma~\ref{lem:J2J3} follows directly from the general convergence result in
Theorem~\ref{thm:Dumbgen}. The correspondence between the notation in
Theorem~\ref{thm:Dumbgen} and the present setting is summarized in
Table~\ref{tab:notation_replacement_J2J3}.

\begin{table}[ht]
\centering
\caption{Notation correspondence used to apply Theorem~\ref{thm:Dumbgen} to Lemma~\ref{lem:J2J3}.}
\label{tab:notation_replacement_J2J3}
\begin{tabular}{c c}
\toprule
\textbf{Notation in Theorem~\ref{thm:Dumbgen}} 
& \textbf{Notation in Lemma~\ref{lem:J2J3}} \\
\midrule
$n$ & $T_0$ \\
$\psi$ & $\psi_i$ \\
$\mathcal U$ & $\mathcal X_i$ \\
$y_t$ & $y_i^t$ \\
$w_t$ & $x_i^t$ \\
$\sigma^2$ & $\sigma_i^2$ \\
$C_1$ & $L_i$ \\
\bottomrule
\end{tabular}
\end{table}
We only need to prove that the densities of the equilibrium coordinates $x_i^{t}$ of the vector $\mathbf{x}^t \in \mathcal{X}$ (which are samples as in \eqref{eq:coordinate_regression_isotonic}), are bounded away from zero on the corresponding domain $\mathcal{X}_i$. 

From \eqref{eq:coordinate_regression_isotonic} we know that $\mathbf{x}^t$ is such that 
$$
\mathbf{F}(\mathbf x^t)
=
-\mathbf B\mathbf p^t+\boldsymbol{\xi}^t, \quad \Leftrightarrow \quad y_i^t
=
\psi_i(x_i^t)+\xi_i^t, \quad \forall i\in\mathcal N,
$$
where, by Assumption~\ref{ass:strong-concavity-estim}, the map $\mathbf F(\mathbf x):=\boldsymbol{\psi}(\mathbf x)-\delta \mathbf G\mathbf x$ is $m$-strongly monotone and implies the variational condition. Recall moreover that $\mathbf{p}^t$, during the exploration phase, are sampled i.i.d from $\operatorname{Unif}(\mathcal{P})$, and the errors $\boldsymbol{\xi}^t$ are sampled from a $N$-dimensional distribution with independent coordinates, and each coordinate is a (zero mean) subgaussian with variance proxies $\sigma_i^2$ for all $i \in \mathcal{N}$. Let 
$$
\mathbf{z}^t = -\mathbf{B} \mathbf{p}^t + \boldsymbol{\xi}^t,
$$
where $\mathbf{p}^t$ and $\boldsymbol{\xi}^t$ are clearly independent. By strong monotonicity of $\mathbf{F}$, we can write 
$$
\mathbf{x}^t = \mathbf{F}^{-1}(\mathbf{z}^t), \qquad \text{or} \qquad \mathbf{X} = \mathbf{F}^{-1}(\mathbf{Z}),
$$
where $\mathbf{X}\sim \mathbf{x}^t$ and $\mathbf{Z}\sim \mathbf{z}^t$. First, $\mathbf{Z}$ has an everywhere positive density. Indeed, assuming $\mathbf{p}$ has density $f_{\mathbf{P}}$,

$$
f_{\mathbf{Z}}(\mathbf{z})=\int_{\mathcal{P}} f_{\boldsymbol{\xi}}(\mathbf{z}+\mathbf{B} \mathbf{p}) f_{\mathbf{P}}(\mathbf{p}) d \mathbf{p} .
$$
Since $f_{\boldsymbol{\xi}}>0$ on $\mathbb{R}^N$, $f_{\boldsymbol{\xi}}(\mathbf{z}+\mathbf{B} \mathbf{p})>0$ for every $\mathbf{z} \in \mathbb{R}^N$ and every $\mathbf{p} \in \mathcal{P}$. Hence, we get
$$
f_{\mathbf{Z}}(\mathbf{z})>0 \quad \forall \mathbf{z} \in \mathbb{R}^N .
$$
Now, since
$$
\mathbf{F}: \mathbb{R}^N \rightarrow \mathbb{R}^N
$$
is continuously differentiable and $m$-strongly monotone, then it is a bijection, so $\mathbf{F}^{-1}$ is well-defined on all of $\mathbb{R}^N$. Moreover, $\nabla F(x)$ is nonsingular everywhere. In fact, strong monotonicity implies
$$
\mathbf{v}^{\top} \nabla \mathbf{F}(\mathbf{x}) \mathbf{v} \geq m\|\mathbf{v}\|^2, \qquad \forall \mathbf{v} \neq \mathbf{0},
$$
so
$\|\nabla \mathbf{F}(\mathbf{x}) \mathbf{v}\| \geq m\|\mathbf{v}\|$ and therefore
$$
|\operatorname{det} \nabla \mathbf{F}(\mathbf{x})| \geq m^N>0 .
$$
By the multivariate change-of-variables formula,
$$
f_{\mathbf{X}}(\mathbf{x})=f_{\mathbf{Z}}(\mathbf{F}(\mathbf{x}))|\operatorname{det} \nabla \mathbf{F}(\mathbf{x})| .
$$
Since $f_{\mathbf{Z}}(\mathbf{F}(\mathbf{x}))>0$ and $|\operatorname{det} \nabla \mathbf{F}(\mathbf{x})|>0$, we get
$$
f_{\mathbf{X}}(\mathbf{x})>0 \quad \forall \mathbf{x} \in \mathbb{R}^N. 
$$
 More explicitly, $f_{\mathbf{X}}(\mathbf{x}) \geq m^D f_{\mathbf{Z}}(\mathbf{F}(\mathbf{x}))>0$. Hence in $\mathcal{X}$, which is compact, $f_{\mathbf{X}}$ is bounded away from zero. Consequently, each marginal density $f_{X_i}$ is also bounded away from zero in $\mathcal{X}_i$. This completes the proof.

\newpage

\section{Computational Complexity for Algorithm~\ref{alg:isotonic_plugin_pricing}}\label{app:comp_comp}

We briefly discuss the computational cost of Algorithm~\ref{alg:isotonic_plugin_pricing}. Let $T_0$ denote the exploration horizon and $N$ the number of consumers.

\textbf{Step 1: Construction of the responses $y_i^t$.}
At each exploration round $t$, the seller computes
\[
y_i^t:=-b_i p_i^t+\delta(\mathbf G\mathbf x^t)_i,
\qquad i\in\mathcal N.
\]
This requires one matrix-vector multiplication $\mathbf G\mathbf x^t$. Hence the cost per round is
\[
O(\mathrm{nnz}(\mathbf G)),
\]
where $\mathrm{nnz}(\mathbf G)$ is the number of nonzero entries of $\mathbf G$. In particular, this is $O(N^2)$ for a dense network and can be much smaller when the network is sparse. Over the entire exploration phase, the total cost of building the regression responses is therefore
\[
O \bigl(T_0\,\mathrm{nnz}(\mathbf G)\bigr).
\]

\textbf{Step 2: Isotonic estimation of the coordinate-wise nonlinearities.}
For each $i\in\mathcal N$, the sample $\{(x_i^t,y_i^t)\}_{t=1}^{T_0}$ is first reordered according to the covariates $x_i^t$, and then the isotonic least-squares estimator is computed by the pool-adjacent-violators algorithm (PAVA). Sorting costs $O(T_0\ln T_0)$, while PAVA itself is linear in the sample size, i.e. $O(T_0)$ \citep{grotzinger1984projections, tibshirani2011nearly}. Hence the cost per coordinate is
\[
O(T_0\ln T_0),
\]
and the total cost over all $N$ coordinates is
\[
O(NT_0\ln T_0).
\]
If the covariates are already stored in sorted order, this reduces to $O(NT_0)$.

\textbf{Step 3: Construction of the plug-in price map in the consumption space.}
Rather than solving the nonlinear equilibrium equation
\[
\widehat{\boldsymbol{\psi}}(\mathbf x)-\delta \mathbf G\mathbf x=-\mathbf B\mathbf p
\]
for each candidate price vector $\mathbf p$, it is computationally more convenient to optimize directly in the consumption space. To this end, define the plug-in price map
\[
\widehat{\mathbf Q}(\mathbf z)
:=
-\mathbf B^{-1}\bigl(\widehat{\boldsymbol{\psi}}(\mathbf z)-\delta \mathbf G\mathbf z\bigr),
\qquad \mathbf z\in\mathcal X.
\]
Evaluating $\widehat{\mathbf Q}(\mathbf z)$ requires:
\begin{itemize}[ leftmargin=1.2em]
    \item one evaluation of $\widehat{\boldsymbol{\psi}}(\mathbf z)$, which is $O(N)$ since the estimator is coordinate-wise;
    \item one matrix-vector multiplication $\mathbf G\mathbf z$, which costs
    \[
    O(\mathrm{nnz}(\mathbf G)).
    \]
\end{itemize}
Hence one evaluation of the plug-in price map $\widehat{\mathbf Q}(\mathbf z)$ costs
\[
O \bigl(N+\mathrm{nnz}(\mathbf G)\bigr).
\]
In particular, for dense networks this is $O(N^2)$, whereas for sparse networks it is nearly linear in the number of edges.

\textbf{Step 4: Optimization in the consumption space.}
Using the plug-in price map $\widehat{\mathbf Q}$, the seller's optimization problem can be rewritten as
\[
\widehat{\mathbf z}\in\arg\max_{\mathbf z\in\mathcal Z_{\mathcal P}}
\widehat{\widetilde r}(\mathbf z),
\qquad
\widehat{\widetilde r}(\mathbf z):=\mathbf z^\top \widehat{\mathbf Q}(\mathbf z),
\]
where
\[
\mathcal Z_{\mathcal P}:=\{\mathbf z\in\mathcal X:\widehat{\mathbf Q}(\mathbf z)\in\mathcal P\}.
\]
Once $\widehat{\mathbf z}$ is obtained, the corresponding plug-in price is
\[
\widehat{\mathbf p}:=\widehat{\mathbf Q}(\widehat{\mathbf z}).
\]

The complexity of this step depends on the outer optimization routine. Typical examples are the following.

\begin{enumerate}[ leftmargin=1.2em]
    \item \textbf{Projected first-order methods.}
    If one uses a projected optimization scheme in the consumption space, then each iteration requires:
    \begin{itemize}[ leftmargin=1.2em]
        \item one evaluation of $\widehat{\mathbf Q}(\mathbf z)$, costing $O(N+\mathrm{nnz}(\mathbf G))$;
        \item one projection onto the feasible set $\mathcal Z_{\mathcal P}$, whose cost we denote by
        \[
        C_{\mathrm{proj}}.
        \]
    \end{itemize}
    Hence, if $K_{\mathrm{opt}}$ projected iterations are performed, the total cost is
    \[
    O \bigl(K_{\mathrm{opt}}(N+\mathrm{nnz}(\mathbf G)+C_{\mathrm{proj}})\bigr).
    \]
    In particular, $C_{\mathrm{proj}}$ depends on the geometry of $\mathcal Z_{\mathcal P}$. If the feasible set has a simple form, such as a box or an Euclidean ball, then $C_{\mathrm{proj}}=O(N)$; for more general nonlinear feasible sets, the projection step may itself require solving an auxiliary optimization problem.

    \item \textbf{Derivative-free optimization.}
    If one prefers not to use a projected first-order method, then $\widehat{\widetilde r}$ may be optimized with derivative-free routines such as coordinate search, pattern search, or Nelder--Mead. If $K_{\mathrm{obj}}$ objective evaluations are required, then the total cost is
    \[
    O \bigl(K_{\mathrm{obj}}(N+\mathrm{nnz}(\mathbf G))\bigr),
    \]
    up to the cost of checking the feasibility condition $\widehat{\mathbf Q}(\mathbf z)\in\mathcal P$.

    \item \textbf{Grid search / exhaustive search.}
    If $\mathcal X$ is low-dimensional or discretized, one may evaluate $\widehat{\widetilde r}$ on a grid. If $M$ grid points are used, then the total cost is
    \[
    O \bigl(M(N+\mathrm{nnz}(\mathbf G))\bigr).
    \]
    This is only practical in very low dimension, since $M$ grows exponentially with $N$ for a full Cartesian grid.
\end{enumerate}

\textbf{Overall complexity.}
Putting together all the steps, the total complexity of the procedure becomes
\[
\underbrace{O \bigl(T_0\,\mathrm{nnz}(\mathbf G)\bigr)}_{\text{Step 1: construction of the responses } y_i^t}
\;+\;
\underbrace{O \bigl(NT_0\ln T_0\bigr)}_{\text{Step 2: isotonic estimation of } \widehat{\boldsymbol{\psi}}}
\;+\;
\underbrace{O \bigl(K_{\mathrm{opt}}(N+\mathrm{nnz}(\mathbf G)+C_{\mathrm{proj}})\bigr)}_{\text{Steps 3--4: evaluation of } \widehat{\mathbf Q}(\mathbf z)\text{ and outer optimization in } \mathbf z\text{-space}},
\]
where $K_{\mathrm{opt}}$ denotes the number of outer optimization iterations or objective evaluations, and $C_{\mathrm{proj}}$ is the cost of projection onto the feasible set $\mathcal Z_{\mathcal P}$ when a projected method is used.

\newpage

\section{Experiments}\label{app:experiments}\label{sec:simulation_setup}

We evaluate the finite-sample performance of Algorithm~\ref{alg:isotonic_plugin_pricing} on synthetic network markets. In line with Remark~\ref{rem:low-dim-N}, we focus on a low-dimensional regime with $N=20$ consumers and horizons
\[
T\in\{25,50,75,100,125\}.
\]
For each value of $T$, we run $10$ independent repetitions and report $95\%$ confidence intervals. For any posted price vector $\mathbf p$, the consumption profile is defined as the unique equilibrium of
\[
\mathbf x
=
\boldsymbol{\phi}^{-1}\!\bigl(\mathbf a-\mathbf B\mathbf p+\delta \mathbf G\mathbf x\bigr),
\]
and the seller's revenue is
\[
r(\mathbf p)=\mathbf p^\top \mathbf x(\mathbf p).
\]

We consider three utility specifications, all with a common function $h_i=h$ across consumers:
\begin{itemize}[leftmargin=1.5em,itemsep=0.2em,topsep=0.2em]
    \item \textbf{Linear--Quadratic:} $h(x)=\frac12 x^2$, so $\phi(x)=x$.
    \item \textbf{Power:} $h(x)=\frac{1}{\gamma+1}x^{\gamma+1}$ with $\gamma\in\{0.2,0.4,0.6,0.8\}$.
    \item \textbf{Discrete-choice / logit:} $h(x)=x\log x+(1-x)\log(1-x)$, so $\phi(x)=\log\!\bigl(\frac{x}{1-x}\bigr)$.
\end{itemize}

Model parameters are chosen as follows. In the linear--quadratic case, we set $\delta=0.5$, $\mathbf a$ linearly spaced in $[3.2,3.8]$, $\mathbf b=(b_1,b_2,\dots,b_N)$ linearly spaced in $[0.8,1.0]$, and $\mathcal X=[0,4]^N$. In the logistic case, we set $\delta=0.3$, $\mathbf a$ linearly spaced in $[1.0,2.0]$, $\mathbf b$ linearly spaced in $[1.5,2.0]$, and $\mathcal X=[0,1]^N$. In the power case, we set $\delta=0.15$, $\mathbf a$ linearly spaced in $[0.6,1.0]$, $\mathbf b$ linearly spaced in $[0.8,1.2]$, and $\mathcal X=[10^{-6},1]^N$. In all cases, the operative price domain $\mathcal P=[\mathbf p_{\mathrm{low}},\mathbf p_{\mathrm{high}}]$ is not fixed ex ante, but is constructed from the local theory-consistent box induced by the map
\[
Q(\mathbf x)=-\mathbf B^{-1}F(\mathbf x),
\qquad
F(\mathbf x)=\boldsymbol{\psi}(\mathbf x)-\delta \mathbf G\mathbf x,
\]
and clipped to ensure nonnegative prices.

The oracle benchmark is computed by optimizing directly over the interior consumption space. More precisely, we solve
\[
\mathbf x^\star\in\arg\max_{\mathbf x\in\mathring{\mathcal X}} Q(\mathbf x)^\top \mathbf x,
\qquad
\mathbf p^\star=Q(\mathbf x^\star).
\]
This guarantees consistency with the structural equilibrium mapping and avoids infeasible benchmark prices.

For each run, the exploration length is
\[
T_0=\lceil cT^\beta\rceil,
\qquad c=1,\quad \beta=0.75.
\]
During exploration, we sample $\mathbf x_t$ uniformly from the interior of $\mathcal X$, set $\mathbf p_t=Q(\mathbf x_t)$, and observe
\[
\widetilde{\mathbf x}_t=\Pi_{\mathcal X}\bigl(\mathbf x_t+\boldsymbol{\varepsilon}_t\bigr),
\]
where the noise is Gaussian with standard deviation $0.05$ in the linear--quadratic and logistic models, and $0.03$ in the power model. The regression targets are
\[
Y_{t,i} = a_i-b_ip_{t,i}+\delta(\mathbf G\widetilde{\mathbf x}_t)_i,
\]
which provide noisy observations of the common structural map.

We estimate a single monotone function $\widehat{\psi}$ by pooling all pairs $(\widetilde x_{t,i},Y_{t,i})$ across consumers and time, and fitting an isotonic least-squares estimator. The plug-in policy is then obtained by solving
\[
\widehat{\mathbf p}\in\arg\max_{\mathbf p\in\mathcal P}\mathbf p^\top \widehat{\mathbf x}(\mathbf p),
\]
where $\widehat{\mathbf x}(\mathbf p)$ is defined through the plug-in equilibrium equation
\[
\widehat{\psi}(x_i)\approx a_i-b_ip_i+\delta(\mathbf G\mathbf x)_i,
\qquad i=1,\dots,N.
\]
Both the oracle and the plug-in optimizer are computed numerically via derivative-free box-constrained optimization.

For each repetition, we record the cumulative regret
\[
R(T)=\sum_{t=1}^T \bigl(r(\mathbf p^\star)-r(\mathbf p_t)\bigr).
\]
We summarize the results by plotting mean regret and mean estimation errors against $T$ on a log--log scale, together with the corresponding empirical slopes.

\textbf{Network structures.} We run the algorithm under several choices of the interaction matrix $\mathbf G$, in order to test the finite-sample performance of the plug-in pricing rule under qualitatively different patterns of network dependence. In particular, we consider sparse local interactions, influencer-driven topologies, and two benchmark dense structures.

\begin{enumerate}[leftmargin=1.5em]
\item \textbf{Sparse circular $\mathbf G$.} We generate $\mathbf G$ as a sparse directed circular network. Specifically, for each $i$ we set
\[
g_{i,(i+1)\pmod N}=w,\qquad
g_{i,(i+2)\pmod N}=\frac{w}{2},\qquad
g_{ii}=0,
\]
with $w=0.08$, and all remaining entries equal to zero. Hence, each node is influenced by its first successor on the circle and, more weakly, by its second successor. To introduce signed interactions, we flip the sign of $10\%$ of the nonzero entries. This yields a sparse directed network with an underlying circular topology and a few negative spillovers. Figure~\ref{fig:tau_2} illustrates the convergence of the cumulative regret. In Figure~\ref{fig:different_exploration} we show the convergence of the isotonic regressions for $T \in \mathcal{T}$ in each utility model specified, demonstrating the fast convergence of this estimator.

\item \textbf{$\mathbf G$ with an influencer.}
We also consider a network with the presence of an influencer, indexed by $i=0$. This node is characterized by a null $i$-th row of $\mathbf G$ and large positive entries in the $i$-th column, so that it affects the rest of the network while remaining unaffected by others. Figure~\ref{fig:influencer_prices} displays the resulting optimal price distribution. Consistently with the theory in Section~\ref{sec:influencers}, the influencer receives a substantially lower optimal price than the rest of the network, in some cases close to zero.

\item \textbf{Null network $\mathbf G=\mathbf O_N$.}
As a baseline benchmark, we also consider the case of a null interaction matrix, $\mathbf G=\mathbf O_N$. In this setting, consumers do not exert any cross-effects on one another, so the equilibrium problem becomes fully decoupled across coordinates. This topology serves as a useful reference point, isolating the role of nonlinear utility curvature from the contribution of network spillovers. Figure~\ref{fig:regret-G-uniform} illustrates the convergence of the cumulative regret.

\item \textbf{Uniform network $\mathbf G=\mathbf 1\mathbf 1^\top-\mathbf I_N$.}
Finally, we consider the dense benchmark
\[
\mathbf G=\mathbf 1\mathbf 1^\top-\mathbf I_N,
\]
where every consumer interacts equally with every other consumer. In this case, all off-diagonal entries are equal to one and diagonal entries are set to zero. This topology provides the opposite extreme of the sparse circular design, allowing us to assess the algorithm under pervasive and homogeneous interaction effects. Figure~\ref{fig:regret-G-uniform} illustrates the convergence of the cumulative regret.
\end{enumerate}

\begin{figure}
    \centering
\includegraphics[width=\linewidth]{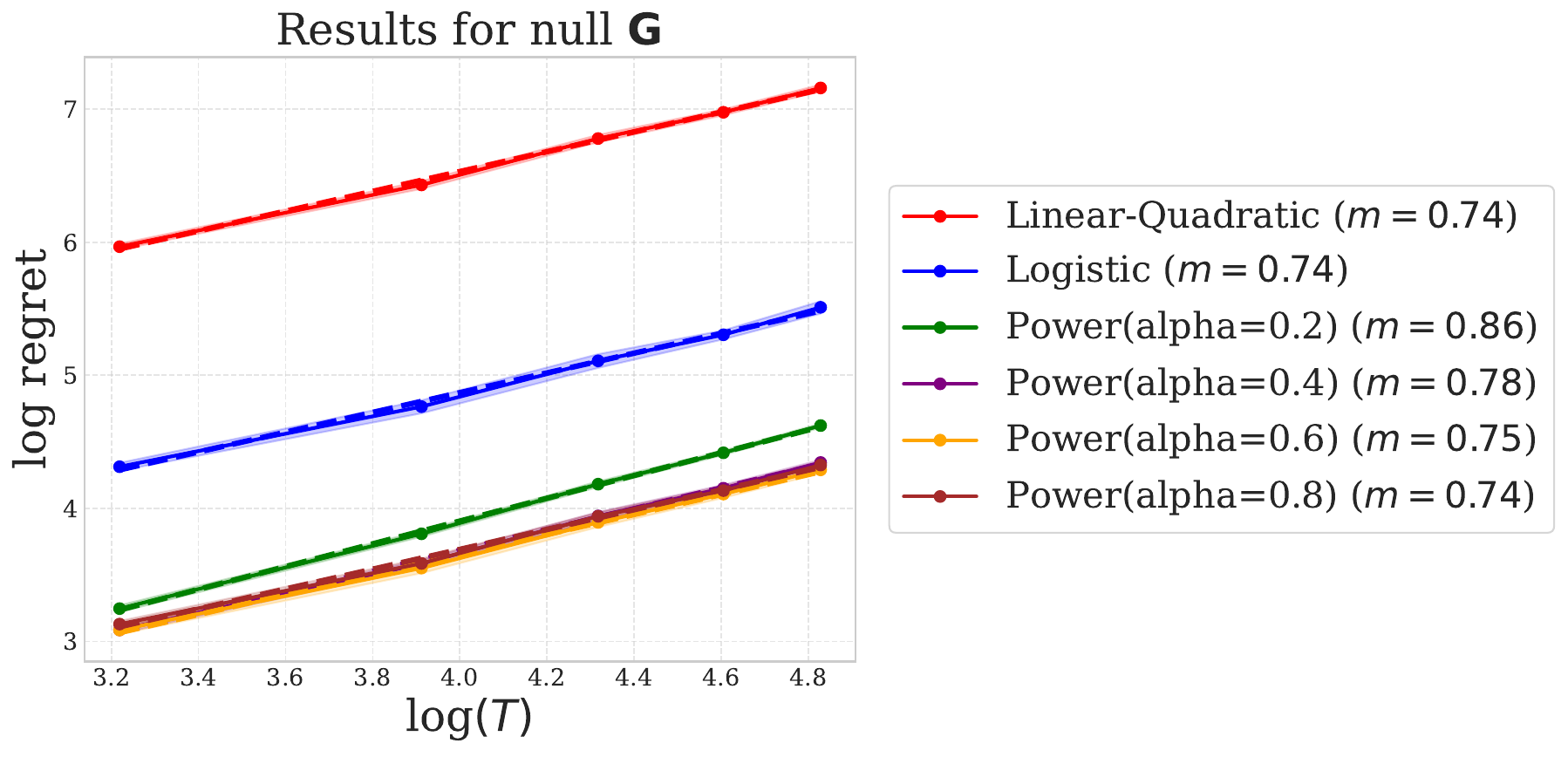}
\includegraphics[width=\linewidth]{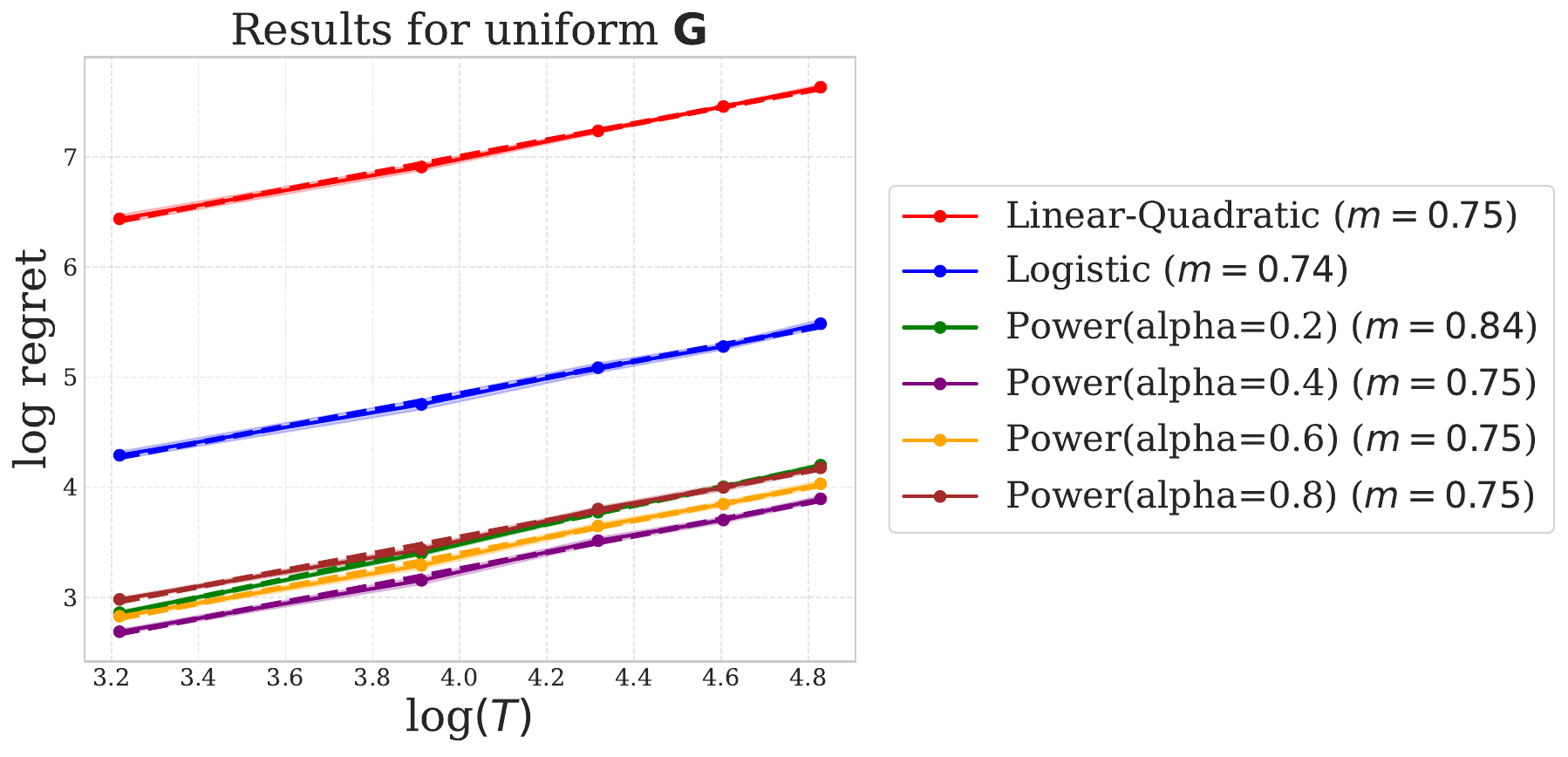}
    \caption{This plot show the regret convergence for $\mathbf{G}$ null (top) and $\mathbf{G}$ uniform (bottom). Average across the $N$ sellers of the cumulative regret as a function of the horizon $T\in\mathcal T$, displayed on a log--log scale. The empirical slopes $m$'s are consistent with the theoretical regret rates predicted by our Theorem~\ref{thm:dynamic_regret_isotonic_network} (same or lower rate), which are respectively (following the legend from top to bottom): $0.75$, NA, $0.87$, $0.81$, $0.78$, and $0.76$.}
    \label{fig:regret-G-uniform}
\end{figure}

\begin{figure}[ht]
    \centering
    \includegraphics[width=1\textwidth]{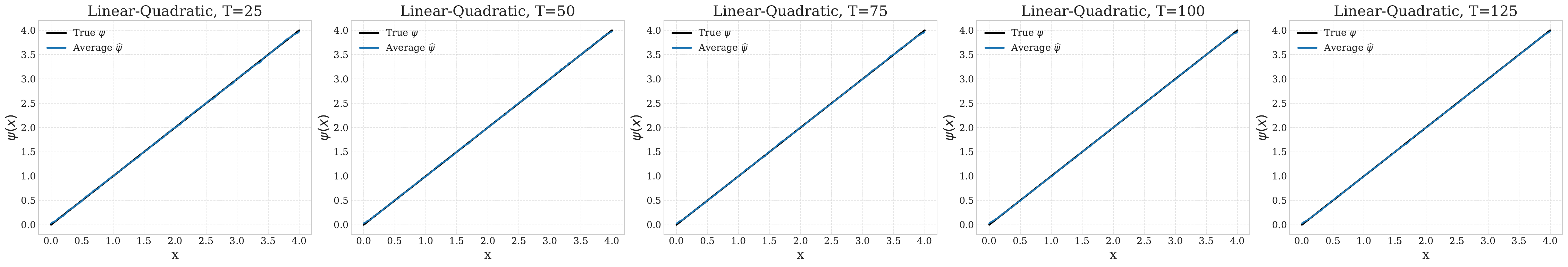}
    \includegraphics[width=1\textwidth]{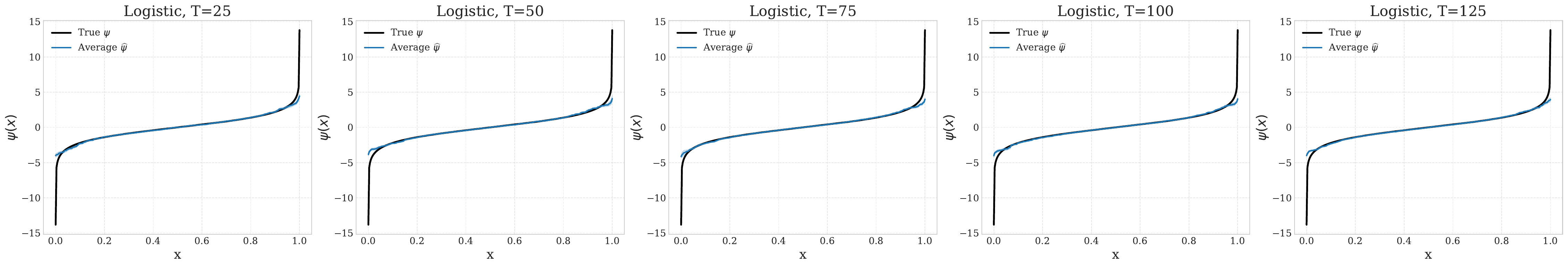}
    \includegraphics[width=1\textwidth]{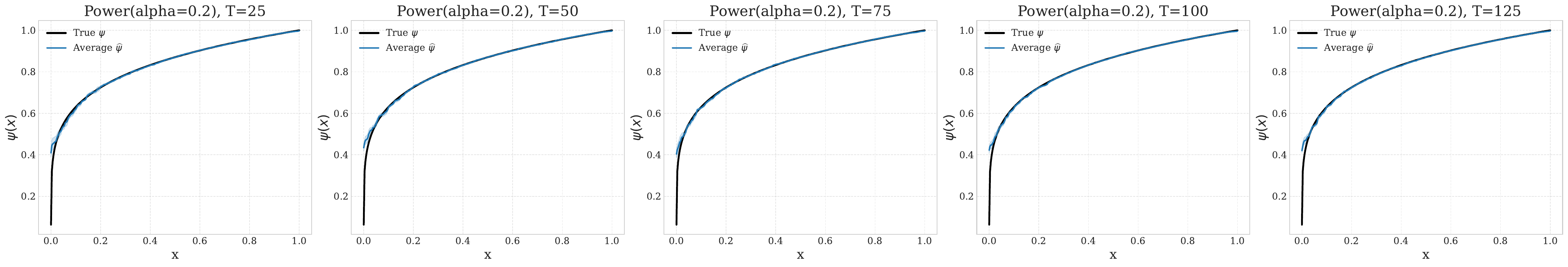}
    \includegraphics[width=1\textwidth]{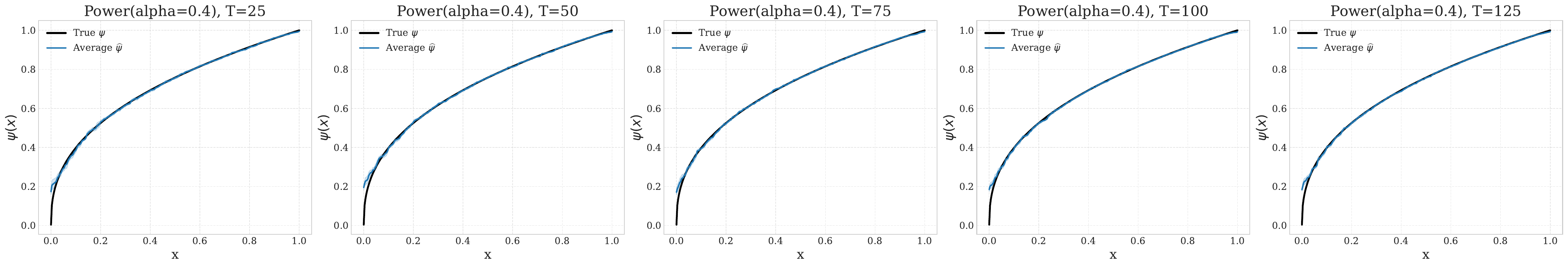}
    \includegraphics[width=1\textwidth]{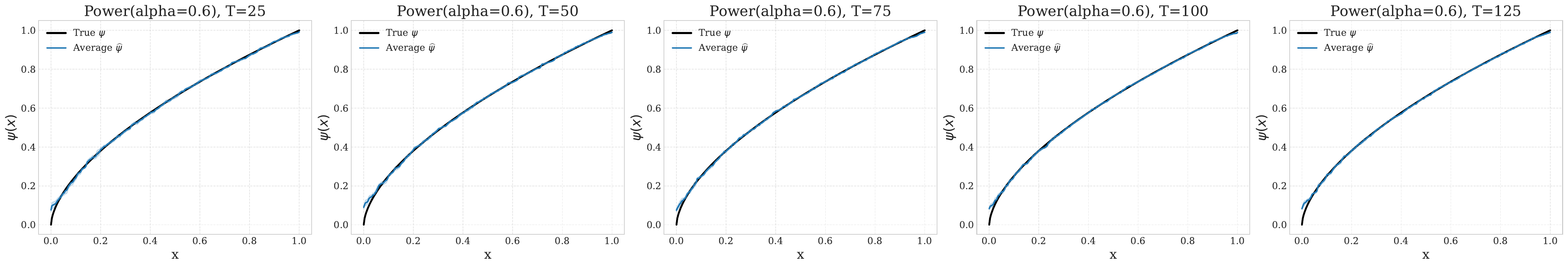}
    \includegraphics[width=1\textwidth]{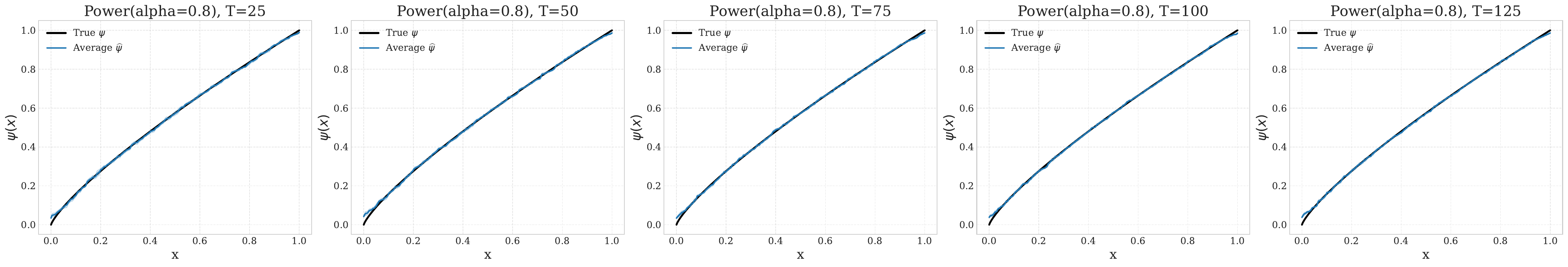}
    \caption{Isotonic regression estimate for varying $T$ in each specified model.}
    \label{fig:different_exploration}
\end{figure}

\clearpage
\newpage

\end{document}